\documentclass[a4paper, 12pt]{amsart}

\usepackage{amsmath}
\usepackage{amssymb}
\usepackage{amsthm}
\usepackage[mathscr]{eucal}

\newcommand{\ANR}{\operatorname{ANR}}
\newcommand{\Ant}{\operatorname{Ant}}
\newcommand{\BV}{\operatorname{BV}}

\newcommand{\CAT}{\operatorname{CAT}}
\newcommand{\CBA}{\operatorname{CBA}}
\newcommand{\CBB}{\operatorname{CBB}}
\newcommand{\closure}{\operatorname{cl}}
\newcommand{\DC}{\operatorname{DC}}
\newcommand{\diam}{\operatorname{diam}}
\newcommand{\id}{\operatorname{id}}

\newcommand{\GCBA}{\operatorname{GCBA}}
\newcommand{\GH}{\operatorname{\mathrm{GH}}}
\newcommand{\LGC}{\operatorname{LGC}}
\newcommand{\ulim}{\operatorname{\lim_{\omega}}}
\newcommand{\Class}{\operatorname{\mathcal{C}}}

\newcommand{\R}{\operatorname{\mathbb{R}}}
\newcommand{\Z}{\operatorname{\mathbb{Z}}}
\newcommand{\N}{\operatorname{\mathbb{N}}}
\newcommand{\Sph}{\operatorname{\mathbb{S}}}

\newcommand{\Haus}{\operatorname{\mathcal{H}}}
\numberwithin{equation}{section}

\theoremstyle{plain}
\newtheorem{thm}{Theorem}[section]
\newtheorem{lem}[thm]{Lemma}  
\newtheorem{prop}[thm]{Proposition}
\newtheorem{cor}[thm]{Corollary}

\newtheorem{slem}[thm]{Sublemma}

\theoremstyle{definition}
\newtheorem{defn}{Definition}[section]
\newtheorem{exmp}{Example}[section]
\newtheorem{prob}{Problem}[section]
\newtheorem{assumption}{Assumption}[section]

\theoremstyle{remark}
\newtheorem{rem}{Remark}[section]

\begin{document} 

\title
[Wall singularity of spaces with an upper curvature bound ]
{Wall singularity of spaces with \\ an upper curvature bound}

\author
[Koichi Nagano]{Koichi Nagano}

\thanks{
Supported
by JSPS KAKENHI Grant Numbers 18740023, 20K03603, 25K06996}

\address
[Koichi Nagano]
{\endgraf 
Department of Mathematics, University of Tsukuba
\endgraf
Tennodai 1-1-1, Tsukuba, Ibaraki, 305-8571, Japan}

\email{nagano@math.tsukuba.ac.jp}

\date{January 30, 2026}


\begin{abstract}
We study typical wall singularity of codimension one
for locally compact geodesically complete metric spaces 
with an upper curvature bound.
We provide a geometric structure theorem of codimension one singularity,
and a geometric characterization of codimension two regularity.
These give us necessary and sufficient conditions for singular sets to be of codimension at least two.
\end{abstract}

\maketitle

\tableofcontents

\section{Introduction}
\label{sec: i}

In Alexandrov geometry,
$\CBA$ spaces with curvature bounded above
and $\CBB$ spaces with curvature bounded below
are similarly defined 
for metric spaces
by the triangle comparison conditions;
however,
their properties are quite different from each other.
Generally speaking,
$\CBA$ spaces
are much more singular than $\CBB$ spaces.
For instance,
$\CBA$ spaces,
unlike $\CBB$ spaces,
may admit singular sets of codimension one 
in their interiors.
We focus on typical wall singularity of codimension one
for locally compact geodesically complete $\CBA$ spaces.

\subsection{Motivations}

A metric space is said to be $\CBA(\kappa)$
for $\kappa \in \R$
if it is locally $\CAT(\kappa)$;
namely,
every point has a $\CAT(\kappa)$ neighborhood.
A metric space is $\CBA$ 
if it is $\CBA(\kappa)$ for some $\kappa$.
A $\CBA(\kappa)$ space is $\GCBA(\kappa)$ 
if it is locally compact, separable,
and locally geodesically complete.
A metric space is $\GCBA$ 
if it is $\GCBA(\kappa)$ for some $\kappa$.
Lytchak and the author \cite{lytchak-nagano1, lytchak-nagano2}
studied fundamental geometric properties of $\GCBA$ spaces
in \cite{lytchak-nagano1},
and topological regularity of $\GCBA$ spaces
in \cite{lytchak-nagano2}.
This paper is devoted to a subsequent study
of them.

Throughout this paper,
we denote by $\dim$ the topological dimension,
and by $\dim_{\mathrm{H}}$ the Hausdorff dimension.
Let $X$ be a $\GCBA$ space.
Then $\dim X$ is equal to $\dim_{\mathrm{H}} X$,
and equal to the supremum of $m \in \N$ such that
$X$ admits an open subset homeomorphic to 
the Euclidean $m$-space $\R^m$
(\cite[Theorem 1.1]{lytchak-nagano1});
moreover,
every relatively compact open subset of $X$ has 
finite topological dimension (\cite[Subsection 5.3]{lytchak-nagano1}).

For a point $x$ in $X$,
we denote by $\Sigma_xX$
the space of directions at $x$ in $X$,
and by $T_xX$ 
the tangent space at $x$ in $X$.
For every $x \in X$,
the space of directions $\Sigma_xX$ is 
a compact, geodesically complete $\CAT(1)$ space,
and the tangent space $T_xX$ is 
a proper, geodesically complete $\CAT(0)$ space;
in particular,
both $\Sigma_xX$ and $T_xX$ are $\GCBA$ spaces.

For $m \in \N$,
we denote by $X^m$ the
\emph{$m$-dimensional part of $X$}
defined as the set of all points $x$ in $X$ with
$\dim T_xX = m$.
A point $x$ in $X$ belongs to $X^m$
if and only if
every sufficiently small neighborhood of $x$
has topological dimension $m$
(\cite[Theorem 1.2]{lytchak-nagano1}).

We say that a point $x$ in $X$ is 
\emph{$m$-regular} 
if the tangent space $T_xX$ is isometric to
an $\ell^2$-product metric space $\R^m \times Y$
for some $Y$.
A point in $X$ is 
\emph{$m$-singular} 
if it is not $m$-regular.
For a subset $A$ of $X$,
we denote by $R_m(A)$
the set of all $m$-regular points in $A$,
and by $S_m(A)$
the set of all $m$-singular points in $A$.
We call $R_m(A)$ the 
\emph{$m$-regular set in $A$},
and $S_m(A)$ the 
\emph{$m$-singular set in $A$}.
We notice that
if $\dim X = n$,
then $S_{n+1}(X)$ coincides with $X$
(see Lemma \ref{lem: ndr}).

A point in a topological space is said to be an 
\emph{$m$-manifold point}
if it has a neighborhood homeomorphic to $\R^m$.
A point is a
\emph{non-manifold point}
if it is not an $m$-manifold point for any $m$.
For an open subset $U$ of a topological space,
we denote by $S(U)$ the
\emph{topologically singular set in $U$}
defined as the set of all non-manifold points in $U$.

As a geometric structure theorem for $\GCBA$ spaces,
Lytchak and the author 
\cite[Theorems 1.2 and 1.3]{lytchak-nagano1}
proved that 
if $X^m$ is non-empty,
then there exists an open subset $M^m$ of $X$
satisfying the following properties:
(1) $M^m$ is a Lipschitz $m$-manifold, and open dense in $X^m$;
(2) $\dim_{\mathrm{H}} \left( \closure_X(X^m) - M^m \right) \le m-1$,
where $\closure_X$ is the closure of $X$;
(3) $\dim_{\mathrm{H}} S_m(M^m) \le m-2$;
(4) $M^m$ admits a $\DC$-structure,
and a $\BV_{\mathrm{loc}}$-Riemannian metric 
continuous on $R_m(M^m)$.
Lytchak and the author 
\cite[Theorem 1.6]{lytchak-nagano1} 
also proved that
for each $m \in \N$ the $m$-singular set 
$S_m(X)$ is countably $(m-1)$-rectifiable;
in particular,
if $\dim X = n$,
then the topological singular set $S(X)$ 
is countably $(n-1)$-rectifiable.

Let us consider the following problem:

\begin{prob}
\label{prob: filtration}
Let $X$ be a $\GCBA$ space of $\dim X = n$.
Let
\[
\emptyset \subset S_1(X) \subset
\cdots \subset S_m(X) \subset \cdots \subset 
S_{n-1}(X) \subset S_n(X) \subset X
\tag{$\ast$}
\]
be the filtration of $X$ made by geometric singular sets.
\begin{enumerate}
\item
Study the geometric structure of each stratum $S_m(X) - S_{m-1}(X)$.
\item
Investigate what does it happen in $X$
in the case where some stratum $S_m(X) - S_{m-1}(X)$ is empty.
\end{enumerate}
\end{prob}

By the virtue of the structure theorem in \cite{lytchak-nagano1}
mentioned above,
if $\dim X = n$,
then the highest stratum $X - S_n(X)$ in $(\ast)$,
the $n$-regular set $R_n(X)$ in $X$, 
is dense in a Lipschitz $n$-manifold $M^n$
open dense in $X$
satisfying
$\dim_{\mathrm{H}} \left( \closure_X(X^n) - M^n \right)
\le n-1$
and 
$\dim_{\mathrm{H}} S_n(M^n) \le n-2$;
moreover, 
$M^n$ admits a $\DC$-structure,
and a $\BV_{\mathrm{loc}}$-Riemannian metric
continuous on $R_n(M^n)$.

In this paper,
we study Problem \ref{prob: filtration}
for the second highest stratum
$S_n(X) - S_{n-1}(X)$
in $(\ast)$.
We provide a geometric structure theorem of codimension one singularity,
and a geometric characterization of codimension two regularity.
These give us 
necessary and sufficient conditions for singular sets to be of codimension at least two.

Our study in this paper can be compared with 
the theory in the realm of $\CBB$ spaces 
with lower curvature bounds.
For every $\CBB$ space without boundary,
the topologically singular set has Hausdorff codimension 
at least two
(\cite[Remark 10.6.1]{burago-gromov-perelman}).
Our study of singularity of $\GCBA$ spaces
seems to be the first attempt to deal with the geometric structure 
of codimension one singularity in metric geometry.
For $\CBB$ spaces with lower curvature bounds,
Perelman \cite{perelman1, perelman2}
and Perelman--Petrunin \cite{perelman-petrunin}
have been successful in providing filtrations
in which each stratum is a topological manifold
in a different method.
Fujioka \cite{fujioka1, fujioka2} has recently brought 
their method for $\CBB$ spaces to $\GCBA$ spaces.
Our method of giving the filtration $(\ast)$ 
for $\GCBA$ spaces is more geometrically direct
in order to elaborate on their singularity.

\subsection{Wall singularity}

Let $X$ be a $\GCBA$ space.
For $m \in \N$,
we say that a point $x$ in $X$ is 
\emph{$m$-wall} 
if it is $m$-singular,
and if it is $(m-1)$-regular.
For a subset $A$ of $X$,
we denote by $W_m(A)$
the set of all $m$-wall points in $A$,
and call it
\emph{$m$-wall set in $A$}.
We notice that
\[
W_m(A) = S_m(A) \cap R_{m-1}(A) = S_m(A) - S_{m-1}(A).
\]

Throughout this paper,
we denote by $\Haus^m$
the $m$-dimensional Hausdorff measure.
Let $T_l^0$ denote
the discrete metric space consisting of 
$l$ points with pairwise distance $\pi$,
and $T_l^1$ the Euclidean cone $C_0(T_l^0)$ over $T_l^0$.
We notice that for every $n$-wall point $x \in W_n(X^n)$ 
in the $n$-dimensional part $X^n$,
the tangent space $T_xX$ is isometric to $\R^{n-1} \times T_l^1$
for some $l \in \N$ with $l \ge 3$.

One of the main results is the following wall singularity theorem:

\begin{thm}\label{thm: wst}
Let $X$ be a $\GCBA$ space.
Then for every $n$-wall point $x \in W_n(X^n)$ in $X^n$,
and for every open neighborhood $U$ of $x$,
there exist a point $x_0 \in W_n(X^n)$ arbitrarily close to $x$,
and an open neighborhood $U_0$ of $x_0$ contained in $U$,
such that $x_0 \in S(U_0)$,
and $U_0$ is homeomorphic to $\R^{n-1} \times T_{l_0}^1$
for some $l_0 \in \N$ with $l_0 \ge 3$;
moreover,
we have
\[
0 < \Haus^{n-1} \left( S(U_0) \right) < \infty.
\]
\end{thm}

\begin{rem}\label{rem: awst}
For every $l \in \N$ with $l \ge 4$,
there exists a purely $2$-dimensional $\GCBA$ space $X$
admitting a $2$-wall point $x \in W_2(X)$
at which $T_xX$ is isometric to $\R \times T_l^1$,
and of which every open neighborhood 
is not homeomorphic to $\R \times T_l^1$
(\cite[Examples 4.4 and 4.5]{nagano-shioya-yamaguchi1}).
This is the reason why we select a point $x_0$ 
avoiding $x$ in Theorem \ref{thm: wst}.
\end{rem}

We say that a separable metric space is 
\emph{purely $n$-dimensional}
if every non-empty open subset has topological dimension $n$.

Due to Theorem \ref{thm: wst},
we can describe a geometric characterization 
of codimension $2$ regularity
for $\GCBA$ spaces as follows:

\begin{thm}\label{thm: wrt}
Let $X$ be a $\GCBA$ space,
and let $U$ be a purely $n$-dimensional open subset of $X$.
Then the following are equivalent:
\begin{enumerate}
\item
the $n$-wall set $W_n(U)$ in $U$ is empty;
\item
there exists no open subset $U_0$ contained in $U$
for which $U_0$ is homeomorphic to $\R^{n-1} \times T_{l_0}^1$
for some $l_0 \in \N$ with $l_0 \ge 3$;
\item
$\dim S_n(U) \le n-2$;
\item
$\dim_{\mathrm{H}} S_n(U) \le n-2$;
\item
$\dim S(U) \le n-2$;
\item
$\dim_{\mathrm{H}} S(U) \le n-2$.
\end{enumerate}
\end{thm}

The codimension $2$ regularity of $\GCBA$ spaces
inherits the infinitesimal properties of the tangent spaces
and the spaces of directions.

\begin{thm}\label{thm: iwrt}
Let $X$ be a $\GCBA$ space,
and let $U$ be a purely $n$-dimensional open subset of $X$.
Then the following are equivalent:
\begin{enumerate}
\item
$\dim S(U) \le n-2$;
\item
$\dim S(T_xX) \le n-2$
for any $x \in U$;
\item
$\dim S(\Sigma_xX) \le n-3$
for any $x \in U$.
\end{enumerate}
\end{thm}

\subsection{Relaxed wall singularity}

To prove the results mentioned in Subsection 1.2,
we need a relaxed formulation for wall singularity.

Let $\delta \in (0,1)$.
Let $X$ be a $\GCBA$ space.
A point $x$ in $X$ is said to be \emph{$(m,\delta)$-regular}
if it is an $(m,\delta)$-strained point
(see Definitions \ref{defn: greg} and \ref{defn: mdstr}).
A point in $X$ is \emph{$(m,\delta)$-singular}
if it is not $(m,\delta)$-regular.
For a subset $A$ of $X$,
we denote 
by $R_{m,\delta}(A)$ 
the set of all $(m,\delta)$-regular points in $A$,
and by $S_{m,\delta}(A)$ 
the set of all $(m,\delta)$-singular points in $A$.
We call $R_{m,\delta}(A)$ and $S_{m,\delta}(A)$
the \emph{$(m,\delta)$-regular set in $A$} and 
the \emph{$(m,\delta)$-singular set in $A$},
respectively.
If $\dim X = n$,
then $S_{n+1,\delta}(X)$ coincides with $X$,
provided $\delta$ is sufficiently small
(see Lemma \ref{lem: ndr}).

Let $\delta, \epsilon \in (0,1)$ satisfy $\delta < \epsilon$.
A point in $X$ is \emph{$(m,\delta,\epsilon)$-wall}
if it is $(m,\epsilon)$-singular,
and if it is $(m-1,\delta)$-regular.
For a subset $A$ of $X$,
we denote 
by $W_{m,\delta,\epsilon}(A)$ 
the set of all $(m,\delta,\epsilon)$-wall points in $A$,
and call it the \emph{$(m,\delta,\epsilon)$-wall set in $A$}.
We notice that
\[
W_{m,\delta,\epsilon}(A) = S_{m,\epsilon}(A) \cap R_{m-1,\delta}(A)
= S_{m,\epsilon}(A) - S_{m-1,\delta}(A).
\]

Theorem \ref{thm: wst} follows from 
the following wall singularity theorem:

\begin{thm}\label{thm: dwst}
For every $n \in \N$,
there exist $\delta^{\ast} \in (0,1)$ 
and $\delta \in (0,\delta^{\ast})$ satisfying the following property:
Let $X$ be a $\GCBA$ space.
Then for every $x \in W_{n,\delta,\delta^{\ast}}(X^n)$,
and for every open neighborhood $U$ of $x$,
there exist a point $x_0 \in W_{n,\delta,\delta^{\ast}}(X^n)$ 
arbitrarily close to $x$,
and an open neighborhood $U_0$ of $x_0$ contained in $U$,
such that $x_0 \in S(U_0)$,
and $U_0$ is homeomorphic to $\R^{n-1} \times T_{l_0}^1$
for some $l_0 \in \N$ with $l_0 \ge 3$;
moreover,
we have $S(U_0) = S_{n,\delta^{\ast}}(U_0)$ and
\[
0 < \Haus^{n-1} \left( S(U_0) \right) < \infty.
\]
\end{thm}

At the same time proving Theorem \ref{thm: wrt},
we obtain the following:

\begin{thm}\label{thm: dwrt}
For every $n \in \N$,
there exist $\delta^{\ast} \in (0,1)$ 
and $\delta \in (0,\delta^{\ast})$ satisfying the following:
Let $X$ be a $\GCBA$ space,
and let $U$ be a purely $n$-dimensional open subset of $X$.
Then the following are equivalent:
\begin{enumerate}
\item
the $(n,\delta,\delta^{\ast})$-wall set 
$W_{n,\delta,\delta^{\ast}}(U)$ in $U$ is empty;
\item
there exists no open subset $U_0$ contained in $U$
for which $U_0$ is homeomorphic to $\R^{n-1} \times T_{l_0}^1$
for some $l_0 \in \N$ with $l_0 \ge 3$
such that $S(U_0) = S_{n,\delta^{\ast}}(U_0)$ and
$\Haus^{n-1} \left( S(U_0) \right)$ is positive and finite;
\item
$\dim S_{n,\delta^{\ast}}(U) \le n-2$;
\item
$\dim_{\mathrm{H}} S_{n,\delta^{\ast}}(U) \le n-2$;
\item
the $n$-wall set $W_n(U)$ in $U$ is empty.
\end{enumerate}
\end{thm}

\subsection{Applications}

We discuss some applications
based on the geometric characterization 
of codimension $2$ regularity mentioned above.

The author proved 
a \emph{volume sphere theorem for $\CAT(1)$ homology manifolds}
(\cite[Theorem 1.2]{nagano4}):
For every $m \in \N$,
there exists $\delta \in (0,\infty)$ such that
if a compact $\CAT(1)$ homology $m$-manifold $Z$ satisfies
\[
\Haus^m \left( Z \right)
< \frac{3}{2} \Haus^m \left( \Sph^m \right) + \delta,
\]
then $Z$ is homeomorphic to $\Sph^m$.

As a generalization,
we obtain the following sphere theorem:

\begin{thm}\label{thm: nwvptcat}
For every $m \in \N$,
there exists $\delta \in (0,\infty)$ satisfying the following property:
If a purely $m$-dimensional, compact, geodesically complete
$\CAT(1)$ space $Z$ with $\dim S(Z) \le n-2$
satisfies
\begin{equation}
\Haus^m \left( Z \right)
< \frac{3}{2} \Haus^m \left( \Sph^m \right) + \delta,
\label{eqn: nwvptcata}
\end{equation}
then $Z$ is homeomorphic to $\Sph^m$.
\end{thm}

We denote by $\omega_0^n(r)$ 
the $n$-dimensional Hausdorff measure of a metric ball in $\R^n$
of radius $r$.
For a point $p$ in a metric space,
we denote by $U_r(p)$ the open metric ball of radius $r$
around $p$.
For a metric space $X$,
we define a non-negative value $\mathcal{G}_0^n(X) \in [0,\infty]$ by
\begin{equation}
\mathcal{G}_0^n(X) := \limsup_{t \to \infty} 
\frac{\Haus^n \left( U_t(p) \right)}{\omega_0^n(t)}
\label{eqn: evg}
\end{equation}
for some base point $p$ in $X$.
We remark that
$\mathcal{G}_0^n(X)$ does not depend on the choice of base points.
We call $\mathcal{G}_0^n(X)$ the 
(\emph{upper})
\emph{$n$-dimensional Euclidean volume growth of $X$}.

Let $X$ be a proper, geodesically complete $\CAT(0)$ space.
A relative volume comparison of Bishop--Gromov type
(\cite[Proposition 6.3]{nagano2}, \cite[Proposition 3.2]{nagano4})
for $\CAT(0)$ spaces tells us that
if the $n$-dimensional part $X^n$ is non-empty,
then the following hold:
(1)
$\mathcal{G}_0^n(X) \ge 1$;
(2)
if $\mathcal{G}_0^n(X)$ is finite,
then for any $p \in X^n$ the limit superior in \eqref{eqn: evg} 
turns out to be the limit.

The author proved 
an \emph{asymptotic topological regularity theorem 
for $\CAT(0)$ homology manifolds}
(\cite[Theorem 1.3]{nagano5}):
For every $n \in \N$,
there exists $\delta \in (0,\infty)$ such that
if a $\CAT(0)$ homology $n$-manifold $X$ satisfies
\[
\mathcal{G}_0^n \left( X \right)
< \frac{3}{2} + \delta,
\]
then $X$ is homeomorphic to $\R^n$.

As a generalization,
we obtain:

\begin{thm}\label{thm: nwatrcat}
For every $n \in \N$,
there exists $\delta \in (0,\infty)$ satisfying the following property:
If a purely $n$-dimensional, proper, geodesically complete
$\CAT(0)$ space with $\dim S(X) \le n-2$ satisfies
\begin{equation}
\mathcal{G}_0^n \left( X \right)
< \frac{3}{2} + \delta,
\label{eqn: nwatrcata}
\end{equation}
then $X$ is homeomorphic to $\R^n$.
\end{thm}

\subsection{Sketch of the proof of the wall singularity theorem}

We explain a sketch of 
the proof of Theorem \ref{thm: dwst} on the wall singularity.

For $n \in \N$,
let $\delta^{\ast} \in (0,1)$ be sufficiently small,
and let $\delta \in (0,\delta^{\ast})$ be small enough for $\delta^{\ast}$.
Let $X$ be a $\GCBA$ space.
Take an $(n,\delta,\delta^{\ast})$-wall point 
$x \in W_{n,\delta,\delta^{\ast}}(X^n)$,
and an open neighborhood $U$ of $x$ in $X$.
We may assume that $U$ is purely $n$-dimensional.

\emph{Step} 1.
If an open neighborhood $V$ of $x$ contained in $U$
is sufficiently small,
then we can find an $(n-1,\delta)$-strainer map
$\varphi \colon V \to \R^{n-1}$
with $\varphi = (d_{p_1},\dots,d_{p_{n-1}})$,
where each $d_{p_i}$ is the distance function from $p_i$.
From the regularity of strainer maps
with distance coordinates,
it follows that
$\varphi$ has differential at all points in $V$,
and $\varphi$ is $(1+O_n(\delta))$-Lipschitz
and $(1+O_n(\delta))$-open,
where $O_n(\delta)$ is a positive constant 
depending only on $n$ and $\delta$
with $O_n(\delta) \to 0$ as $\delta \to 0$.

\emph{Step} 2.
For a point $y \in V$,
we denote by $\Pi_y$ the fiber of $\varphi$ through $y$
defined by $\Pi_y := \varphi^{-1}(\{\varphi(y)\})$.
Note that the fiber $\Pi_y$ is uniformly locally geometrically contractible
(\cite[Theorems 1.11 and 9.1]{lytchak-nagano1}).
In addition,
we see that
the intrinsic metric $d_{\Pi_y}$ on $\Pi_y$
is $(1+O_n(\delta))$-Lipschitz equivalent to the ambient metric on $X$.
If a sequence $(y_k)$ in $V$ converges to $y$,
then $d_{\Pi_{y_k}}$ is converges to $d_{\Pi_u}$ uniformly.

\emph{Step} 3.
Due to the Lytchak open map theorem \cite[Proposition 1.1]{lytchak2},
we can find an $(n,\delta^{\ast})$-singular point 
$x_0 \in S_{n,\delta^{\ast}}(V)$
arbitrarily close to the given point $x$
such that
for some sufficiently small open metric ball $U_{r_0}(x_0)$,
and for some $l_0 \in \N$ with $l \ge 3$,
the following hold for all $r \in (0,r_0/10)$:
(1)
The restriction $\varphi |_{S_{n,\delta^{\ast}}(U_r(x_0))}$ is 
a bi-Lipschitz embedding into $\R^n$.
(2)
For every $y \in S_{n,\delta^{\ast}}(U_r(x_0))$,
the slice $\Pi_y \cap U_r(x_0)$ of the fiber
is homeomorphic to $T_{l_0}^1$,
and isometric to a metric tree with single vertex $y$.
(3)
The subset $Y_r(x_0)$ of $U$ defined by
\[
Y_r(x_0) := \bigcup 
\left\{ \, \Pi_y \cap U_{r_0}(x_0) \mid 
y \in S_{n,\delta^{\ast}}(U_r(x_0)) \, \right\}
\]
is open in $U$.
Based on the properties (1)--(3),
we define a map
$\Phi$ from $Y_r(x_0)$ to $\R^{n-1} \times T_{l_0}^1$
by $\Phi(y) := \left( \varphi(y), d_{\Pi_y}(y_0,y) \right)$.
Then the map $\Phi$ is an open embedding.

Thus we finish the proof of Theorem \ref{thm: dwst}.

\subsection{Outline of this paper}

In Section 2,
we preprare basic concepts in the geometry 
of metric spaces with curvature bounded above.

From Section 3 to Section 6,
we discuss locally geodesically complete
open subsets of $\CBA$ spaces.
In Section 3,
we study a multi-fold spherical suspension structure
for $\CAT(1)$ spaces.
In Section 4,
we study a condition to recognize 
an almost multi-fold spherical suspension structure.
In Section 5,
we discuss the local geometric regularity
of locally geodesically complete open subsets of $\CBA$ spaces.
In Section 6,
we study the regularity of maps on them with distance coordinates.
The part from Section 3 to Section 6 for $\GCBA$ spaces
has been already discussed by 
Lytchak and the author \cite{lytchak-nagano1}.
For the proofs of our wall singularity theorems,
we need to generalize the studies in \cite{lytchak-nagano1}
slightly,
especially,
when we control the relaxed terms $\delta^{\ast}$ and $\delta$.

From Section 7,
we research the wall singularity of $\GCBA$ spaces.
In Section 7,
we discuss the structure of singular sets in $\GCBA$ spaces.
In Section 8,
we study the fibers of strainer maps with distance coordinates.
In Section 9,
we prove Theorems \ref{thm: dwst} and \ref{thm: wst}
on the wall singularity.
In Section 10,
we discuss a geometric characterization of codimension $2$ regularity,
and prove Theorems \ref{thm: wrt}, \ref{thm: iwrt},
and \ref{thm: dwrt}.
In Section 11,
we provide applications,
and show Theorems \ref{thm: nwvptcat}
and \ref{thm: nwatrcat}.

\subsection*{Acknowledgments}

The author would like to express his gratitude to
Alexander Lytchak 
for valuable discussions 
in continuous private communications.
The author would like to thank 
Tadashi Fujioka,
Takashi Shioya, and Takao Yamaguchi 
for their interests in this work.

\section{Preliminaries}

We refer the readers to 
\cite{alexander-kapovitch-petrunin-0},
\cite{alexander-kapovitch-petrunin}, 
\cite{alexandrov-berestovskii-nikolaev},
\cite{ballmann}, 
\cite{bridson-haefliger}, 
\cite{burago-burago-ivanov},
\cite{buyalo-schroeder} 
for basic facts in metric geometry,
especially on spaces with curvature bound above.

\subsection{Metric spaces}

Let $r \in (0,\infty)$.
For a point $p$ in a metric space,
we denote by $U_r(p)$, $B_r(p)$, and $S_r(p)$
the open metric ball of radius $r$ around $p$,
the closed one, and the metric sphere, respectively.

Let $X$ be a metric space with metric $d_X$.
Let $d_X \wedge \pi$ be the $\pi$-truncated metric on $X$
defined by $d_X \wedge \pi := \min \{ d_X, \pi \}$.
The 
\emph{Euclidean cone $C_0(X)$ over $X$}
is defined as the cone 
$[0,\infty) \times X / \{ 0 \} \times X$
over $X$
equipped with the Euclidean metric
$d_{C_0(X)}$ given by
\[
d_{C_0(X)} \left( [(t_1,x_1)], [(t_2,x_2)] \right)^2 := 
t_1^2 + t_2^2 - 2t_1t_2 \cos \left( (d_X \wedge \pi) (x_1,x_2) \right).
\]
We write an element $[(t,x)]$ in $C_0(X)$ as $tx$,
and denote by $0$ the vertex of $C_0(X)$.
We set $|tx| := t$.
For metric spaces $Y$ and $Z$,
we denote by $Y \ast Z$ the spherical join of $Y$ and $Z$.
Note that $C_0(Y \ast Z)$ is isometric to 
the $\ell^2$-direct product metric space
$C_0(Y) \times C_0(Z)$
of $C_0(Y)$ and $C_0(Z)$.

\subsection{Maps between metric spaces}

Let $c \in (0,\infty)$.
Let $X$ be a metric space with metric $d_X$,
and $Y$ a metric space with metric $d_Y$.
A map $f \colon X \to Y$ is said to be
\emph{$c$-Lipschitz} 
if $d_Y(f(x_1),f(x_2)) \le cd_X(x_1,x_2)$
for all $x_1, x_2 \in X$.
A map $f \colon X \to Y$ is said to be
\emph{$c$-bi-Lipschitz} 
if $f$ is $c$-Lipschitz,
and if $d_X(x_1,x_2) \le c d_Y(f(x_1),f(x_2))$
for all $x_1, x_2 \in X$
(consequently, $c \in [1,\infty)$).
A $1$-bi-Lipschitz homeomorphism
is nothing but an isometry,
and a $1$-bi-Lipschitz embedding is 
an isometric embedding.
A map $f \colon X \to Y$ is 
\emph{$c$-open}
if for any $r \in (0,\infty)$ and $x \in X$
such that $B_{cr}(x)$ is complete,
the ball $U_r(f(x))$ in $Y$ is contained in the image $f(U_{cr}(x))$
of the ball $U_{cr}(x)$ in $X$.
In the case where $X$ is complete,
if a map $f \colon X \to Y$ is $c$-open,
then $f$ is surjective.
Moreover,
in the case where $X$ is complete,
a $c$-Lipschitz map $f \colon X \to Y$
is a $c$-bi-Lipschitz homeomorphism for some $c \in [1,\infty)$
if and only if $f$ is an injective $c$-open map.

\subsection{Geodesic metric spaces}

Let $X$ be a metric space.
A \emph{geodesic $\gamma \colon I \to X$} 
means an isometric embedding from an interval $I$.
For a pair of points $p, q$ in $X$,
a \emph{geodesic $pq$ in $X$ from $p$ to $q$}
means the image of an isometric embedding
$\gamma \colon [a,b] \to X$
from a bounded closed interval $[a,b]$
with $\gamma(a) = p$ and $\gamma(b) = q$.

For $r \in (0,\infty]$,
a metric space $X$ is said to be
\emph{$r$-geodesic}
if every pair of points in $X$ with distance smaller than $r$
can be joined by a geodesic in $X$.
A metric space is 
\emph{geodesic}
if it is $\infty$-geodesic.
A geodesic metric space is proper
if and only if
it is complete and locally compact.

For $r \in (0,\infty]$,
a subset $C$ of a metric space is said to be 
\emph{$r$-convex}
if $C$ itself is $r$-geodesic as a metric subspace, 
and if every geodesic joining two points in $C$
is contained in $C$.
A subset $C$ of a metric space is 
\emph{convex}
if $C$ is $\infty$-convex.

\subsection{Convergences of metric spaces}

Let
$d_{\GH}$ denote the Gromov--Hausdorff distance
between metric spaces.
Let $X$ be a metric space with metric $d_X$,
and $Y$ a metric space with metric $d_Y$.
For $\epsilon \in (0,\infty)$,
a map
$f \colon X \to Y$ 
is said to be an \emph{$\epsilon$-approximation}
if 
\[
\sup_{x_1, x_2 \in X}
\left\vert 
d_Y \left( f(x_1), f(x_2) \right) - d_X(x_1,x_2)
\right\vert < \epsilon,
\]
and if $\bigcup_{p \in f(X)} U_r(p)$ coincides with $Y$.
If $d_{\GH}(X,Y) < \epsilon$ for $\epsilon > 0$,
then there exists a $2\epsilon$-approximation
$f \colon X \to Y$.
Conversely,
if there exists an $\epsilon$-approximation $f \colon X \to Y$, 
then $d_{\GH}(X,Y) < 2\epsilon$.
We say that a sequence $(X_k)$ of metric spaces 
converges to a metric space $X$ 
\emph{in the Gromov--Hausdorff topology}
if $\lim_{k \to \infty} d_{\GH}(X_k,X) = 0$.
We say that a sequence $(X_k,p_k)$ of pointed geodesic metric spaces 
converges to a pointed metric space $(X,p)$
\emph{in the pointed Gromov--Hausdorff topology}
if for every $r \in (0,\infty)$
there exists a sequence $(\epsilon_k)$ in $(0,\infty)$
with $\lim_{k \to \infty}\epsilon_k = 0$
such that for each $k$
there exists an
$\epsilon_k$-approximation $\varphi_k \colon B_r(p) \to B_r(p_k)$
with $\varphi_k(p) = p_k$.

A \emph{non-principal ultrafilter $\omega$ on $\N$}
is a finitely additive probability measure on $\N$ such that
$\omega(A) = \{0,1\}$ for every subset $A$,
and $\omega(E) = 0$ if $E$ is finite.
If $k \in A$ with $\omega(A) = 1$,
then $k$ is said to be \emph{$\omega$-large}.
If $f \colon \N \to B$ is a sequence in a compact metric space $B$,
then there exists a unique $x$ in $B$ such that
$\omega(f^{-1}(U)) = 1$ 
for every neighborhood $U$ of $x$;
the unique $x$ is called the \emph{$\omega$-limit of $f$},
and denoted by $\ulim f$.

Throughout this paper,
we fix a non-principal ultrafilter $\omega$ on $\N$.
Let $(X_k,p_k)$ be a sequence of pointed metric spaces
with metrics $d_{X_k}$.
We denote by $X_{\omega}^0$
the set of all sequences $(x_k)$
with $x_k \in X_k$ 
such that
$d_{X_k}(p_k,x_k)$ are uniformly bounded.
Let 
$d_{\omega} \colon 
X_{\omega}^0 \times X_{\omega}^0 \to [0,\infty)$
be the function defined by
$d_{\omega}((x_k),(y_k)) := \ulim d_{X_k}(x_k,y_k)$.
We define 
$\ulim (X_k,p_k)$ as the quotient
metric space $(X_{\omega}^0,d_{\omega}) / d_{\omega}=0$,
called the \emph{ultralimit of $(X_k,p_k)$ 
with respect to $\omega$}.
For $\xi = [(x_k)] \in \ulim X_k$,
we write $x_k \to \xi$.
Note that
if $\diam X_k$ are uniformly bounded,
then $\ulim X_k$ does not depend on the choices of
$p_k$.
If each $X_k$ is complete,
then so is $\ulim X_k$.
If each $X_k$ is $r$-geodesic,
then so is $\ulim X_k$.
If each $X_k$ is geodesically complete,
then so is $\ulim X_k$.
If a sequence $(X_k)$ of compact metric spaces 
converges to some compact metric space $X$ 
in the Gromov--Hausdorff topology,
then $\ulim X_k$ is isometric to $X$.
If a sequence $(X_k,p_k)$ of pointed proper geodesic metric spaces 
converges to a pointed metric space $(X,p)$
in the pointed Gromov--Hausdorff topology,
then $\ulim X_k$ is isometric to $X$ and 
$p_k \to p$.

\subsection{CAT$\boldsymbol{(\kappa)}$ spaces}

For $\kappa \in \R$,
we denote by $M_{\kappa}^n$ 
the simply connected, complete Riemannian $n$-manifold 
of constant curvature $\kappa$,
and denote by $D_{\kappa}$ the diameter of $M_{\kappa}^n$.
A metric space $X$ is said to be 
$\CAT(\kappa)$
if $X$ is $D_{\kappa}$-geodesic,
and if every geodesic triangle in $X$ with perimeter 
smaller than $2D_{\kappa}$
is not thicker than the comparison triangle 
with the same side lengths in $M_{\kappa}^2$.

Let $X$ be a $\CAT(\kappa)$ space.
Every pair of points in $X$ with distance smaller than $D_{\kappa}$
can be uniquely joined by a geodesic.
Let $p \in X$.
For every $r \in (0,D_{\kappa}/2]$,
the balls $U_r(p)$ and $B_r(p)$ are convex.
Along the geodesics emanating from $p$,
for every $r \in (0,D_{\kappa})$
the balls $U_r(p)$ and $B_r(p)$ are contractible inside themselves.
Every open subset of $X$ is an $\ANR$ 
(\cite{kramer}).
For $x, y \in U_{D_{\kappa}}(p) - \{p\}$,
we denote by $\angle pxy$
the angle at $p$
between $px$ and $py$.
Put $\Sigma_p'X := \{ \, px \mid x \in U_{D_{\kappa}}(p) - \{p\} \, \}$.
The angle $\angle$ at $p$
is a pseudo-metric on $\Sigma_p'X$.
The 
\emph{space of directions $\Sigma_pX$ at $p$}
is defined as the $\angle$-completion of
the quotient metric space $\Sigma_p'X / \angle = 0$.
For $x \in U_{D_{\kappa}}(p) - \{p\}$,
we denote by $x_p' \in \Sigma_pX$
the starting direction of $px$ at $p$.
The 
\emph{tangent space $T_pX$ at $p$} 
is defined as $C_0(\Sigma_pX)$.
The space $\Sigma_pX$ is $\CAT(1)$,
and the space $T_pX$ is $\CAT(0)$.
In fact,
for a metric space $\Sigma$,
the Euclidean cone $C_0(\Sigma)$ 
is $\CAT(0)$ if and only if $\Sigma$ is $\CAT(1)$.
For two metric spaces $Y$ and $Z$,
the spherical join $Y \ast Z$ is $\CAT(1)$
if and only if $Y$ and $Z$ are $\CAT(1)$.

Let $\ulim (X_k,p_k)$ be the ultraimit of a sequence of 
pointed $\CAT(\kappa)$ spaces.
If for some $r \in (0,D_{\kappa}/2)$
we take $x_k, y_k \in U_{D_{\kappa}/2}(p_k) - B_r(p_k)$,
then we have the following upper semi-continuity of angles:
\begin{equation}
\ulim \angle p_kx_ky_k \le \angle pxy.
\label{eqn: uscangle}
\end{equation}

\subsection{Geodesically complete CAT$\boldsymbol{(\kappa)}$ spaces}

We refer the readers to \cite{lytchak-nagano1} for 
the basic properties of $\GCBA$ spaces, that is,
locally compact, separable, locally geodesically complete
metric spaces with an upper curvature bound.
Recall that
a $\CAT(\kappa)$ space is said to be 
\emph{locally geodesically complete}
(or has \emph{geodesic extension property})
if every geodesic defined on a compact interval can be extended to
a local geodesic beyond endpoints.
A $\CAT(\kappa)$ space is 
\emph{geodesically complete}
if every geodesic can be extended to a local geodesic defined on $\R$.
Every locally geodesically complete, complete $\CAT(\kappa)$ space
is geodesically complete.
The geodesical completeness for compact (resp.~proper) 
$\CAT(\kappa)$ spaces
is preserved under the (resp.~pointed) Gromov--Hausdorff limit.

Let $X$ be a proper, geodesically complete $\CAT(\kappa)$ space.
For every $p \in X$,
the space $\Sigma_pX$ is compact and geodesically complete,
and $T_pX$ is proper and geodesically complete.
In fact,
for a $\CAT(1)$ space $\Sigma$,
the Euclidean cone $C_0(\Sigma)$ is geodesically complete
if and only if $\Sigma$ is geodesically complete
and not a singleton.
For two $\CAT(1)$ spaces $Y$ and $Z$,
the spherical join $Y \ast Z$ is geodesically complete
if and only if
$Y$ and $Z$ are geodesically complete
and not a singleton.

\subsection{Dimension of CAT$\boldsymbol{(\kappa)}$ spaces}

The \emph{geometric dimension} $\dim_{\mathrm{G}}$ 
is defined as
the smallest function from the class of $\CBA$ spaces
to $\N \cup \{\infty\}$
such that
(1) $\dim_{\mathrm{G}} X = 0$ if and only if $X$ is discrete;
(2) $\dim_{\mathrm{G}} X \ge 1 + \dim_{\mathrm{G}} \Sigma_pX$ 
for all $p \in X$.
In this case,
we have
\[
\dim_{\mathrm{G}} X = 1 + \sup_{p \in X} \dim_{\mathrm{G}} \Sigma_pX.
\]

Let $X$ be a $\CBA$ space.
Then $\dim_{\mathrm{G}} X$ coincides the supremum
of covering dimensions $\dim B$ of compact subsets $B$ of $X$;
if in addition $X$ is separable,
then $\dim_{\mathrm{G}} X = \dim X$
(\cite[Theorem A]{kleiner}).

Let $\ulim (X_k,p_k)$ be the ultraimit of a sequence of 
pointed $\CAT(\kappa)$ spaces.
Then we have the following (\cite[Lemma 11.1]{lytchak2}):
\begin{equation}
\dim_{\mathrm{G}} \left( \ulim (X_k,p_k) \right) 
\le \ulim \dim_{\mathrm{G}} X_k.
\label{eqn: lscdim}
\end{equation}

Let $X$ be a $\GCBA$ space.
Every relatively compact open subset of $X$
has finite covering dimension
(see \cite[Subsection 5.3]{lytchak-nagano1}).
The covering dimension $\dim X$ 
is equal to the Hausdorff dimension $\dim_{\mathcal{H}}X$ of $X$,
and equal to the supremum of $m$
such that $X$ has an open subset $U$ homeomorphic to $\R^m$
(\cite[Theorem 1.1]{lytchak-nagano1}).

We say that a separable metric space is
\emph{pure-dimensional}
if it is purely $n$-dimensional for some $n$.

We have the following characterization 
(\cite[Proposition 8.1]{lytchak-nagano2}):

\begin{prop}\label{prop: pure}
\emph{(\cite{lytchak-nagano2})}
Let $X$ be a proper, geodesically complete,
geodesic $\CAT(\kappa)$ space.
Let $W$ be a connected open subset of $X$.
Then the following are equivalent:
\begin{enumerate}
\item
$W$ is pure-dimensional;
\item
for every $p \in W$ the space $\Sigma_pX$ is pure-dimensional;
\item
for every $p \in W$ the space $T_pX$ is pure-dimensional.
\end{enumerate}
\end{prop}

We also know the following property
(\cite[Lemma 2.3]{nagano4}):

\begin{lem}
\label{lem: limpd}
\emph{(\cite{nagano4})}
Let $(X_k,p_k)$
be a sequence of pointed 
proper geodesic $\GCBA$ $\CAT(\kappa)$ spaces
converging to some pointed metric space $(X,p)$
in the pointed Gromov--Hausdorff topology.
If all the spaces $X_k$
are purely $n$-dimensional,
then the space $X$ is purely $n$-dimensional.
\end{lem}

\section{Spherical structure of CAT(1) spaces}

In this section,
we formulate a condition 
to find a multi-fold spherical suspension structure
for $\CAT(1)$ spaces.
We refer to \cite{lytchak2} for related studies of Lytchak
for geodesically complete $\CAT(1)$ spaces.

\subsection{Lune lemmas for CAT(1) spaces}

Let $Z$ be a $\CAT(1)$ space with metric $d_Z$.
We say that two subsets $A_1, A_2$ of $Z$ are
\emph{perpendicular}
if for each $(z_1,z_2) \in A_1 \times A_2$
we have $d_Z(z_1,z_2) = \pi/2$.

Let $A_1, A_2$ be mutually perpendicular subsets of $Z$.
Let $A_1 \ast A_2$ be the spherical join of 
the metric subspaces $A_1$ and $A_2$.
We define a continuous map 
$\Theta \colon A_1 \ast A_2 \to Z$ by
$\Theta(t, z_1,z_2) := 
\gamma_{z_1z_2}(t)$,
where 
$\gamma_{z_1z_2} \colon [0,\pi/2] \to Z$ 
is a (unique) geodesic from $z_1$ to $z_2$.
We call $\Theta$ the \emph{ruling map} joining $A_1$ and $A_2$.
Let $z_1 \in A_1$ and $z_2 \in A_2$,
and let $z_1^{\ast} = [(0,z_1,\cdot)]$ and 
$z_2^{\ast} = [(\pi/2,\cdot,z_2)]$
be the points in $A_1 \ast A_2$ corresponding to
$z_1$ and $z_2$, respectively.
For the geodesic $z_1^{\ast}z_2^{\ast}$ of length $\pi/2$
in $A_1 \ast A_2$,
the ruling map $\Theta$ sends $z_1^{\ast}z_2^{\ast}$
to $z_1z_2$ isometrically.

We review the following lune lemma,
which can be verified similarly to
the original lune lemma shown by Ballmann--Brin 
\cite[Lemma 2.5]{ballmann-brin}.

\begin{lem}\label{lem: lunelem}
Let $Z$ be a $\CAT(1)$ space.
Let $\{p_1,q_1\}$ and $\{p_2,q_2\}$
be mutually perpendicular subsets of $Z$.
Assume $d_Z(p_1,q_1) \ge \pi$.
Then the ruling map $\Theta_0 \colon \{p_1,q_1\} \ast \{p_2,q_2\} \to Z$
joining $\{p_1,q_1\}$ and $\{p_2,q_2\}$
is an isometric embedding.
If in addition $d_Z(p_2,q_2) < \pi$,
then $\{p_1,q_1\}$ and $p_2q_2$ are perpendicular, 
and the ruling map 
$\Theta \colon \{p_1,q_1\} \ast p_2q_2 \to Z$
joining $\{p_1,q_1\}$ and $p_2q_2$
is also an isometric embedding.
\end{lem}

The lune lemma can be generalized as follows
(\cite[Lemma 4.1]{lytchak2}):

\begin{lem}\label{lem: ruling}
\emph{(\cite{lytchak2})}
Let $Z$ be a $\CAT(1)$ space.
If $C_1$ and $C_2$ are mutually perpendicular, 
closed $\pi$-convex subsets of $Z$,
then the ruling map
$\Theta \colon C_1 \ast C_2 \to Z$
is an isometric embedding.
\end{lem}

\subsection{Suspenders in CAT(1) spaces}

Let $Z$ be a $\CAT(1)$ space with metric $d_Z$.
We say that 
a point $p_0$ in $Z$ is a \emph{suspender} in $Z$
if there exists another point $q_0$ in $Z$
such that for every $z \in Z$ we have
\begin{equation}
d_Z(p_0,z) + d_Z(z,q_0) = \pi.
\label{eqn: sph}
\end{equation}
In this case,
we say that $p_0$ and $q_0$ are \emph{opposite} to each other.
By \eqref{eqn: sph},
if $p_0$ and $q_0$ are mutually opposite suspenders in $Z$,
then $d_Z(p_0,q_0) = \pi$.

\begin{exmp}\label{exmp: ex1susp}
If $Z$ is isometric to a spherical suspension
$\Sph^0 \ast Z_0$
for some $\CAT(1)$ space $Z_0$,
then the suspension points $p_0, q_0$ in $Z$ 
with $\Sph^0 = \{ p_0, q_0 \}$
are mutually opposite suspenders in $Z$.
\end{exmp}

Let $z$ be a point in $Z$.
A point
$\bar{z}$ in $Z$ is said to be an \emph{antipode of $z$}
if $d_Z(z,\bar{z}) \ge \pi$.
The set of all antipodes of $z$
is denoted by $\Ant(z)$,
and called the \emph{antipodal set of $z$}.
If $Z$ is geodesically complete,
then for every $z \in Z$
the antipodal set $\Ant(z)$ of $z$ is non-empty,
provided $Z$ is not a one-point space.

\begin{lem}\label{lem: p1susp}
Let $Z$ be a $\CAT(1)$ space.
Assume that there exists mutually opposite suspenders
$p_0$ and $q_0$ in $Z$.
Then the following hold:
\begin{enumerate}
\item
$\Ant(p_0) = \{q_0\}$ and $\Ant(q_0) = \{p_0\}$;
moreover,
for each point $z \in Z - \{p_0,q_0\}$
the points $p_0$ and $q_0$ are 
joined by a geodesic of length $\pi$
through $z$;
\item
for all $z_1, z_2 \in Z$ with $d_Z(z_1,z_2) \ge \pi$
we have
\begin{equation}
d_Z(z_1,p_0) + d_Z(p_0,z_2) = \pi,
\quad
d_Z(z_1,q_0) + d_Z(q_0,z_2) = \pi;
\label{eqn: p1suspa}
\end{equation}
in particular,
we have $\diam Z = \pi$. 
\end{enumerate}
\end{lem}

\begin{proof}
Since $d_Z(p_0,q_0) = \pi$,
from \eqref{eqn: sph} we derive
$\Ant(p_0) = \{q_0\}$ and $\Ant(q_0) = \{p_0\}$.
Let $z \in Z-\{p_0,q_0\}$.
Then $z \in U_{\pi}^{\ast}(p_0) \cap U_{\pi}^{\ast}(q_0)$.
From \eqref{eqn: sph} we conclude that the union 
$p_0z \cup zq_0$ of the geodesics $p_0z$ and $zq_0$ 
is a geodesic of length $\pi$ joining $p_0$ and $q_0$.
Take $z_1, z_2 \in Z$ with $d_Z(z_1,z_2) \ge \pi$.
By \eqref{eqn: sph},
we have
$d_Z(p_0,z_1) + d_Z(z_1,q_0) = \pi$
and 
$d_Z(p_0,z_2) + d_Z(z_2,q_0) = \pi$.
These equalities lead to \eqref{eqn: p1suspa}.
\end{proof}

\begin{rem}\label{rem: afterp1susp0}
Let $p_0, q_0 \in Z$ satisfy $d_Z(p_0,q_0) = \pi$.
If the equalities \eqref{eqn: p1suspa} hold
for all $z_1, z_2 \in Z$ with $d_Z(z_1,z_2) \ge \pi$,
then $\Ant(p_0) = \{q_0\}$ and $\Ant(q_0) = \{p_0\}$.
In general, the converse is false.
\end{rem}

\begin{rem}\label{rem: afterp1susp}
Assume that $Z$ is geodesically complete.
Let $p_0, q_0 \in Z$ satisfy $d_Z(p_0,q_0) = \pi$.
Then the following are equivalent:
\begin{enumerate}
\item
$p_0$ and $q_0$ are mutually opposite suspenders in $Z$;
\item
$\Ant(p_0) = \{q_0\}$ and $\Ant(q_0) = \{p_0\}$;
\item
the equalities \eqref{eqn: p1suspa} hold
for all $z_1, z_2 \in Z$ with $d_Z(z_1,z_2) \ge \pi$.
\end{enumerate}
\end{rem}

\subsection{Round subsets of CAT(1) spaces}

Let $Z$ be a $\CAT(1)$ space with metric $d_Z$.
We say that
a subset $A$ of $Z$ is
\emph{symmetric}
if $A$ contains $\bigcup_{p \in A} \Ant(p)$.

Lytchak \cite{lytchak2} studied rigid properties derived from
the existence of symmetric subsets
of geodesically complete $\CAT(1)$ spaces.
Instead of the geodesical completeness,
we introduce the following condition.

We say that a symmetric subset $A$ of $Z$ is 
\emph{round}
if every point in $A$ is a suspender in $Z$.
From Lemma \ref{lem: p1susp} it follows that
if $Z$ admits a round subset,
then $\diam Z = \pi$.
A subset $A$ of $Z$ is round
if and only if
for each $p \in A$
there exists a unique point $q \in A$ such that
$p$ and $q$ are mutually opposite suspenders in $Z$.

\begin{exmp}\label{exmp: exround}
Assume that 
$Z$ isometrically splits as a spherical join $Z_1 \ast Z_2$.
Both the factors $Z_1$ and $Z_2$ are
symmetric in $Z$.
Assume in addition that $Z_1$ is isometric to $\Sph^{m-1}$.
For $l \in \{ 1, \dots, m \}$,
let $Z_1^{l-1}$ be a subset of $Z_1$
such that $Z_1^{l-1}$ is isometric to $\Sph^{l-1}$.
Then $Z_1^{l-1}$ is round.
\end{exmp}

For a subset $A$ of $Z$,
we denote by $A^{\perp}$
the set of all points $z$ in $Z$ with $d_Z(z,A) \ge \pi/2$,
and call it the \emph{polar set of $A$}.
If $C$ is a closed $\pi$-convex subset of $Z$,
then for each $z \in Z - \left( C \cup C^{\perp} \right)$
there exists a unique foot point 
$z_0$ in $C$ satisfying $d_Z(z,z_0) = d_Z(z,C)$
(see \cite[Proposition II.2.4 and Exercise II.2.6]{bridson-haefliger}).

\begin{lem}\label{lem: polar}
Let $Z$ be a $\CAT(1)$ space.
Let $C$ be a round, closed $\pi$-convex subset of $Z$.
For a point $z$ in $Z - \left( C \cup C^{\perp} \right)$,
let $z_0$ be a unique foot point in $C$
satisfying $d_z(z,z_0) = d_Z(z,C)$.
Then there exists a point $z_0^{\perp}$ in $Y^{\perp}$
such that $z \in z_0z_0^{\perp}$ and 
$d_Z \left( z_0, z_0^{\perp} \right) = \pi/2$.
In particular,
if $Z-C$ is non-empty,
then $C^{\perp}$ is non-empty too.
\end{lem}

\begin{proof}
The choice of $z_0$ implies $d_Z(z,z_0) \in (0,\pi/2)$.
Since $C$ is round,
there exists a unique point $\bar{z}_0$ in $C$
such that $z_0$ and $\bar{z}_0$ are mutually opposite suspenders in $Z$.
From Lemma \ref{lem: p1susp}
it follows that the union $z_0z \cup z\bar{z}_0$
of the geodesics $z_0z$ and $z\bar{z}_0$
is a geodesic $z_0\bar{z}_0$ of length $\pi$
joining $z_0$ and $\bar{z}_0$.
Choose the midpoint $z_0^{\perp}$ in $z_0\bar{z}_0$
between $z_0$ and $\bar{z}_0$.
The $\pi$-convexity of $C$ implies $z_0^{\perp} \in Z-C$.
To show $z_0^{\perp} \in C^{\perp}$,
take $p \in C - \{z_0,\bar{z}_0\}$.
For some unique $q \in C - \{z_0,\bar{z}_0\}$,
the points $p$ and $q$ are mutually opposite suspenders in $Z$.
The geodesics $z_0p$ and $z_0q$ of length $< \pi$
are contained in $C$.
From the choice of $z_0$ we derive
$\angle_{z_0} \left( p, z_0^{\perp} \right) \ge \pi/2$
and $\angle_{z_0} \left( q, z_0^{\perp} \right) \ge \pi/2$.
Hence
$d_Z \left( p, z_0^{\perp} \right) \ge \pi/2$ and 
$d_Z \left( q, z_0^{\perp} \right) \ge \pi/2$.
Therefore $z_0^{\perp} \in C^{\perp}$.
\end{proof}

The existence of a round closed $\pi$-convex subset
gives us a splitting.

\begin{lem}\label{lem: roundsplit}
Let $Z$ be a $\CAT(1)$ space.
Let $C$ be a round, closed $\pi$-convex subset of $Z$.
Then $C$ and the polar set $C^{\perp}$ of $C$ are perpendicular.
Moreover,
the ruling map $\Theta \colon C \ast C^{\perp} \to Z$
joining $C$ and $C^{\perp}$
is an isometry.
\end{lem}

\begin{proof}
We first show that 
$C$ and $C^{\perp}$ are perpendicular.
Take $p \in C$ and $z \in C^{\perp}$.
Choose $q \in C$ such that
$p$ and $q$ are mutually opposite suspenders in $Z$.
By Lemma \ref{lem: p1susp},
the union $pz \cup zq$ 
of the geodesics $pz$ and $zq$ is a geodesic of length $\pi$
joining $p$ and $q$. 
From the choice of $z \in C^{\perp}$
we derive $d_Z(p,z) = \pi/2$.
Hence $C$ and $C^{\perp}$ are perpendicular.

We next prove that $C^{\perp}$ is closed and $\pi$-convex.
Since $C$ is closed, so is $C^{\perp}$.
Take $z_1, z_2 \in C^{\perp}$ with $d_Z(z_1,z_2) < \pi$.
Choose $p_0, q_0 \in C$ such that
$p_0$ and $q_0$ are mutually opposite suspenders in $Z$.
Since $C$ and $C^{\perp}$ are perpendicular,
we already know $d_Z(p_0,z_i) = \pi/2$
and $d_Z(q_0,z_i) = \pi/2$ for all $i \in \{1,2\}$.
From Lemma \ref{lem: lunelem}
we conclude that 
$\{p_0,q_0\}$ and $z_1z_2$ are perpendicular.
Hence $z_1z_2$ are contained in $C^{\perp}$.
This implies the $\pi$-convexity of $C^{\perp}$.

By Lemma \ref{lem: ruling},
the ruling map $\Theta \colon C \ast C^{\perp} \to Z$ 
is an isometric embedding.
The surjectivity of $\Theta$ follows from Lemma \ref{lem: polar}. 
\end{proof}

\begin{rem}\label{rem: afterroundsplit}
If $Z$ is geodesically complete,
and if $C$ is a symmetric, closed $\pi$-convex subset of $Z$,
then $C$ and $C^{\perp}$ are perpendicular,
and the ruling map $\Theta \colon C \ast C^{\perp} \to Z$
is an isometry
(\cite[Proposition 4.2]{lytchak2}).
\end{rem}

\subsection{Multi-fold suspenders in CAT(1) spaces}

Let $Z$ be a $\CAT(1)$ space with metric $d_Z$.
Let $m \in \N$.
We say that an $m$-tuple $(p_1,\dots,p_m)$ of points in $Z$ is
an \emph{$m$-suspender} in $Z$
if there exists another $m$-tuple $(q_1,\dots,q_m)$ of points in $Z$
satisfying the following:
\begin{enumerate}
\item
for all $i \in \{ 1, \dots, m \}$,
the points $p_i$ and $q_i$ 
are mutually opposite suspenders in $Z$;
\item
for all distinct $i, j \in \{1, \dots, m\}$,
we have
$d_Z(p_i,p_j) = \pi/2$, 
$d_Z(p_i,q_j) = \pi/2$, 
and $d_Z(q_i,q_j) = \pi/2$.
\end{enumerate}
In this case, we say that $(p_1,\dots,p_m)$ and 
$(q_1,\dots,q_m)$ are \emph{opposite} to each other.

\begin{exmp}\label{exmp: exmsusp}
Assume that $Z$ isometrically splits as 
an $m$-fold spherical suspension $\Sph^{m-1} \ast Z_0$
for some $\CAT(1)$ space $Z_0$.
Let $(p_1,\dots,p_m)$ be an $m$-tuple 
of orthonormal points in $\Sph^{m-1}$.
For $i \in \{ 1, \dots, m \}$,
let $q_i \in \Sph^{m-1}$ 
be the unique antipode in $\Sph^{m-1}$ of $p_i$.
Then $(p_1,\dots,p_m)$ and $(q_1,\dots,q_m)$
are mutually opposite $m$-suspenders in $Z$.
\end{exmp}

From the existence of an $m$-suspenders,
we find an $m$-fold spherical suspension structure.
We begin with the following:

\begin{lem}\label{lem: roundcircle}
Let $Z$ be a $\CAT(1)$ space.
Let $p_0$ and $q_0$ are mutually opposite suspenders in $Z$.
Assume that
$p_1$ and $q_1$ are mutually opposite suspenders in $Z$
such that $d_Z(p_0,p_1) \in (0,\pi)$.
Then there exists a round, closed $\pi$-convex subset $C_1$ of $X$
containing $p_0, q_0$ and $p_1, q_1$
such that
$C_1$ is isometric to $\Sph^1$.
\end{lem}

\begin{proof}
Put $C_1 := p_0p_1 \cup p_1q_0 \cup q_0q_1 \cup q_1p_0$.
From Lemma \ref{lem: p1susp},
it follows that
for all distinct $i, j \in \{ 0, 1 \}$,
the unions $p_ip_j \cup p_jq_i$ and $p_iq_j \cup q_jq_i$
are geodesics of length $\pi$ joining $p_i$ and $q_i$.
Hence $Y_1$ is isometric to $\Sph^1$.

We prove the roundness of $C_1$.
Take $p \in C_1$.
By symmetry, we may assume $p \in p_0p_1$.
Since $C_1$ is isometric to $\Sph^1$,
there exists a unique point $q \in q_0q_1$ with $d_Z(p,q) = \pi$.
It suffices to show that
$p$ and $q$ are mutually opposite spherical points in $Z$.
Take $z \in Z-C_1$.
From Lemma \ref{lem: p1susp}
we conclude that
the unions $p_0z \cup zq_0$ and $p_1z \cup zq_1$
are geodesics of length $\pi$.
On the closed hemisphere $\Sph_+^2$,
there exists a point $\tilde{z}$ in the interior of $\Sph_+^2$,
and some points $\tilde{p}_0, \tilde{p}_1, \tilde{q}_0, \tilde{q}_1$
on the boundary equator of $\Sph_+^2$
such that
$\triangle \tilde{p}_0\tilde{p}_1\tilde{z}$
and
$\triangle \tilde{q}_0\tilde{q}_1\tilde{z}$
are comparison triangles for 
$\triangle p_0p_1z$ and $\triangle q_0q_1z$,
respectively.
Let $\tilde{p} \in \triangle \tilde{p}_0\tilde{p}_1\tilde{z}$
and $\tilde{q} \in \triangle \tilde{q}_0\tilde{q}_1\tilde{z}$
be comparison points for $p \in \triangle p_0p_1z$ 
and $q \in \triangle q_0q_1z$,
respectively.
On $\Sph_+^2$ we have
$d_{\Sph^2}(\tilde{p},\tilde{z}) + d_{\Sph^2}(\tilde{z},\tilde{q}) = \pi$,
where $d_{\Sph^2}$ is the standard metric on $\Sph^2$.
This together with the $\CAT(1)$ property of $Z$
implies $d_Z(p,z) + d_Z(z,q) = \pi$.
Thus 
$p$ and $q$ are mutually opposite spherical points in $Z$.
\end{proof}

We are going to prove the following:

\begin{prop}\label{prop: m0susp}
Let $Z$ be a $\CAT(1)$ space.
Assume that there exist mutually opposite $m$-suspenders
$(p_1,\dots,p_m)$ and $(q_1,\dots,q_m)$ in $Z$.
Then there exists 
a round, closed $\pi$-convex subset $C_{m-1}$ of $Z$
containing $(p_1,\dots,p_m)$ and $(q_1,\dots,q_m)$
such that $C_{m-1}$ is isometric to $\Sph^{m-1}$.
Moreover,
$Z$ isometrically splits as a spherical join
$C_{m-1} \ast C_{m-1}^{\perp}$,
where $C_{m-1}^{\perp}$ is the polar set of $C_{m-1}$.
\end{prop}

\begin{proof}
The proof is accomplished by induction on $m$.
The case of $m=1$ is obvious.
Assume that the proposition holds true for $m$.
Let $(p_1,\dots,p_m,p_{m+1})$ and $(q_1,\dots,q_m,q_{m+1})$ are
mutually opposite $(m+1)$-suspenders in $Z$.
For $i \in \{ 1, \dots, m, m+1 \}$,
put $p_{+i} := p_i$ and $p_{-i} := q_i$.
By the inductive assumption,
we find a round, closed $\pi$-convex subset 
$C_{m-1}$ of $X$ containing 
$(p_{\pm 1},\dots,p_{\pm m})$
such that
$Z$ isometrically splits as a spherical join
$C_{m-1} \ast C_{m-1}^{\perp}$,
where $C_{m-1}$ is isometric to $\Sph^{m-1}$.

We now prove that
$\{p_{+(m+1)},p_{-(m+1)}\}$ and $C_{m-1}$ are perpendicular.
If $m=1$, 
then this claim follows from the definition.
Let $m \ge 2$.
Take $z \in C_{m-1}$.
Since $C_{m-1}$ is isometric to $\Sph^{m-1}$,
for some $i \in \{ \pm 1, \dots, \pm m \}$
we have $d_Z(p_i,z) \le \pi/2$. 
We may assume $d_Z(p_{+1},z) \le \pi/2$.
As an equator of $C_{m-1}$,
we find a round, closed $\pi$-convex subset 
$C_{m-2}$ of $Z$ containing 
$(p_{\pm 2},\dots,p_{\pm m})$
such that $C_{m-2}$ is isometric to $\Sph^{n-2}$.
Let $z_0$ be a unique foot point in $C_{m-2}$
with $d_Z(z,z_0) = d_Z(z,C_{m-2})$.
By the assumption,
there exists another, 
round, closed $\pi$-convex subset $C_{m-1}'$
of $Z$ containing 
$(p_{\pm 2},\dots,p_{\pm (m+1)})$
such that $C_{m-1}'$ is isometric to $\Sph^{m-1}$.
Since $C_{m-2}$ is also an equator of $C_{m-1}'$,
we have
$d_Z(p_{\pm (m+1)},z_0) = \pi/2$.
Looking at 
$\triangle p_{\pm (m+1)}p_{+1}z_0$,
we obtain $d_Z(p_{\pm (m+1)},z) = \pi/2$;
indeed,
by Lemma \ref{lem: lunelem}
the ruling map
joining $\{p_{+(m+1)},p_{-(m+1)}\}$ and $p_{+1}z_0$
is an isometric embedding.
Thus $\{p_{+(m+1)},p_{-(m+1)}\}$ and $C_{m-1}$ are perpendicular.

From Lemma \ref{lem: ruling}, we conclude that
the ruling map
\[
\Theta_m \colon 
\{p_{+(m+1)},p_{-(m+1)}\} \ast C_{m-1} \to Z
\]
joining $\{p_{+(m+1)},p_{-(m+1)}\}$ and $C_{m-1}$
is an isometric embedding.
Let $C_m$ be the image of $\Theta_m$.
This is isometric to $\Sph^m$.
Lemma \ref{lem: roundcircle}
implies that
$C_m$ is a round, closed $\pi$-convex subset of $Z$,
and Lemma \ref{lem: roundsplit} does that
$Z$ isometrically splits as 
$C_m \ast C_m^{\perp}$.
\end{proof}

\subsection{Progress of spherical splittings of CAT(1) spaces}

Let $Z$ be a $\CAT(1)$ space with metric $d_Z$.
If $Z$ admits an $m$-suspenders,
then Proposition \ref{prop: m0susp} implies that
$Z$ isometrically splits as an $m$-fold spherical suspension.
The next lemma gives us a criterion
for an issue whether
the space $Z$ splits as an $(m+1)$-fold spherical suspension.

\begin{lem}\label{lem: pm0susp}
Let $Z$ be a $\CAT(1)$ space.
Let $(p_1,\dots,p_m)$ be an $m$-suspender in $Z$.
Assume that there exists a suspender $p_0$ in $Z$
with
\begin{equation}
\sum_{i=1}^m \cos^2 d_Z(p_i,p_0) < 1.
\label{eqn: pm0suspa}
\end{equation}
Then we find $p_{m+1} \in Z$
for which the $(m+1)$-tuple $(p_1,\dots,p_m,p_{m+1})$
is an $(m+1)$-suspender in $Z$.
\end{lem}

\begin{proof}
By Proposition \ref{prop: m0susp},
there exists a round, closed $\pi$-convex subset $C_{m-1}$ of $Z$
containing $(p_1,\dots,p_m)$
such that $Z$ splits as 
$C_{m-1} \ast C_{m-1}^{\perp}$,
where
$C_{m-1}$ is isometric to $\Sph^{m-1}$.
Take the suspender $q_0$ in $Z$ opposite to $p_0$.
The assumption \eqref{eqn: pm0suspa}
implies $p_0 \in Z-C_{m-1}$ and 
$q_0 \in Z-C_{m-1}$.

If $p_0 \in C_{m-1}^{\perp}$,
then $q_0 \in C_{m-1}^{\perp}$,
and hence by Lemma \ref{lem: roundsplit}
the $(m+1)$-tuple
$(p_1,\dots,p_m,p_0)$
is an $(m+1)$-suspenders in $Z$.

Next we assume 
$p_0 \not\in C_{m-1}^{\perp}$.
Let $z_0 \in C_{m-1}$
be a unique foot point with 
$d_Z(p_0,z_0) = d_Z(p_0,C_{m-1})$,
and take $\bar{z}_0 \in C_{m-1}$
such that $z_0$ and $\bar{z}_0$ are mutually opposite
suspenders in $Z$.
By Lemma \ref{lem: polar},
there exists $p_{m+1} \in C_{m-1}^{\perp}$
with $p_0 \in z_0p_{m+1}$ and $d_Z(z_0,p_{m+1}) = \pi/2$.
We denote by $C_1$ the union of the four geodesics
$p_0z_0$, $z_0q_0$, 
$q_0\bar{z}_0$, and $\bar{z}_0p_0$.
This is isometric to $\Sph^1$.
Lemma \ref{lem: roundcircle} leads to the roundness of $C_1$.
Hence there exists $q_{m+1} \in C_1 \cap C_{m-1}^{\perp}$
such that $p_{m+1}$ and $q_{m+1}$ 
are mutually opposite suspenders in $Z$.
From Lemma \ref{lem: roundsplit} it follows that
the $(m+1)$-tuple
$(p_1,\dots,p_m,p_{m+1})$
is an $(m+1)$-suspender in $Z$.
\end{proof}

\subsection{Spherical cores of CAT(1) spaces}

Let $Z$ be a $\CAT(1)$ space with metric $d_Z$.
We denote by $C_Z$
the set of all spherical points in $Z$.
We call $C_Z$
the \emph{spherical core} of $Z$.
The set $C_Z$ is possibly empty.

As an application of Proposition \ref{prop: m0susp},
we conclude the following:

\begin{prop}\label{prop: sphcore}
Let $Z$ be a $\CAT(1)$ space.
Assume that the spherical core $C_Z$ of $Z$ is non-empty.
Then $C_Z$ is isometric to a unit sphere
in some Hilbert space.
Moreover,
$Z$ isometrically splits as a spherical join
$C_Z \ast C_Z^{\perp}$.   
\end{prop}

\begin{proof}
First we prove that
$C_Z$ is a round, closed $\pi$-convex subset of $Z$.
The set $C_Z$ is symmetric in $Z$,
and also round in $Z$.
To show the closedness of $C_Z$,
we take a sequence $(p_k)$ in $C_Z$
converging to some $p \in Z$.
For each $k \in \N$,
choose $q_k \in C_Z$ such that
$p_k$ and $q_k$ are mutually opposite suspenders in $Z$.
By Lemma \ref{lem: roundcircle},
for all distinct $k, l \in \N$,
we have
$d_Z(q_k,q_l ) = d_Z(p_k, p_l)$.
Hence $(q_k)$ is a Cauchy sequence in $Z$.
Since $Z$ is assumed to be complete,
we find the limit point $q$ in $Z$ of $(q_k)$.
For every $z \in Z$,
we see
$d_Z(p,z) + d_Z(z,q) = \pi$.
This implies $p \in C_Z$, and hence $C_Z$ is closed in $Z$.
The $\pi$-convexity of $C_Z$ follows from Lemma \ref{lem: roundcircle}.

Assume that there exists an $m$-suspenders $(p_1,\dots,p_m)$ in $Z$.
By Proposition \ref{prop: m0susp},
there exists a round, closed $\pi$-convex subset $C_{m-1}$ of $Z$
containing $(p_1,\dots,p_m)$
such that $Z$ isometrically splits as
$C_{m-1} \ast C_{m-1}^{\perp}$,
where
$C_{m-1}$ is isometric to $\Sph^{m-1}$.
Assume that $C_Z - C_{m-1}$ is non-empty.
As shown above,
the spherical core $C_Z$ is a round, closed $\pi$-convex subset of $Z$.
Using Lemma \ref{lem: pm0susp},
we see that there exists a point $p_{m+1}$ in $Z$ for which
$(p_1,\dots,p_m,p_{m+1})$ is an $(m+1)$-suspender  in $Z$.
By Proposition \ref{prop: m0susp},
there exists a round, closed $\pi$-convex subset $C_m$ of $Z$
containing $(p_1,\dots,p_m,p_{m+1})$
such that $Z$ isometrically splits as 
$C_m \ast C_m^{\perp}$,
where
$C_m$ is isometric to $\Sph^m$.
Thus we have inductively proved the proposition.
\end{proof}

\begin{rem}\label{rem: aftersphcore}
It was shown in \cite[Corollary 4.4 and Remark 4.1]{lytchak2} that
Proposition \ref{prop: sphcore}
holds true for geodesically complete $\CAT(1)$ spaces.
\end{rem}

\section{Almost spherical structure of CAT(1) spaces}

In this section,
we study a condition to recognize an almost spherical suspension
structure for $\CAT(1)$ spaces.
We refer to \cite[Section 6]{lytchak-nagano1}
for related studies 
for compact geodesically complete $\CAT(1)$ spaces.

\subsection{Penetrability}

Let $Z$ be a $\CAT(1)$ space with metric $d_Z$.
We say that $Z$ is 
\emph{penetrable}
if there exists a dense subset $Z_0$ of $Z$
such that for each $p \in Z_0$
there exists $q \in Z_0$ with $d_Z(p,q) \ge \pi$.
If $Z$ is penetrable, then $\diam Z \ge \pi$.
If $Z$ has at least two path-connected components,
then it is penetrable.

The penetrability of $\CAT(1)$ spaces is closed 
under the ultralimit.
More precisely,
for every sequence of pointed 
penetrable $\CAT(1)$ spaces $(Z_k,z_k)$,
the ultralimit $\ulim (Z_k,z_k)$ is penetrable too.

We discuss the penetrability 
by reason of the following examples.

\begin{exmp}\label{exmp: gcpenet}
If $Z$ is geodesically complete,
then $Z$ is penetrable,
provided $Z$ contains at least two points.
Indeed, for each $p \in Z$,
and for each $z \in U_{\pi}^{\ast}(p)$,
there exists a geodesic $pq$ of length $\pi$
for some $q \in \Ant(p)$
as a prolongation of the geodesic $pz$.
\end{exmp}

\begin{exmp}\label{exmp: linkpenet}
Let $U$ be a locally geodesically complete open subset 
of a $\CAT(\kappa)$ space $X$.
Then for each $x \in U$
the space of directions $\Sigma_xX$ is penetrable.
Indeed, 
for the dense subset $\Sigma_x'X$ 
consisting of all genuine directions at $x$,
for each $\xi \in \Sigma_x'X$
there exists $\eta \in \Sigma_x'X$
with $d_{\Sigma_xX}(\xi,\eta) = \pi$.
\end{exmp}

There exists a geodesically complete $\CAT(1)$ space
possessing some point at which 
the space of directions is not geodesically complete.

\begin{exmp}\label{exmp: linknotgc}
Let $\Sph_+^m$ be a closed unit $m$-hemisphere 
with boundary equator $\Sph^{m-1}$.
For each $z \in \Sph^{m-1}$,
let us prepare a closed half line $[a_z,\infty)$
for some $a_z \in \R$.
Let
$Z := 
\Sph_+^m \sqcup (\bigsqcup_{z \in \Sph^{m-1}} [a_z,\infty)) / z=a_z$
be the quotient space obtained by attaching
$[a_z,\infty)$ to $\Sph_+^{m-1}$ at all $z \in \Sph^{m-1}$.
This becomes a geodesically complete $\CAT(1)$ space
equipped with the intrinsic metric induced from $\Sph_+^m$
and the half lines.
For each $z \in \Sph^{m-1}$ in $Z$,
the space of directions $\Sigma_zZ$ 
is the disjoint union of $\Sph_+^{m-1}$ and a single point.
\end{exmp}

For the sake of studies of the spaces of directions 
in locally geodesically complete 
open subsets of $\CAT(\kappa)$ spaces,
we consider penetrable $\CAT(1)$ spaces
instead of geodeically complete ones.

\subsection{Relaxed suspenders in CAT(1) spaces}

Let $\delta \in (0,1)$.
Let $Z$ be a $\CAT(1)$ space with metric $d_Z$.
We say that 
a point $p_0$ in $Z$ is a \emph{$(1,\delta)$-suspender} 
if there exists another point $q_0$ in $Z$
satisfying
\begin{equation}
\sup_{z \in Z} \left\{ \, d_Z(p_0,z) + d_Z(z,q_0) \, \right\} < \pi + \delta.
\label{eqn: dsph}
\end{equation}
In this case,
we say that $p_0$ and $q_0$ are \emph{opposite} to each other.

For penetrable $\CAT(1)$ spaces,
the notion of the $(1,\delta)$-suspenders 
is a perturbed one of the rigid $1$-suspenders. 
 
\begin{lem}\label{lem: o1dsusp}
Let $Z$ be a penetrable $\CAT(1)$ space.
Let $p_0$ and $q_0$ be mutually opposite 
$(1,\delta)$-suspenders in $Z$.
Then 
\begin{equation}
\inf_{z \in Z} \left\{ \, d_Z(p_0,z) + d_Z(z,q_0) \, \right\} > \pi - \delta.
\label{eqn: o1dsuspa}
\end{equation}
\end{lem}

\begin{proof}
By \eqref{eqn: dsph},
for some sufficiently small $a \in (0,\delta)$,
we have
\[
\sup_{z \in Z} \left\{ \, d_Z(p_0,z) + d_Z(z,q_0) \, \right\} 
< \pi + \delta - a.
\]
Take $z \in Z$.
Since $Z$ is penetrable,
we find a point $z_0 \in Z$
arbitrarily close to $z$,
and an antipode $\bar{z}_0 \in \Ant(z_0)$ of $z_0$.
Then
\[
\pi \le d_Z(z_0,p_0) + d_Z(p_0,\bar{z}_0),
\quad
\pi \le d_Z(z_0,q_0) + d_Z(q_0,\bar{z}_0).
\]
Summing up these inequalities,
we obtain
\[
d_Z(p_0,z_0) + d_Z(z_0,q_0) > \pi - \delta + a.
\]
Since $z_0$ can be chosen arbitrarily closely to $z$,
we see \eqref{eqn: o1dsuspa}.
\end{proof}

\begin{rem}\label{rem: aftero1dsusp}
Let $Z$ be a penetrable $\CAT(1)$ space.
From \eqref{eqn: dsph} and \eqref{eqn: o1dsuspa} it follows that
if $p_0$ and $q_0$ are mutually opposite $(1,\delta)$-suspenders in $Z$,
then $\vert d_Z(p_0,q_0) - \pi \vert < \delta$.
\end{rem}

We also have the following
(cf.~Lemma \ref{lem: p1susp}):

\begin{lem}\label{lem: p1dsusp}
Let $Z$ be a $\CAT(1)$ space.
If there exists mutually opposite $(1,\delta)$-suspenders
$p_0$ and $q_0$ in $Z$,
then the following hold: 
\begin{enumerate}
\item
$\diam \Ant(p_0) < 2\delta$ and $\diam \Ant(q_0) < 2\delta$;
\item
for all $z_1, z_2 \in Z$ with $d_Z(z_1,z_2) \ge \pi$
we have
\begin{equation}
d_Z(z_1,p_0) + d_Z(p_0,z_2) < \pi + 2\delta,
\quad
d_Z(z_1,q_0) + d_Z(q_0,z_2) < \pi + 2\delta;
\label{eqn: p1dsuspa}
\end{equation}
in particular,
we have $\diam Z \le \pi + 2\delta$. 
\end{enumerate}
\end{lem}

\begin{proof}
For every $\bar{p}_0 \in \Ant(p_0)$,
by \eqref{eqn: dsph},
we have $d_Z(\bar{p}_0,q_0) < \delta$.
Hence $\diam \Ant(p_0) < 2\delta$.
Similarly, 
we see $\diam \Ant(q_0) < 2\delta$

Take $z_1, z_2 \in Z$ with $d_Z(z_1,z_2) \ge \pi$.
By \eqref{eqn: dsph},
we have
\[
d_Z(p_0,z_1) + d_Z(z_1,q_0) < \pi + \delta,
\quad
d_Z(p_0,z_2) + d_Z(z_2,q_0) < \pi + \delta.
\]
These inequalities lead to \eqref{eqn: p1dsuspa}.
\end{proof}

\begin{rem}\label{rem: aftera1dsusp}
By Lemma \ref{lem: p1dsusp},
we see that
if a $\CAT(1)$ space
$Z$ admits mutually opposite $(1,\delta)$-suspenders $p_0$ and $q_0$,
then either $Z$ is a geodesic space
of $\diam Z \le \pi + 2\delta$,
or 
$Z$ has two geodesic-connected components $Z_+$ and $Z_-$
with $p_0 \in Z_+$, $q_0 \in Z_-$, and 
$\diam Z_{\pm} < 2\delta$.
If in addition $Z$ is geodesically complete,
then either $Z$ is geodesic or 
it is homeomorphic to $\Sph^0$.
\end{rem}

Assume that $Z$ is geodesically complete.
If a point in $Z$ has 
antipodal set of a sufficiently small diameter,
then it is a $(1,\delta)$-suspender.
More precisely, we see the following
(cf.~Remark \ref{rem: afterp1susp}):

\begin{lem}\label{lem: e1dsusp}
Let $Z$ be a geodesically complete $\CAT(1)$ space.
If a point $p$ in $Z$ satisfies $\diam \Ant(p) < \delta$,
then 
for every $q \in \Ant(p)$ we have
\begin{equation}
\sup_{z \in Z} \left\{ \, d_Z(p,z) + d_Z(z,q) \, \right\}< \diam Z + \delta;
\label{eqn: e1dsuspa}
\end{equation}
in particular,
if in addition $\diam Z = \pi$,
then $p$ and $q$ are mutually opposite $(1,\delta)$-suspenders in $Z$.
\end{lem}

\begin{proof}
Let $z \in Z$.
If $z = p$ or $z \in \Ant(p)$,
then we have \eqref{eqn: e1dsuspa}.
Assume $z \in U_{\pi}^{\ast}(p)$.
Since $Z$ is geodesically complete,
there exists a geodesic $pq_0$ of length $\pi$
through $z$ for some $q_0 \in \Ant(p)$.
Then we have
\[
d_Z(p,z) + d_Z(z,q) 
\le \pi + d_Z(q,q_0) < \pi + \delta \le d_Z(p,q) + \delta,
\]
and hence \eqref{eqn: e1dsuspa}.
\end{proof}

\subsection{Relaxed multi-fold suspenders in CAT(1) spaces}

Let $Z$ be a $\CAT(1)$ space with metric $d_Z$.
Let $m \in \N$ and $\delta \in (0,\pi)$.
We say that an $m$-tuple $(p_1,\dots,p_m)$ of points in $Z$ is
an \emph{$(m,\delta)$-suspender} 
if there exists another $m$-tuple $(q_1,\dots,q_m)$ of points in $Z$
satisfying the following:
\begin{enumerate}
\item
for all $i \in \{ 1, \dots, m \}$,
the points $p_i$ and $q_i$ 
are mutually opposite $(1,\delta)$-suspenders in $Z$;
\item
for all distinct $i, j \in \{ 1, \dots, m \}$,
we have
\begin{equation}
d_Z(p_i,p_j) < \frac{\pi}{2} + \delta,
\quad
d_Z(p_i,q_j) < \frac{\pi}{2} + \delta,
\quad
d_Z(q_i,q_j) < \frac{\pi}{2} + \delta.
\label{eqn: dortho}
\end{equation}

\end{enumerate}
In this case, we say that $(p_1,\dots,p_m)$ and 
$(q_1,\dots,q_m)$ are \emph{opposite} to each other.

\begin{rem}\label{rem: aftermdsusp}
If $Z$ admits a $(2,\delta)$-suspender,
then it is geodesic,
provided $\delta$ is small enough
(see Remark \ref{rem: aftera1dsusp}).
\end{rem}

For penetrable $\CAT(1)$ spaces,
the notion of the $(m,\delta)$-suspenders 
is a perturbed one of the rigid $m$-suspenders.

\begin{lem}\label{lem: omdsusp}
Let $Z$ be a penetrable $\CAT(1)$ space.
Let $(p_1,\dots,p_m)$ and $(q_1,\dots,q_m)$ be mutually opposite 
$(m,\delta)$-suspenders in $Z$.
Then 
\begin{enumerate}
\item
for all $i \in \{ 1, \dots, m \}$
we have
\begin{equation}
\inf_{z \in Z} \left\{ \, d_Z(p_i,z) + d_Z(z,q_i) \, \right\} > \pi - \delta;
\label{eqn: omdsuspa}
\end{equation}
\item
for all distinct $i, j \in \{ 1, \dots, m \}$, 
\begin{equation}
d_Z(p_i,p_j) > \frac{\pi}{2} - 2\delta,
\quad
d_Z(p_i,q_j) > \frac{\pi}{2} - 2\delta,
\quad
d_Z(q_i,q_j) > \frac{\pi}{2} - 2\delta.
\label{eqn: omdsuspb}
\end{equation}
\end{enumerate}
\end{lem}

\begin{proof}
As shown in Lemma \ref{lem: o1dsusp},
we already know \eqref{eqn: omdsuspa}.
By \eqref{eqn: omdsuspa},
for all distinct $i, j \in \{ 1, \dots, m \}$
we have
\[
d_Z(p_i,p_j) \ge d_Z(p_i,q_i) - d_Z(p_j,q_i)
> \pi - \delta - d_Z(p_j,q_i)
> \frac{\pi}{2} - 2\delta.
\]
Similarly,
we have 
$d_Z(p_i,q_j) > \pi/2 - 2\delta$
and $d_Z(q_i,q_j) > \pi/2 - 2\delta$.
Hence we have \eqref{eqn: omdsuspb}.
\end{proof}

In the sequel, 
we will need the following:

\begin{lem}\label{lem: amdsusp}
Let $Z$ be a $\CAT(1)$ space.
Assume that $(p_1,\dots,p_m)$ and $(q_1,\dots,q_m)$ are
mutually opposite $(m,\delta)$-suspenders in $Z$.
For each $i \in \{ 1, \dots, m \}$,
take $\bar{p}_i \in \Ant(p_i)$.
Then $(p_1,\dots,p_m)$ and $(\bar{p}_1,\dots,\bar{p}_m)$ are
mutually opposite $(m,2\delta)$-suspenders in $Z$.
\end{lem}

\begin{proof}
By Lemma \ref{lem: p1dsusp},
for all $i \in \{ 1, \dots, m \}$
we know that $\diam \Ant(p_i)$ is smaller than $2\delta$.
From this property we derive the claim.
\end{proof}

\subsection{Stability of relaxed multi-fold suspenders}

We first state the following stability under ultralimits:

\begin{lem}\label{lem: umdstab}
Let $(Z_{\omega},z_{\omega})$ 
be the ultralimit $\ulim(Z_k,z_k)$ of a sequence of 
pointed $\CAT(1)$ spaces.
Let $(p_1,\dots,p_m)$ and $(q_1,\dots,q_m)$ be 
mutually opposite $(m,\delta)$-suspenders in $Z_{\omega}$.
For each $i \in \{ 1, \dots, m \}$,
take $p_{i,k}, q_{i,k} \in Z_k$
with $p_{i,k} \to p_i$, $q_{i,k} \to q_i$.
Then for all $\omega$-large $k$
the pair of $(p_{1,k},\dots,p_{m,k})$ and $(q_{1,k},\dots,q_{m,k})$ are
mutually opposite $(m,\delta)$-suspenders in $Z_k$.
\end{lem}

\begin{proof}
If $(p_1,\dots,p_m)$ and $(q_1,\dots,q_m)$ are
mutually opposite $(m,\delta)$-suspenders in $Z_{\omega}$,
then for some sufficiently small $a \in (0,\delta)$
they are mutually opposite $(m,\delta-a)$-suspenders in $Z_{\omega}$.
This property leads to the lemma.
\end{proof}

We next claim that
every surjective $1$-Lipschitz map
between $\CAT(1)$ spaces
sends $(m,\delta)$-suspenders
to $(m,\delta)$-ones.

\begin{lem}\label{lem: slstab}
Let $\varphi \colon Z_1 \to Z_2$ 
be a surjective $1$-Lipschitz map
between $\CAT(1)$ spaces.
Let $(p_1,\dots,p_m)$ and $(q_1,\dots,q_m)$ be 
mutually opposite $(m,\delta)$-suspenders in $Z_1$.
Then the pair of $(\varphi(p_1),\dots,\varphi(p_m))$ and 
$(\varphi(q_1),\dots,\varphi(q_m))$
are mutually opposite $(m,\delta)$-suspenders in $Z_2$.
\end{lem}

\begin{proof}
Since $f$ is surjective and $1$-Lipschitz,
we see the lemma.
\end{proof}

\begin{rem}\label{rem: afterslstab}
In order to obtain Lemmas \ref{lem: umdstab} and \ref{lem: slstab},
we formulate the notions of $\delta$-suspenders and
$(m,\delta)$-suspenders
by the one-sided open 
inequalities \eqref{eqn: dsph} and \eqref{eqn: dortho}.
\end{rem}

\subsection{Progress of relaxed multi-fold suspenders}

The next lemma tells us whether 
a relaxed $m$-fold suspender
can be extended to a relaxed $(m+1)$-fold one
(cf.~Lemma \ref{lem: pm0susp}).

\begin{lem}\label{lem: pmdsusp}
For every $\epsilon \in (0,1)$,
and for every $m \in \N$,
there exists $\delta \in (0,1)$ 
satisfying the following:
Let $Z$ be a penetrable $\CAT(1)$ space with metric $d_Z$,
and
let $(p_1,\dots,p_m)$ be an $(m,\delta)$-suspender in $Z$.
If there exists a $(1,\delta)$-suspender $p_0$ in $Z$
such that
\begin{equation}
\sum_{i=1}^m \cos^2 d_Z \left( p_i, p_0 \right) < 1-\epsilon,
\label{eqn: pmdsuspa}
\end{equation}
then there exists $p_{m+1} \in Z$
for which the $(m+1)$-tuple
$(p_1,\dots,p_m,p_{m+1})$
is an $(m+1,\epsilon)$-suspender in $Z$.
\end{lem}

\begin{proof}
Suppose that for some sequence $(\delta_k)$ in $(0,\epsilon)$
with $\delta_k \to 0$, 
there exists a sequence $(Z_k)$
of penetrable $\CAT(1)$ spaces with metric $d_{Z_k}$
such that each $Z_k$
admits an $(m,\delta_k)$-suspender 
$(p_{1,k}, \dots, p_{m,k})$,
and a $(1,\delta_k)$-suspender $p_{0,k}$
with
\[
\sum_{i=1}^m \cos^2 d_{Z_k} \left( p_{i,k}, p_{0,k} \right) < 1-\epsilon.
\]
Suppose in addition that
there exists no point $p_k$ in $Z_k$ 
for which $(p_{1,k}, \dots, p_{m,k},p_k)$ is
an $(m+1,\epsilon)$-suspender in $Z_k$.
Let $(Z_{\omega},z_{\omega})$
be the ultralimit $\ulim(Z_k,z_k)$ 
for some $z_k \in Z_k$.
For each $i \in \{ 1, \dots, m \}$,
take $p_i \in Z_{\omega}$
with $p_{i,k} \to p_i$.
Lemma \ref{lem: omdsusp} implies that
$(p_1,\dots,p_m)$ is an $m$-suspender in $Z_{\omega}$.
From the present assumption we derive
\[
\sum_{i=1}^m \cos^2 d_{Z_{\omega}} \left( p_i, p_0 \right) \le 1-\epsilon.
\]
By Lemma \ref{lem: pm0susp},
there exists $p_{m+1} \in Z_{\omega}$
such that $(p_1,\dots,p_m,p_{m+1})$
is an $(m+1)$-suspender in $Z_{\omega}$.
For each $k$,
we choose $p_{m+1,k} \in Z_k$
with $p_{m+1,k} \to p_{m+1}$.
From Lemma \ref{lem: umdstab},
it follows that for all $\omega$-large $k$
the $(m+1)$-tuple $(p_{1,k},\dots,p_{m,k},p_{m+1,k})$ 
is an $(m+1,\epsilon)$-suspender in $Z_k$.
This is a contradiction.
\end{proof}

\subsection{Dimensions versus relaxed multi-fold suspenders}

First we show the following:

\begin{lem}\label{lem: capmdsusp}
For every $n \in \N$
there exists $\delta \in (0,1)$
such that
if a penetrable $\CAT(1)$ space $Z$ 
of $\dim_{\mathrm{G}} Z \le n-1$
admits an $(m,\delta)$-suspender,
then $m \le n$. 
\end{lem}

\begin{proof}
Suppose that
for some sequences $(\delta_k)$ in $(0,1)$ with $\delta_k \to 0$
and $(m_k)$ in $\N$,
there exists a sequence $(Z_k)$ of penetrable
$\CAT(1)$ spaces of $\dim_G Z_k \le n-1$
such that
each $Z_k$ admits an 
$(m_k,\delta_k)$-suspender
for $m_k$ with $m_k \ge n+1$.
Let $(Z_{\omega},z_{\omega})$ be the ultralimit $\ulim (Z_k,z_k)$ 
for some $z_k \in Z_k$.
From \eqref{eqn: lscdim}
it follows that
$\dim_{\mathrm{G}} Z_{\omega} \le n-1$.
Set $m := \ulim m_k$.
Then $m \ge n+1$.
The limit $m$ is possibly infinite.
We may assume that $m$ is finite.
Lemma \ref{lem: omdsusp} implies that
$Z_{\omega}$ admits an $m$-suspender.
By Proposition \ref{prop: m0susp}, 
the ultralimit
$Z_{\omega}$ isometrically splits as $\Sph^{m-1} \ast Z_0$
for some $Z_0$.
This implies 
$\dim_{\mathrm{G}} Z_{\omega} \ge n$.
This is a contradiction.
\end{proof}

\begin{rem}\label{rem: aftercapmdsusp}
Similarly to Lemma \ref{lem: capmdsusp},
we see that
for every $m \in \N$
there exists $\delta \in (0,1)$
such that
if a penetrable $\CAT(1)$ space $Z$ 
admits an $(m,\delta)$-suspender,
then $m \le \dim_{\mathrm{G}} Z + 1$.
\end{rem}

In the sequel, we will use the following:

\begin{lem}\label{lem: fullsusp}
For every $n \in \N$, and for every $c \in [1,\infty)$,
there exists $\delta \in (0,1)$
such that
if a penetrable $\CAT(1)$ space $Z$ 
of $\dim_{\mathrm{G}} Z \le n-1$
admits an $(n,\delta)$-suspender $(p_1,\dots,p_n)$,
then no point $z$ in $Z$
satisfies 
\[
\max_{i \in \{ 1, \dots, n \}} 
\left\vert d_Z \left( p_i, z \right) - \frac{\pi}{2} \right\vert < c\delta.
\]
\end{lem}

Lemma \ref{lem: fullsusp} can be shown 
by an ultralimit argument combined with
Proposition \ref{prop: m0susp}
and Lemma \ref{lem: omdsusp}.

\subsection{Almost spherical suspension structure}

If a $\CAT(1)$ space $Z$
isometrically splits as $\Sph^{m-1} \ast Z_0$,
then
we find mutually opposite $m$-suspenders in $Z$
contained in the spherical factor $\Sph^{m-1}$
(see Example \ref{exmp: exmsusp}).
Moreover,
if for some $m \in \N$ a $\CAT(1)$ space $Z$ satisfies
\[
d_{\GH}(Z,\Sph^{m-1} \ast Z_0) < \delta
\]
for some metric space $Z_0$,
then we find
mutually opposite $(m,10\delta)$-suspenders in $Z$.

In the following two propositions,
from the existence of relaxed multi-fold suspenders,
we derive an almost suspension structure
for penetrable $\CAT(1)$ spaces
(cf.~\cite[Subsections 6.3 and 13.2]{lytchak-nagano1}).

Let $\Class$ denote the set of all isometry classes
of compact penetrable $\CAT(1)$ spaces.
If a $\CAT(1)$ space $X \in \Class$ 
isometrically splits as a spherical join $Y \ast Z$,
then both the factors $Y$ and $Z$ belong to $\Class$.

We first state the following: 

\begin{prop}\label{prop: tmdsusp}
Let $\Class_0$
be a relatively compact subset of $\Class$
with respect to the Gromov--Hausdorff topology.
Then for every $\epsilon \in (0,1)$
there exists $\delta \in (0,1)$
such that 
if $Z \in \Class_0$ admits an $(m,\delta)$-suspender,
then there exists $Z_0 \in \Class$
satisfying
\[
d_{\GH}(Z, \Sph^{m-1} \ast Z) < \epsilon.
\]
\end{prop}

For $n \in \N$,
let $\Class(n-1)$ denote the set of all isometry classes
of compact penetrable $\CAT(1)$ spaces 
of covering dimension $\le n-1$.
If a $\CAT(1)$ space $X \in \Class(n-1)$ 
isometrically splits as an $m$-fold spherical suspension 
$\Sph^{m-1} \ast Z_0$,
then $Z_0$ belong to $\Class(n-m-1)$.

Proposition \ref{prop: tmdsusp}
can be shown 
by an ultralimit argument combined with
Proposition \ref{prop: m0susp}
and Lemma \ref{lem: omdsusp}.
The details are left to the readers.

Similarly to Proposition \ref{prop: tmdsusp},
by Lemma \ref{lem: capmdsusp} we see the following:

\begin{prop}\label{prop: tndsusp}
For $n \in \N$,
let $\Class_0$
be a relatively compact subset of $\Class(n-1)$
with respect to the Gromov--Hausdorff topology.
Then for every $\epsilon \in (0,1)$
there exists $\delta \in (0,1)$
such that 
if $Z \in \Class_0$ 
admits an $(m,\delta)$-suspender,
then $m \le n$, and 
there exists $Z_0 \in \Class(n-m-1)$
satisfying
\[
d_{\GH} \left( Z, \Sph^{m-1} \ast Z_0 \right) < \epsilon;
\]
in particular,
if $\epsilon$ is sufficiently small with respect to $n$,
then $Z$ is homotopy equivalent to $\Sph^{m-1} \ast Z_0$.
\end{prop}

The homotopy equivalence stated above
follows from 
the Petersen homotopic stability theorem \cite[Theorem A]{petersen}
for $\LGC$ spaces.

\begin{rem}\label{rem: aftertndsusp}
In Proposition \ref{prop: tndsusp},
if $m=n$, then $Z_0$ must be empty,
hence $d_{\GH} \left( Z, \Sph^{n-1} \right) < \epsilon$.
If $m=n-1$, then $Z_0$ is a finite discrete metric space
consisting of at least two points.
\end{rem}

\section{Local geometric regularity}

We study the local geometric structure of locally geodesically complete
open subsets of $\CAT(\kappa)$ spaces.
We refer to \cite[Sections 5 and 7]{lytchak-nagano1}
for the same studies for $\GCBA$ spaces.

\subsection{Upper semi-continuity of spaces of directions}
\label{ssec: usc}

We mention the upper semi-continuity of spaces of directions,
already shown in \cite[Lemmas 13.1 and 13.2]{lytchak2} 
for geodesically complete spaces.

Let $X$ be a $\CAT(\kappa)$ space,
and let $U$ be a locally geodesically complete open subset of $X$.
For a point $x$ in $U$,
let $(x_k)$ be a sequence in $U$ converging to $x$.
Let $(X_{\omega},x_{\omega})$ be the ultralimit $\ulim (X,x_k)$
of the pointe metric spaces.
Let $\left( \ulim \Sigma_{x_k}X, \xi_{\omega} \right)$
the ultralimit of some sequence of the pointed spaces of directions
$\left( \Sigma_{x_k}X, \xi_k \right)$.
We define a surjective $1$-Lipschitz map 
$\varphi_{x_{\omega}} \colon 
\Sigma_{x_{\omega}}X_{\omega} \to \ulim \Sigma_{x_k}X$ 
in the following way:
For each $\eta_{\omega} \in \Sigma_{x_{\omega}}'X_{\omega}$,
pick $y_{\omega} \in U_{D_{\kappa}}(x_{\omega}) - \{x_{\omega}\}$ 
with $\eta_{\omega} = (y_{\omega})_{x_{\omega}}'$,
and take $y_k \in U_{D_{\kappa}}(x_k) - \{x_k\}$
with $y_k \to y$.
Let $\varphi_{x_{\omega}}'(\eta_{\omega})$ 
be the point in $\ulim \Sigma_{x_k}X$
with $(y_k)_{x_k}' \to \varphi_{x_{\omega}}'(\eta_{\omega})$.
Thus we find a well-defined map
$\varphi_{x_{\omega}}' \colon 
\Sigma_{x_{\omega}}X_{\omega}' \to \ulim \Sigma_{x_k}X$.
Since each open ball 
$U_r(x_k)$ is locally geodesically complete,
$\varphi_{x_{\omega}}'$ is surjective;
moreover,
$U_r(x)$ is locally geodesically complete too.
From \eqref{eqn: uscangle} it follows that
$\varphi_{x_{\omega}}'$ is $1$-Lipschitz.
Extending the map $\varphi_{x_{\omega}}'$ on
$\Sigma_{x_{\omega}}'X_{\omega}$
to the completion $\Sigma_{x_{\omega}}X_{\omega}$,
we obtain a surjective $1$-Lipschitz map 
$\varphi_{x_{\omega}} \colon 
\Sigma_{x_{\omega}}X_{\omega} \to \ulim \Sigma_{x_k}X$.

In summary,
we have the following upper semi-continuity:

\begin{lem}\label{lem: uusclink}
Let $U$ be a locally geodesically complete open subset 
of a $\CAT(\kappa)$ space.
For $x \in U$,
let $(x_k)$ be a sequence in $U$ converging to $x$.
Then
there exists a surjective $1$-Lipschitz map 
$\varphi_x \colon \Sigma_xX \to \ulim \Sigma_{x_k}X$ such that
for every point $y$ in $U_{D_{\kappa}}(x)-\{x\}$,
and for every sequence $(y_k)$ in $U_{D_{\kappa}}(x_k)-\{x_k\}$
with $y_k \to y$,
we have
$(y_k)_{x_k}' \to \varphi_x(y_x')$. 
\end{lem}

\begin{rem}
\label{rem: auusclink}
Let $X$ be a $\CAT(\kappa)$ space.
Let $(x_k)$ be a sequence in $X$ converging to some $x \in X$.
Let $(X_{\omega},x_{\omega})$ be the ultralimit $\ulim (X,x_k)$.
The map
$\varphi_{\omega} \colon X \to X_{\omega}$
given by $\varphi_{\omega}(y) := [y]$
is an isometric embedding 
satisfying $\varphi_{\omega}(x) = x_{\omega}$.
In general, 
the map $\varphi_{\omega}$ is not surjective.
\end{rem}

\subsection{Infinitesimal geometric regularity}

Recall that an $m$-regular point in a $\GCBA$ space
means a point at which the space of directions isometrically splits as
an $m$-fold spherical suspension.
An $m$-singular point means a non-$m$-regular point.
We now introduce a relaxed infinitesimal geometric regularity, 
briefly mentioned in Section \ref{sec: i}.

\begin{defn}\label{defn: greg}
Let $m \in \N$ and $\delta \in (0,1)$.
Let $U$ be a locally geodesically complete open subset
of a $\CBA$ space $X$.
We say that a point $x$ in $U$ is 
\emph{$(m,\delta)$-regular}
if there exists an $(m,\delta)$-suspender in $\Sigma_xX$.
A point in $U$ is \emph{$(m,\delta)$-singular}
if it is not $(m,\delta)$-regular.
For a subset $A$ of $U$,
we denote by $R_{m,\delta}(A)$
the set of all $(m,\delta)$-regular points in $A$,
and by $S_{m,\delta}(A)$
the set of all $(m,\delta)$-singular points.
We call $R_{m,\delta}(A)$
the \emph{$(m,\delta)$-regular set in $A$},
and
$S_{m,\delta}(A)$
the \emph{$(m,\delta)$-singular set in $A$}.
\end{defn}

Let $m \in \N$ and $\delta \in (0,1)$.
Let $U$ be a locally geodesically complete open subset
of a $\CAT(\kappa)$ space,
and let $A$ be a subset of $U$.
Then
we have $R_m(A) \subset R_{m,\delta}(A)$;
in particular,
$S_{m,\delta}(A) \subset S_m(A)$.
For every $n \in \N$ with $m \le n$,
we have $R_{n,\delta}(A) \subset R_{m,\delta}(A)$.
For every $\epsilon \in (0,1)$ 
with $\delta \le \epsilon$,
we have $R_{m,\delta}(A) \subset R_{m,\epsilon}(A)$.

From the upper semi-continuity of spaces of directions,
we deduce:

\begin{lem}\label{lem: mdopen}
Let $U$ be a locally geodesically complete open subset
of a $\CAT(\kappa)$ space $X$,
and $A$ a subset of $U$.
Then $R_{m,\delta}(A)$ is open in $A$;
in particular,
$S_{m,\delta}(A)$ is closed in $A$.
\end{lem}

\begin{proof}
Suppose that for some $x \in R_{m,\delta}(A)$
there exists a sequence $(x_k)$ in $S_{m,\delta}(A)$ 
with $x_k \to x$.
Then we find an $(m,\delta)$-suspender
$(\xi_1,\dots,\xi_m)$ in $\Sigma_xX$.
By Lemma \ref{lem: uusclink},
the $m$-tuple $\left( \varphi_x(\xi_1), \dots, \varphi_x(\xi_m) \right)$
in $\ulim \Sigma_{x_k}X$
is an $(m,\delta)$-suspender.
Hence for $\omega$-large $k$
we find an $(m,\delta)$-suspender 
$\left( \xi_{1,k}, \dots, \xi_{m,k} \right)$
in $\Sigma_{x_k}X$.
This is a contradiction.
\end{proof}

Lemma \ref{lem: capmdsusp} implies the following:

\begin{lem}\label{lem: ndr}
For every $n \in \N$,
there exists $\delta \in (0,1)$
such that
if a locally geodesically complete open subset $U$
of a $\CAT(\kappa)$ space $X$
satisfies $\dim_{\mathrm{G}} U \le n$,
then the set $R_{n+1,\delta}(U)$ is empty;
in particular,
$S_{n+1,\delta}(U)$ coincides with $U$.
\end{lem}

\begin{proof}
For $n \in \N$,
let $\delta \in (0,1)$ be small enough.
Suppose that
for some $x \in U$
the space $\Sigma_xX$ admits an $(n+1,\delta)$-suspender.
On the other hand,
we have $\dim_{\mathrm{G}} \Sigma_xX \le n-1$.
This contradicts Lemma \ref{lem: capmdsusp}.
\end{proof}

The following 
is substantially shown in \cite[Proposition 7.3]{lytchak-nagano1}.

\begin{prop}\label{prop: 1dr}
Let $\delta \in (0,1)$.
Let $U$ be a locally compact, locally geodesically complete 
open subset of a $\CAT(\kappa)$ space.
Then for every $p \in X$
there exists $r \in (0,D_{\kappa}/2)$ with $U_r(x) \subset U$
such that for every $x \in U_r(p) - \{p\}$
the direction $p_x'$ is a $(1,\delta)$-suspender in $\Sigma_xX$.
Moreover,
for every compact subset $B$ of $U$,
the set $S_{1,\delta}(B)$ has at most finitely many points.
\end{prop}

\subsection{Strainers}

Based on \cite[Section 7]{lytchak-nagano1},
we discuss the notions of strainers 
for $\CAT(\kappa)$ spaces
in more general setting.

\begin{defn}\label{defn: mdstr}
Let $m \in \N$ and $\delta \in (0,1)$.
Let $X$ be a $\CAT(\kappa)$ space with metric $d_X$.
For a point $x$ in $X$,
we say that
an $m$-tuple $(p_1,\dots,p_m)$ of points in $X - \{x\}$
is an 
\emph{$(m,\delta)$-strainer at $x$}
if the following hold:
\begin{enumerate}
\item
for each $i \in \{ 1, \dots, m \}$ 
we have $d_X(p_i,x) < D_{\kappa}/2$;
\item
the $m$-tuple 
$\left( (p_1)_x', \dots, (p_m)_x' \right)$
is an $(m,\delta)$-suspender in $\Sigma_xX$;
\end{enumerate}
in this case,
$x$ is said to be
\emph{$(m,\delta)$-strained by $(p_1,\dots,p_m)$}.
Two $(m,\delta)$-strainers
$(p_1,\dots,p_m)$ and $(q_1,\dots,q_m)$ at $x$
are \emph{opposite} 
if the $(m,\delta)$-suspenders
$\left( (p_1)_x', \dots, (p_m)_x' \right)$
and 
$\left( (q_1)_x', \dots, (q_m)_x' \right)$
are opposite to each other in $\Sigma_xX$.

For a subset $A$ of $X$ with $\diam A < D_{\kappa}/2$,
we say that
an $m$-tuple $(p_1,\dots,p_m)$ of points in $A$
is an 
\emph{$(m,\delta)$-strainer at $A$}
if 
\begin{enumerate}
\item
for each $i \in \{ 1, \dots, m \}$ we have
$\sup_{x \in A} d_X(p_i,x) < D_{\kappa}/2$;
\item
the $m$-tuple $(p_1,\dots,p_m)$ is an $(m,\delta)$-strainer at 
all $x \in A$;
\end{enumerate}
in this case,
 $A$ is said to be
\emph{$(m,\delta)$-strained by $(p_1,\dots,p_m)$}.
Two $(m,\delta)$-strainers $(p_1,\dots,p_m)$ and $(q_1,\dots,q_m)$
at $A$
are \emph{opposite} 
if they are opposite to each other at all $x \in A$.
\end{defn}

\begin{rem}
Let $U$ be a locally geodesically complete open subset 
of a $\CAT(\kappa)$ space.
A point $x$ in $U$ is $(m,\delta)$-regular
if and only if $\{x\}$ is $(m,\delta)$-strained by some
$(m,\delta)$-strainer.
\end{rem}

\begin{rem}\label{rem: strdim}
Let $m \in \N$ and $\delta \in (0,1)$.
Let $U$ be a locally geodesically complete open subset
of a $\CAT(\kappa)$ space $X$.
If $\delta$ is sufficiently small,
then for every $(m,\delta)$-strained point $x$ in $U$ we have
$\dim_G T_xX \ge m$
(see Remark \ref{rem: aftercapmdsusp}).
If $X$ is $\GCBA$,
and if $4m\delta < 1$,
then for every $(m,\delta)$-strained point $x$ in $U$ we have
$\dim T_xX \ge m$
(\cite[Lemma 11.7]{lytchak-nagano1}).
\end{rem}

\subsection{Stability of strainers}

From the stability of suspenders
and the upper semi-continuity
of spaces of directions,
we derive the following:

\begin{lem}\label{lem: usstab0}
Let $U$ be a locally geodesically complete open subset 
of a $\CAT(\kappa)$ space $X$.
For $x \in U$,
let $(x_k)$ be a sequence in $U$ with $x_k \to x$.
Let $(p_1,\dots,p_m)$ and $(q_1,\dots,q_m)$
are mutually opposite $(m,\delta)$-strainers at $x$.
For each $i \in \{ 1, \dots, m \}$,
choose some sequences $(p_{i,k})$ and $(q_{i,k})$
in $X$ with $p_{i,k} \to p_i$ and $q_{i,k} \to q_i$,
respectively.
Then for all sufficiently large $k$
the $m$-tuples 
$(p_{1,k},\dots,p_{m,k})$ and $(q_{1,k},\dots,q_{m,k})$ 
are mutually opposite $(m,\delta)$-strainers at $x_k$.
\end{lem}

\begin{proof}
The $m$-tuples $\left( (p_1)_x', \dots, (p_m)_x' \right)$ and
$\left( (q_1)_x', \dots, (q_m)_x' \right)$
are mutually opposite $(m,\delta)$-suspenders in $\Sigma_xX$.
By Lemma \ref{lem: uusclink},
we find a surjective $1$-Lipschitz map 
$\varphi_x$ from $\Sigma_xX$ to the ultralimit $\ulim \Sigma_{x_k}X$ 
so that
$(p_{i,k})_{x_k}' \to \varphi_x((p_i)_x')$ and
$(q_{i,k})_{x_k}' \to \varphi_x((q_i)_x')$
for all $i \in \{ 1, \dots, m \}$. 
Since $\varphi_x$ is surjective and $1$-Lipschitz,
by Lemma \ref{lem: slstab} we conclude that
$\left( \varphi_x((p_1)_x'), \dots, \varphi_x((p_m)_x') \right)$
and
$\left( \varphi_x((q_1)_x'), \dots, \varphi_x((q_m)_x') \right)$
are mutually opposite $(m,\delta)$-suspenders in $\ulim \Sigma_{x_k}X$.
By Lemma \ref{lem: umdstab},
for all $\omega$-large $k$
the $m$-tuples
$(p_{1,k},\dots,p_{m,k})$
and 
$(q_{1,k},\dots,q_{m,k})$
are mutually opposite $(m,\delta)$-strainers at $x_k$.
This leads to the lemma.
\end{proof}

Similarly to Lemma \ref{lem: usstab0},
via a generalization of Lemma \ref{lem: uusclink}
based on Lemmas \ref{lem: umdstab}
and \ref{lem: slstab},
we see the following:

\begin{lem}\label{lem: usstab}
For $r \in (0,D_{\kappa})$,
let $(X_{\omega},x_{\omega})$ 
be the ultralimit 
of a sequence $(X_k,x_k)$ of $\CAT(\kappa)$ spaces
such that each open ball
$U_r(x_k)$ is locally geodesically complete.
Let $(p_1,\dots,p_m)$ and $(q_1,\dots,q_m)$
are mutually opposite $(m,\delta)$-strainers at $x_{\omega}$.
For each $i \in \{ 1, \dots, m \}$,
choose some sequences $(p_{i,k})$ and $(q_{i,k})$
in $X_k$ with $p_{i,k} \to p_i$ and $q_{i,k} \to q_i$,
respectively.
Then for all $\omega$-large $k$
the $m$-tuples 
$(p_{1,k},\dots,p_{m,k})$ and $(q_{1,k},\dots,q_{m,k})$ 
are mutually opposite $(m,\delta)$-strainers at $x_k$.
\end{lem}

\subsection{Straining radii}

We discuss the notion of straining radii
in our setting
(cf.~\cite[Subsection 7.5]{lytchak-nagano1}).

\begin{lem}\label{lem: sr0}
Let $U$ be a locally geodesically complete open subset 
of a $\CAT(\kappa)$ space $X$ with metric $d_X$.
If a point $x_0$ in $U$
is $(m,\delta)$-strained by
$(p_1,\dots,p_m)$,
then
there exists $r_0 \in (0,D_{\kappa}/2)$ with $U_{2r_0}(x_0) \subset U$ 
such that
for every $x \in U_{r_0}(x_0)$,
and for any $q_i \in U$ satisfying $x \in p_iq_i$ and $d_X(q_i,x) = r_0$,
$i \in \{ 1, \dots, m \}$,
the $m$-tuples
$(p_1,\dots,p_m)$ and $(q_1,\dots,q_m)$
are mutually opposite $(m,2\delta)$-strainers at $x$.
\end{lem}

\begin{proof}
By Lemma \ref{lem: usstab0},
we find $r_0 \in (0,D_{\kappa}/2)$ with $U_{2r_0}(x_0) \subset U$
such that
for every $x \in U_{r_0}(x_0)$
the $m$-tuple $(p_1,\dots,p_m)$ is an $(m,\delta)$-strainer at $x$.
For each $i \in \{ 1, \dots, m \}$,
take a point $q_i$ in $U$ satisfying $x \in p_iq_i$ and $d_X(q_i,x) = r_0$.
Lemma \ref{lem: amdsusp} implies that
$(p_1,\dots,p_m)$ and $(q_1,\dots,q_m)$
are mutually opposite $(m,2\delta)$-strainers at $x$.
\end{proof}

Based on Lemma \ref{lem: sr0},
we formulate the following:

\begin{defn}\label{defn: sr}
Let $U$ be a locally geodesically complete open subset 
of a $\CAT(\kappa)$ space $X$ with metric $d_X$.
Assume that a point $x$ in $U$
is $(m,\delta)$-strained by
$(p_1,\dots,p_m)$.
We denote by $\sigma_{(p_1,\dots,p_m)}(x)$ the 
supremum of $r \in (0,D_{\kappa}/2)$
such that
$U_{2r}(x)$ is contained in $U$,
and for every $y \in U_r(x)$,
and for any $q_i \in U$ satisfying $y \in p_iq_i$ and $d_X(q_i,y) = r$,
$i \in \{ 1, \dots, m \}$,
the $m$-tuples
$(p_1,\dots,p_m)$ and $(q_1,\dots,q_m)$
are mutually opposite $(m,2\delta)$-strainers at $y$.
We call $\sigma_{(p_1,\dots,p_m)}(x)$
the 
\emph{strained radius of $x$ with respect to $(p_1,\dots,p_m)$}.
We notice that
by Lemma \ref{lem: sr0}
the value $\sigma_{(p_1,\dots,p_m)}(x)$ is certainly positive.
\end{defn}

\begin{rem}
\label{rem: asr}
In Definition \ref{defn: sr},
we see that
the function $\sigma_{(p_1,\dots,p_m)}$ on $U_{r_0}(x_0)$
is continuous.
\end{rem}

\subsection{Angles around strained points}

Around strained points,
we find regular triangles 
whose base angles are almost equal to
the comparison ones.
We restate \cite[Lemma 7.6]{lytchak-nagano1}
in the following form:

\begin{lem}\label{lem: jl}
For a subset $A$ of a $\CAT(\kappa)$ space 
with $\diam A < D_{\kappa}/2$,
let $(p_1,\dots,p_m)$ and $(q_1,\dots,q_m)$
be mutually opposite $(m,\delta)$-stainers at $A$.
Take $x, y \in A$ with $x \neq y$.
Then for each $i \in \{ 1, \dots, m \}$ we have
\begin{equation}
\left\vert 
\angle p_ixy - 
\angle \tilde{p}_i\tilde{x}\tilde{y} 
\right \vert
< 2 \delta,
\label{eqn: jla}
\end{equation}
where $\angle \tilde{p}_i\tilde{x}\tilde{y}$
is the comparison angle in $M_{\kappa}^2$ for $\angle p_ixy$.
Moreover,
\begin{equation}
\pi - 2\delta < \angle p_ixy + \angle p_iyx < \pi.
\label{eqn: jlb}
\end{equation}
If in addition $d(p_i,x) = d(p_i,y)$,
then
\begin{equation}
\frac{\pi}{2} - 2\delta <
\angle p_ixy < \frac{\pi}{2}.
\label{eqn: jlc}
\end{equation}
\end{lem}

\begin{proof}
Since $(p_1,\dots,p_m)$ and $(q_1,\dots,q_m)$ are 
mutually opposite $(m,\delta)$-strainers at $x$ and $y$,
for each $i \in \{ 1, \dots, m \}$
we have 
\[
\angle p_ixy + \angle q_ixy \ge \angle p_ixq_i  > \pi - \delta,
\quad
\angle p_iyx + \angle q_iyx \ge \angle p_iyq_i > \pi - \delta
\]
All the edges in 
$\triangle p_ixy$ and $\triangle q_ixy$
have length smaller than $D_{\kappa}/2$.
Comparing them with their comparison triangles
in $M_{\kappa}^2$,
we see
\[
\angle p_ixy + \angle p_iyx < \pi,
\quad
\angle q_ixy + \angle q_iyx < \pi.
\]
Combining these inequalities leads to 
the desired ones \eqref{eqn: jla} and \eqref{eqn: jlb}.
If $d(p_i,x) = d(p_i,y)$,
then both $\angle p_ixy$ and $\angle p_iyx$
are smaller than $\pi/2$.
This together with \eqref{eqn: jlb} implies \eqref{eqn: jlc}.
\end{proof}

\subsection{Relaxed wall points}

We introduce a relaxed wall singularity
mentioned in Section \ref{sec: i}.

\begin{defn}\label{defn: dwall}
Let $m \in \N$, and 
let $\delta, \epsilon \in (0,1)$ satisfy $\delta < \epsilon$.
Let $U$ be a locally geodesically complete open subset
of a $\CAT(\kappa)$ space $X$.
We say that a point $x$ in $U$ is an
\emph{$(m,\delta,\epsilon)$-wall point}
if $x$ belongs to $S_{m,\epsilon}(U) - S_{m-1,\delta}(U)$,
where $S_{0,\delta}(U)$ is defined to be empty.
For a subset $A$ of $U$,
we denote by $W_{m,\delta,\epsilon}(A)$
the set of all $(m,\delta,\epsilon)$-wall points in $A$,
and call it 
the \emph{$(m,\delta,\epsilon)$-wall set in $A$}.
\end{defn}

We verify the following:

\begin{lem}\label{lem: walldwall}
For every $n \in \N$,
there exists $\delta^{\ast} \in (0,1)$ 
such that for every $\delta \in (0,\delta^{\ast})$
the following holds:
If a locally geodesically complete open subset $U$
of a $\CAT(\kappa)$ space $X$ satisfies
$\dim_{\mathrm{G}} U \le n$,
then $W_n(U)$ is contained in $W_{n,\delta,\delta^{\ast}}(U)$.
\end{lem}

\begin{proof}
Suppose that for some sequences 
$(\delta_k)$, $(\delta_k^{\ast})$ in $(0,1)$
with $\delta_k < \delta_k^{\ast}$ and $\delta_k^{\ast} \to 0$
there exists a sequence $(x_k)$ of points in 
locally geodesically complete open subsets $U_k$
of $\CAT(\kappa)$ spaces $X_k$
satisfing
$\dim_{\mathrm{G}} U_k \le n$
such that
each $x_k$ belongs to 
$W_n(U_k) \cap (U_k - W_{n,\delta_k,\delta_k^{\ast}}(U_k))$.
Since we have $x_k \in W_n(U_k)$,
we have $x_k \in R_{n-1,\delta_k}(U_k)$.
Hence $x_k \in R_{n,\delta_k^{\ast}}(U_k)$.
Then we find an $(n,\delta_k^{\ast})$-suspender in $\Sigma_{x_k}X_k$.
Let $\left( \ulim \Sigma_{x_k}X_k, \xi_{\omega} \right)$
the ultralimit of some sequence of the pointed spaces of directions
$\left( \Sigma_{x_k}X_k, \xi_k \right)$.
Lemma \ref{lem: omdsusp} implies that
$\ulim \Sigma_{x_k}X_k$ admits an $n$-suspender.
From \eqref{eqn: lscdim} it follows that
$\dim_{\mathrm{G}} \left( \ulim \Sigma_{x_k}X_k \right) \le n-1$.
By Proposition \ref{prop: m0susp},
$\ulim \Sigma_{x_k}X_k$ is isometric to $\Sph^{n-1}$.

On the other hand,
since $x_k \in W_n(U_k)$,
each $\Sigma_{x_k}X_k$ 
isometrically splits as $\Sph^{n-2} \ast T_k^0$
for some discrete metric space $T_k^0$ of cardinality at least $3$.
In particular,
$\ulim \Sigma_{x_k}X_k$
isometrically 
splits as $\Sph^{n-2} \ast T_{\omega}^0$
for some discrete metric space $T_{\omega}^0$ 
of cardinality at least $3$.
This is a contradiction.
\end{proof}

\section{Regularity of maps with distance coordinates}

We discuss maps on $\CBA$ spaces with distance coordinates
as a generalization of the studies by Lytchak and the author 
\cite[Section 8]{lytchak-nagano1}
(see also \cite[Sections 4 and 5]{lytchak-nagano-stadler}
by Lytchak, Stadler, and the author).

\subsection{Maps with distance coordinates}

For a point $p$ in a metric space $X$ with metric $d_X$,
we denote by $d_p$ the distance function from $p$
given by $d_p(x) := d_X(p,x)$.

Let $U$ be an open subset of a locally geodesic metric space. 
A function $f \colon U \to \R$ is said to be
\emph{convex on $U$}
if for every geodesic $\gamma \colon I \to U$ in $U$
the composition $f \circ \gamma$ is convex on $I$.

\begin{exmp}\label{exmp: conv}
Let $X$ be a $\CAT(\kappa)$ space.
From the $\CAT(\kappa)$ property
it follows that
for each $r \in (0,D_{\kappa}/2]$,
and for each $p \in X$,
the distance function $d_p \colon U_r(p) \to \R$
from $p$
on $U_r(p)$
is convex.
\end{exmp}

Let $X$ be a $\CAT(\kappa)$ space with metric $d_X$,
and let $p, x, y$ be distinct three points in $X$
with pairwise distance $< D_{\kappa}$.
Let $\gamma_{xy} \colon [a,b] \to X$ be a geodesic in $X$
from $x$ to $y$.
Then
we have the following so-called first variation formula
for distance functions:
\begin{equation}
(d_p \circ \gamma_{xy})^+(a)
= - \cos \angle_x(p,y),
\label{eqn: 1vfd}
\end{equation}
where 
$(d_p \circ \gamma_{xy})^+(a)$
is the right derivative of $d_p \circ \gamma_{xy}$ at $a$.

Let $U$ be a locally geodesically complete open subset 
of $X$ with $\diam U < D_{\kappa}/2$.
Let $(p_1,\dots,p_m)$ be an $m$-tuple of points in $X-U$
such that for each $i \in \{ 1, \dots, m \}$ we have
$d_X(p_i,x) < D_{\kappa}/2$,
and define the map $\varphi \colon U \to \R^m$ 
by $\varphi = (d_{p_1},\dots,d_{p_m})$
with distance coordinates.
Since each coordinate $d_{p_i}$ is Lipschitz and convex,
for each $x \in U$
the map $\varphi$ has directional derivative
$D_x\varphi \colon T_xX \to \R^m$
(see \cite{lytchak1}).
By the first variation formula \eqref{eqn: 1vfd},
for every $y \in U - \{x\}$,
and for every $s \in [0,\infty)$, we have
\begin{equation}
(D_x\varphi) \left( [(s,y_x')] \right)
= \left( -s \cos \angle_x(p_1,y), \dots, -s \cos \angle_x(p_m,y) \right).
\label{eqn: 1vfs}
\end{equation}

\subsection{Strainer maps}

We introduce the following:

\begin{defn}\label{defn: strm}
Let $m \in \N$ and $\delta \in (0,1)$.
Let $U$ be an open subset 
of a $\CAT(\kappa)$ space with $\diam U < D_{\kappa}/2$.
We say that a map
$\varphi \colon U \to \R^m$ 
is an \emph{$(m,\delta)$-strainer map} on $U$
if $U$ is locally geodesically complete,
and if there exists an $(m,\delta)$-strainer at $U$
with $\varphi = (d_{p_1},\dots,d_{p_m})$.
\end{defn}

\begin{exmp}\label{exmp: strf}
Let $U$ be a locally geodesically complete open subset 
of a $\CAT(\kappa)$ spaces.
If a point $p$ in $U$ is $\delta$-roughly conical,
then for some $r \in (0,D_{\kappa}/2)$
the distance function
$d_p \colon U_r(p)-\{p\} \to \R$ from $p$
is a $(1,\delta)$-strainer map on $U_r(p)-\{p\}$.
\end{exmp}

For $u \in \R^m$
we denote by $\Vert u \Vert$
the Euclidean norm of $u$ on $\R^m$.

From Lemma \ref{lem: jl}
we derive the following first variation inequalities for strainer maps
(cf.~\cite[Proposition 8.5]{lytchak-nagano1}).

\begin{lem}\label{lem: 1stvfstr0}
Let 
$\varphi \colon U \to \R^m$
be an $(m,\delta)$-strainer map on $U$
with $\varphi = (d_{p_1},\dots,d_{p_m})$.
Let $\gamma \colon [a,b] \to U$ be a geodesic
joining distinct two points in $U$.
Then for each $i \in \{ 1, \dots, m \}$ we have
\[
\left\vert
\frac{(d_{p_i} \circ \gamma)(b) - (d_{p_i} \circ \gamma)(a)}
{d(\gamma(b),\gamma(a))}
- (d_{p_i} \circ \gamma)^+(a)
\right\vert
< 2\delta.
\] 
In particular, we have
\[
\left\vert
\frac{\Vert (\varphi \circ \gamma)(b) - (\varphi \circ \gamma)(a) \Vert}
{d(\gamma(b),\gamma(a))}
- \Vert (\varphi \circ \gamma)^+(a) \Vert
\right\vert
< 2 \sqrt{m} \delta,
\]
where $(\varphi \circ \gamma)^+(a) = 
((d_{p_1} \circ \gamma)^+(a), \dots, (d_{p_m} \circ \gamma)^+(a))$.
If in addition for some $j \in \{ 1, \dots, m \}$ we have
$(d_{p_j} \circ \gamma)(a) = (d_{p_j} \circ \gamma)(b)$,
then 
$\Vert (d_{p_j} \circ \gamma)^+(a) \Vert < 2\delta$.
Moreover,
if we have $(\varphi \circ \gamma)(a) = (\varphi \circ \gamma)(b)$,
then 
\[
\Vert (\varphi \circ \gamma)^+(a) \Vert < 2 \sqrt{m} \delta.
\]
\end{lem}

Lemma \ref{lem: 1stvfstr0} tells us the following geometric property:

\begin{prop}\label{prop: 1stvfstr}
Let 
$\varphi \colon U \to \R^m$
be an $(m,\delta)$-strainer map on $U$
with $\varphi = (d_{p_1},\dots,d_{p_m})$.
Let $\gamma \colon [a,b] \to U$ be a geodesic
joining distinct two points in $U$.
Then for all $s, t \in [a,b)$ we have
\[
\Vert (\varphi \circ \gamma)^+(s) - (\varphi \circ \gamma)^+(t) \Vert
\le 4 \sqrt{m} \delta.
\]
If for some $s_1, s_2 \in [a,b)$ with $s_1 \neq s_2$
we have
$(\varphi \circ \gamma)(s_1) = (\varphi \circ \gamma(s_2))$,
then for all $t \in [a,b)$ we have
\[
\Vert (\varphi \circ \gamma)^+(t) \Vert
\le 6 \sqrt{m} \delta.
\]
\end{prop}

\subsection{Pseudo-strainer maps}

Recall that
every distance function from a single point
in a metric space
is $1$-open.
We are going to show that 
every $(m,\delta)$-strainer map is $c$-open
for some constant $c \in (0,1]$.
To do this,
we consider more general maps with distance coordinates.

\begin{defn}\label{defn: sstrm}
Let $m \in \N$ and $\delta \in (0,1)$.
Let $X$ be a $\CAT(\kappa)$ space with metric $d_X$,
and let $U$ be an open subset 
of $X$ with $\diam U < D_{\kappa}/2$.
We say that a map $\varphi \colon U \to \R^m$ is a
\emph{$(m,\delta)$-pseudo-strainer map} on $U$
if $U$ is locally geodesically complete,
and if 
there exists an $m$-tuple $(p_1,\dots,p_{m-1},p_m)$ of points in $X-U$
such that
\begin{enumerate}
\item
for each $i \in \{ 1, \dots, m \}$ we have
$\sup_{x \in U}d_X(p_i,x) < D_{\kappa}/2$;
\item
for each $x \in U$, and for $i, j \in \{ 1, \dots, m \}$ with $i \neq j$,
we have
\begin{equation}
\left\vert \angle_x (p_i,p_j) - \frac{\pi}{2} \right\vert
< 2\delta.
\label{eqn: sstrma}
\end{equation}
\end{enumerate}
\end{defn}

Mimicking some argument on an openness in 
\cite[Theorem 9.4]{burago-gromov-perelman},
we show the following
(cf.~\cite[Lemma 8.1]{lytchak-nagano1} in the $\GCBA$ setting):

\begin{lem}\label{lem: wwopen}
For $m \in \N$, 
let $\delta \in (0,1)$ satisfy $2(m-1)\delta < 1$.
Let $\varphi = (d_{p_1},\dots,d_{p_m}) \colon U \to \R^m$
be an $(m,\delta)$-pseudo-strainer map on $U$
in a $\CAT(\kappa)$ space $X$.
Then $\varphi$
is locally $1/(1-2(m-1)\delta)\sqrt{m}$-open.
In particular,
it is locally $2/\sqrt{m}$-open,
whenever $4(m-1)\delta < 1$.
\end{lem}

\begin{proof}
Let $\delta \in (0,1)$ be small enough.
For a fixed point $x_0$ in $U$,
choose $r \in (0,\infty)$ with $B_r(x_0) \subset U$.
Take an element $u_0 = (u_0^1, \dots, u_0^m)$
in $\R^m$ with $\Vert u_0 \Vert < r/2$.
The goal is to find a point $y_{\ast}$ in $U$ 
with $\varphi(y_{\ast}) = \varphi(x_0) + u_0$
such that
$(1-2(m-1)\delta) d_X(x_0,y_{\ast}) \le \Vert u_0 \Vert_{\ell^1}$,
where
$d_X$ is the metric of $X$,
and
$\Vert u_0 \Vert_{\ell^1}$
is the $\ell^1$-norm of $u_0$. 
To do this,
we are going to construct $y_{\ast}$ as a limit point of a sequence
$y_1, y_2, \dots$.

First we find $y_1 \in U$ with 
$d_X(x_0,y_1) \le \Vert u_0 \Vert_{\ell^1}$
such that
\begin{equation}
\Vert \varphi(y_1) - (\varphi(x_0) + u_0) \Vert_{\ell^1}
\le 2(m-1) \delta \Vert u_0 \Vert_{\ell^1}.
\label{eqn: wwopen1}
\end{equation}
Let $(e_1,\dots,e_m)$ 
denote the standard orthonormal basis of $\R^m$.
For $u_0^1 \in \R$,
we choose $x_1 \in U$ 
with $x_1 \in p_1x_0$ 
and
$\vert d_{p_1}(x_1) - d_{p_1}(x_0) \vert = \vert u_0^1 \vert$.
Notice that if $u_0^1 = 0$, then $x_1 = x_0$.
From our assumption \eqref{eqn: sstrma}
and the first variation formula \eqref{eqn: 1vfd},
we derive
$\vert d_{p_i}(x_1) - d_{p_i}(x_0) \vert < 2\delta \vert u_0^1 \vert$
whenever $i \neq 1$.
This implies
\[
\Vert \varphi(x_1) - (\varphi(x_0) + u_0^1 e_1) \Vert_{\ell^1}
\le 2(m-1) \delta \vert u_0^1 \vert.
\]
Successively,
for $u_0^m \in \R$,
we choose $x_m \in U$ 
with $x_m \in p_mx_{m-1}$ 
and 
$\vert d_{p_m}(x_m) - d_{p_m}(x_{m-1}) \vert = \vert u_0^m \vert$;
in particular,
\[
\Vert
\varphi(x_m) - (\varphi(x_{m-1}) + u_0^m e_m ) \Vert_{\ell^1}
\le 2(m-1) \delta \vert u_0^m \vert.
\]
Put $y_1 := x_m$.
Then
we have $d_X(x_0,y_1) \le \Vert u_0 \Vert_{\ell^1}$
and \eqref{eqn: wwopen1}.

Inductively,
for $y_k \in U$ and $u_k := \varphi(x_0) + u_0 - \varphi(y_k)$,
we can construct 
$y_{k+1} \in U$ with $d_X(y_k,y_{k+1}) \le \Vert u_k \Vert_{\ell^1}$
such that
\[
\Vert \varphi(y_{k+1}) - (\varphi(y_k) + u_k) \Vert_{\ell^1}
\le 2(m-1) \delta \Vert u_k \Vert_{\ell^1}.
\]
For each $k \in \N$, we have
\begin{equation}
\Vert u_{k+1} \Vert_{\ell^1}
\le 2(m-1) \delta \Vert u_k \Vert_{\ell^1}
\le (2(m-1) \delta)^{k+1} \Vert u_0 \Vert_{\ell^1}
\label{eqn: wwopen2}
\end{equation}
\begin{equation}
d_X(x_0,y_{k+1}) \le \sum_{i=0}^kd_X(y_i,y_{i+1})
\le \frac{1-(2(m-1)\delta)^{k+1}}{1-2(m-1)\delta} 
\Vert u_0 \Vert_{\ell^1},
\label{eqn: wwopen3}
\end{equation}
where we set $y_0 := x_0$.
Hence $(y_k)$ is a Cauchy sequence in $B_r(x_0)$.
Let $y_{\ast}$ be the limit of $(y_k)$ in $B_r(x_0)$.
Then \eqref{eqn: wwopen2} implies 
$\varphi(y_{\ast}) = \varphi(x_0) + u_0$,
and \eqref{eqn: wwopen3}
does
$(1-2(m-1)\delta) d_X(x_0,y_{\ast}) \le \Vert u_0 \Vert_{\ell^1}$.
\end{proof}

As a corollary of Lemma \ref{lem: wwopen},
we see:

\begin{lem}\label{lem: wopen}
For $m \in \N$, 
let $\delta \in (0,1)$ satisfy $2(m-1)\delta < 1$.
Then every $(m,\delta)$-strainer map 
is locally $1/(1-2(m-1)\delta)\sqrt{m}$-open.
In particular,
it is locally $2/\sqrt{m}$-open,
whenever $4(m-1)\delta < 1$.
\end{lem}

\begin{proof}
By Lemma \ref{lem: omdsusp},
every quasi $(m,\delta)$-strainer map is
an $(m,\delta)$-strainer map.
From Lemma \ref{lem: wwopen},
we conclude the lemma.
\end{proof}

\subsection{Topological regularity of strainer maps}

We show that every strainer map
is almost a submetry
(cf.~\cite[Subsection 8.3]{lytchak-nagano1}).

\begin{prop}
\label{prop: almsubm}
For every $\epsilon \in (0,1)$,
and for every $m \in \N$,
there exists $\delta \in (0,1)$ satisfying the following:
Let $\varphi \colon U \to \R^m$ be an $(m,\delta)$-strainer map.
Then $\varphi$ is $(1+\epsilon)$-Lipschitz and $(1+\epsilon)$-open,
whenever $U$ is convex.
In particular,
for every $x \in U$
the directional derivative $D_x\varphi \colon T_xX \to \R^m$
of $\varphi$ at $x$
is $(1+\epsilon)$-Lipschitz and $(1+\epsilon)$-open.
\end{prop}

\begin{proof}
Take $\epsilon \in (0,1)$ and $m \in \N$.
From an ultralimit argument we see that
if $\delta \in (0,1)$ is sufficiently small,
then 
for all $x \in U$ and $\xi \in T_xX$ with $|\xi| = 1$,
we have $\Vert (D_x\varphi)(\xi) \Vert < 1+\epsilon$;
in particular,
if in addition $U$ is convex,
then we see that $\varphi$ is $(1+\epsilon)$-Lipschitz.

As seen in Lemma \ref{lem: wopen},
for each $x \in U$
the directional derivative $D_x\varphi$
is $c$-open for some $c \in (0,1)$ depending only on $m$ and $\delta$.
Hence for every $u \in \Sph^{m-1}$
we find $\xi \in T_xX$ with $(D_x\varphi)(\xi) = u$
and $\vert \xi \vert < c$.
From an ultralimit argument we see that
if $\delta \in (0,1)$ is sufficiently small,
then for every $u \in \Sph^{m-1}$
we find $\xi \in T_xX$ with $(D_x\varphi)(\xi) = u$
and $\vert \xi \vert < 1+\epsilon$.
Applying the Lytchak open map theorem 
\cite[Theorem 1.2]{lytchak2},
we conclude that $\varphi$ is $(1+\epsilon)$-open.
Thus we prove the proposition.
\end{proof}

From Lemma \ref{lem: ndr} it follows that
if $U$ is a locally geodesically open subset of 
a $\CAT(\kappa)$ space
satisfies $\dim_{\mathrm{G}} U \le n$,
then $R_{n+1,\delta}(U)$ is empty,
provided $\delta$ is small enough.
In this case,
one hopes that
all $(n,\delta)$-strainer maps on $U$
are locally bi-Lipschitz embeddings.

\begin{prop}\label{prop: fullstr}
For every $\epsilon \in (0,1)$,
and for every $n \in \N$,
there exists 
$\delta \in (0,1)$ such that
if $\varphi \colon U \to \R^n$
is an $(n,\delta)$-strainer map with $\dim_{\mathrm{G}} U \le n$,
then 
for every $x_0 \in U$, and 
for every $r_0 \in (0,D_{\kappa}/2)$
with $B_{r_0}(x_0) \subset U$
the restriction
$\varphi |_{U_{r_0}(x_0)}$ is 
a $(1+\epsilon)$-bi-Lipschitz embedding into $\R^n$.
\end{prop}

\begin{proof}
Let $\delta \in (0,1)$ be small enough.
Let $\varphi \colon U \to \R^n$ be an $(n,\delta)$-strainer map
on $U$ in a $\CAT(\kappa)$ space $X$
with $\varphi = (d_{p_1},\dots,d_{p_n})$.
Take $x_0 \in U$ and $r_0 \in (0,D_{\kappa}/2)$
with $B_{r_0}(x_0) \subset U$.
Notice that for every $x \in B_{r_0}(x_0)$ we have
$\dim_{\mathrm{G}} \Sigma_xX \le n-1$.

We prove that the restriction $\varphi |_{U_{r_0}(x_0)}$
is injective.
Suppose that
we find distinct points $x, y$ in $U_{r_0}(x_0)$
with $\varphi(x) = \varphi(y)$.
The penetrable $\CAT(1)$ space $\Sigma_xX$
admits an $(n,\delta)$-suspender $((p_1)_x',\dots,(p_n)_x')$.
Since $\varphi(x) = \varphi(y)$,
by Lemma \ref{lem: jl} we have
$\left\vert \angle_{x}(p_i,y) - \pi/2 \right\vert < 2\delta$.
This contradicts Lemma \ref{lem: fullsusp}.
Hence $\varphi |_{U_{r_0}(x_0)}$ is injective.

From Proposition \ref{prop: almsubm}
and the injectivity of $\varphi |_{U_{r_0}(x_0)}$,
we conclude that $\varphi |_{U_{r_0}(x_0)}$ is 
$(1+\epsilon)$-bi-Lipschitz embedding 
into $\R^n$.
\end{proof}

As a corollary of Proposition \ref{prop: fullstr},
we have:

\begin{prop}\label{prop: gfullstr}
For every $n \in \N$,
there exists 
$\delta \in (0,1)$ such that
if $U$ is a locally geodesically complete open subset 
of a $\CAT(\kappa)$ space
with $\dim_{\mathrm{G}} U \le n$,
then $R_{n,\delta}(U)$ is a Lipschitz $n$-manifold.
\end{prop}

\section{Structure of singular sets in GCBA spaces}

From now on,
we focus on the geometric structure of $\GCBA$ spaces.
In this section,
we discuss the structure of their singular sets.

\subsection{Rectifiability of singular sets}

We first recall the following
(\cite[Theorem 1.6]{lytchak-nagano1}):

\begin{thm}\label{thm: srect}
\emph{(\cite{lytchak-nagano1})}
If $X$ is a $\GCBA(\kappa)$ space,
then for every $m \in \N$
the $m$-singular set $S_m(X)$ is a countable union
of subsets that are bi-Lipschitz equivalent to some compact
subsets of $\R^{m-1}$;
in particular,
$S_m(X)$ is countably $(m-1)$-rectifiable,
and satisfies
\[
\dim_{\mathrm{H}} S_m(X) \le m-1.
\]
\end{thm}

As a corollary of Theorem \ref{thm: srect},
we have:

\begin{cor}\label{cor: tsrect}
Let $X$ be a $\GCBA(\kappa)$ space of $\dim X = n$.
Then the topological singular set $S(X)$
is a countable union
of subsets that are bi-Lipschitz equivalent to some compact
subsets of $\R^{n-1}$.
In particular,
$S(X)$ is countably $(n-1)$-rectifiable,
and satisfies
\[
\dim_{\mathrm{H}} S(X) \le n-1.
\]
\end{cor}

\begin{proof}
If $\dim X = n$,
then 
Proposition \ref{prop: gfullstr} implies that
$R_n(X)$ is contained in $X-S(X)$
(see also \cite[Corollary 11.2]{lytchak-nagano1}).
This together with Theorem \ref{thm: srect}
leads to the claim.
\end{proof}

\subsection{Rectified parametrizations of singular sets}

\begin{lem}
\label{lem: infinj}
For every $\epsilon \in (0,1)$,
and for every $m \in \N$,
there exists $\delta \in (0,\epsilon)$ satisfying the following:
Let $\varphi \colon U \to \R^m$ be an $(m,\delta)$-strainer map
on $U$ contained in a $\GCBA(\kappa)$ space.
Then 
for every sequence $(x_k)$ in $S_{m+1,\epsilon}(U) - \{x\}$
with $x_k \to x$
we have
\begin{equation}
\liminf_{k \to \infty}
\frac{\Vert \varphi(x_k) - \varphi(x) \Vert}{d(x_k,x)}
\ge 1 - \epsilon.
\label{eqn: infinj0}
\end{equation}
\end{lem}

\begin{proof}
Let $\delta \in (0,\epsilon)$ be small enough
with $4\sqrt{m}\delta < \epsilon$.
Suppose that
for some sequence $(x_k)$ in $S_{m+1,\epsilon}(U) - \{x\}$
converging to $x$ 
we have
\[
\liminf_{k \to \infty}
\frac{\Vert \varphi(x_k) - \varphi(x) \Vert}{d(x_k,x)}
< 1 - \epsilon.
\]
Let $k$ be sufficiently large.
By Lemma \ref{lem: 1stvfstr0},
we have
\[
\left\vert
\frac{\Vert \varphi(x_k) - \varphi(x) \Vert}{d(x_k,x)}
-
\left\Vert (D_{x_k}\varphi)(x_{x_k}') \right\Vert
\right\vert
< 2 \sqrt{m} \delta.
\]
Hence we have
\[
\left\Vert (D_{x_k}F)(x_{x_k}') \right\Vert
< 1 - \epsilon + 2 \sqrt{m} \delta < 1 - \frac{\epsilon}{2}.
\]
By Proposition \ref{prop: 1dr}, 
the point $x$ is a $(1,\delta)$-strainer at $x_k$,
provided $k$ is large enough.
Lemma \ref{lem: pmdsusp} implies that
$x_k$ belongs to $R_{m+1,\epsilon/2}(U)$.
This is a contradiction.
\end{proof}

Let $X$ be a $\CAT(\kappa)$ space with metric $d_X$.
Let $\varphi \colon U \to \R^m$ be a locally Lipschitz map
from an open subset $U$ of $X$.
Let $B$ be a closed subset of $U$.
Assume that there exists $c \in (0,1]$ 
such that
for every point $x$ in $B$,
and for every sequence $(x_k)$ in $B - \{x\}$ converging to $x$,
we have
\[
\liminf_{k \to \infty} 
\frac{\Vert \varphi(x_k) - \varphi(x) \Vert}{d_X(x_k,x)}
\ge c.
\]
Due to the Lytchak open map theorem \cite[Proposition 1.1]{lytchak2} 
for metric spaces,
there exists an open dense subset $B_0$ of $B$ such that
the restriction $\varphi|_{B_0} \colon B_0 \to \R^m$
is a locally $(1/c)$-bi-Lipschitz embedding into $\R^m$.
If in addition $B$ is separable, 
then $B$ is countably $m$-rectifiable.

Thus we conclude the following:

\begin{prop}
\label{prop: regsing}
For every $\epsilon \in (0,1/2)$,
and for every $m \in \N$,
there exists $\delta \in (0,\epsilon)$ satisfying the following property:
Let $\varphi \colon U \to \R^m$ be an $(m,\delta)$-strainer map
on $U$ contained in a $\GCBA(\kappa)$ space.
Let $B$ be a closed subset of $U$.
Then there exists an open dense subset 
$R_\varphi(S_{m+1,\epsilon}(B))$
of $S_{m+1,\epsilon}(B)$ such that
the restriction
\[
\varphi|_{R_\varphi(S_{m+1,\epsilon}(B))} \colon 
R_\varphi(S_{m+1,\epsilon}(B)) \to \R^m
\]
is a locally $(1+2\epsilon)$-bi-Lipschitz embedding into $\R^m$.
If in addition $B$ is separable,
then the set $S_{m+1,\epsilon}(B)$ is countably $m$-rectifiable.
\end{prop}

\begin{proof}
By Lemma \ref{lem: mdopen},
the set $S_{m+1,\epsilon}(B)$ is closed in $U$.
From Lemma \ref{lem: infinj},
for every point $x$ in $S_{m+1,\epsilon}(B)$,
and for every sequence $(x_k)$ in $S_{m+1,\epsilon}(B) - \{x\}$
converging to $x$,
we derive \eqref{eqn: infinj0}.
From the Lytchak open map theorem \cite[Proposition 1.1]{lytchak2} 
for metric spaces,
we deduce the proposition.
\end{proof}

\subsection{Structure of singular sets}

We prepare the following:

\begin{lem}
\label{lem: bregsing}
For every $n \in \N$,
there exists a strictly monotone increasing $(n+1)$-tuple
$(\delta_1, \dots, \delta_n, \delta_{n+1})$
in $(0,1/2)$
satisfying the following property:
Let $X$ be a $\GCBA(\kappa)$ space of $\dim X \le n$.
If for some $m \in \{ 1, \dots, n, n+1 \}$
a non-empty closed subset $B$ of $X$
is contained in $S_{m,\delta_m}(X)$,
then 
there exists a function 
\[
\nu \colon B_0 \to \{ 0, 1, \dots, m-1 \}
\]
on an open dense subset $B_0$ of $B$
satisfying the following properties:
\begin{enumerate}
\item
for $x \in B_0$
we have $\nu(x) = 0$ 
if and only if
$x$ is isolated in $B$;
\item
if for $x \in B_0$ we have $\nu(x) \ge 1$,
then there exists a $(\nu(x),\delta_{\nu(x)})$-strainer map
$\varphi_x \colon V_x \to \R^{\nu(x)}$
from an open neighborhood $V_x$ of $x$
such that
\begin{enumerate}
\item
$\varphi_x$ is $(1+\delta_{\nu(x)+1})$-Lipschitz and
$(1+\delta_{\nu(x)+1})$-open;
\item
the restriction $\varphi_x|_{B_0}$
is a $(1+2\delta_{\nu(x)+1})$-bi-Lipschitz embedding
into $\R^{\nu(x)}$;
in particular,
for every $y \in V_x$ we have
\[
\Haus^0 \left( \varphi_x^{-1} (\varphi_x(y)) \cap B \right) \le 1.
\]
\end{enumerate}
\end{enumerate}
\end{lem}

\begin{proof}
For $n \in \N$,
let $\delta_{n+1} \in (0,1/2)$ be sufficiently small,
and take a strictly monotone increasing sequence 
$(\delta_1, \dots, \delta_n, \delta_{n+1})$
in $(0,1/2)$
such that each $\delta_i$ is small enough for $\delta_{i+1}$.
For an arbitrary subset $A$ of $U$,
by Lemma \ref{lem: ndr} we have $A = S_{n+1,\delta_{n+1}}(A)$,
and we observe
\[
A = S_{1,\delta_1}(A) 
\sqcup
\left(
\bigsqcup_{l=1}^{n-1} S_{l+1,\delta_{l+1}}
\left(
\bigcap_{k=1}^l R_{k,\delta_k}(A)
\right)
\right)
\sqcup
\left(
\bigcap_{i=1}^n R_{i,\delta_i}(A)
\right).
\]
Since $B$ is contained in $S_{m,\delta_m}(U)$
for $m \in \{ 1, \dots, n, n+1 \}$,
we have
\begin{equation}
B =
\begin{cases}
S_{1,\delta_1}(B) 
& \text{if $m = 1$,} \\
S_{1,\delta_1}(B)
\sqcup
\left(
\displaystyle
\bigsqcup_{l=1}^{m-1} S_{l+1,\delta_{l+1}}
\left(
\bigcap_{k=1}^l R_{k,\delta_k}(B)
\right)
\right)
& \text{if $m \ge 2$.}
\end{cases}
\label{eqn: bregsing0}
\end{equation}
In the case of $m = 1$,
we immediately see the lemma
since $S_{1,\delta_1}(X)$ is at most locally finite
by Proposition \ref{prop: 1dr}.

Let $m \ge 2$.
Denote by $E(B)$ the set of all isolated points in $B$,
and put $B^{\ast} := B - E(B)$.
Similarly to \eqref{eqn: bregsing0},
we have
\[
B^{\ast} =
S_{1,\delta_1}(B^{\ast})
\sqcup
\left(
\displaystyle
\bigsqcup_{l=1}^{m-1} S_{l+1,\delta_{l+1}}
\left(
\bigcap_{k=1}^l R_{k,\delta_k}(B^{\ast})
\right)
\right).
\]
Let $l \in \{ 1, \dots, m-1 \}$.
From Lemma \ref{lem: mdopen},
Propositions \ref{prop: almsubm} and \ref{prop: regsing},
it follows that 
for every point
$x$ in 
$S_{l+1,\delta_{l+1}}
\left(
\bigcap_{k=1}^l R_{k,\delta_k}(B^{\ast})
\right)$
we find an $(l,\delta_l)$-strainer map
$\varphi_x \colon V_x \to \R^l$
from an open neighborhood $V_x$ of $x$
contained in $\bigcap_{k=1}^l R_{k,\delta_k}(U - E(B))$,
and we find an open dense subset
$R_{\varphi_x} \left( S_{l+1,\delta_{l+1}}(V_x \cap B^{\ast}) \right)$
of $S_{l+1,\delta_{l+1}}(V_x \cap B^{\ast})$
satisfying the following properties:
(a)
$\varphi_x$ is $(1+\delta_{l+1})$-Lipschitz and
$(1+\delta_{l+1})$-open;
(b)
the restriction 
$\varphi_x|_{R_{\varphi_x} \left( S_{l+1,\delta_{l+1}}(V_x \cap B^{\ast}) \right)}$
is a $(1+2\delta_{l+1})$-bi-Lipschitz embedding
into $\R^l$.
Define a subset $B_l^{\ast}$ of $B^{\ast}$ by
\[
B_l^{\ast} := \bigcup
\left\{
R_{\varphi_x} \left( S_{l+1,\delta_{l+1}}(V_x \cap B^{\ast}) \right)
\mid
x \in S_{l+1,\delta_{l+1}}
\left(
\bigcap_{k=1}^l R_{k,\delta_k}(B^{\ast})
\right)
\right\}.
\]
Observe that
$B_l^{\ast}$ is open dense in 
$S_{l+1,\delta_{l+1}}
\left(
\bigcap_{k=1}^l R_{k,\delta_k}(B^{\ast})
\right)$.
Let
\[
B_0^{\ast} := \bigcup_{l=1}^m B_l^{\ast},
\quad
B_0 := B_0^{\ast} \sqcup E(B) \sqcup S_{1,\delta_1}(B^{\ast}).
\]
Then $B_0^{\ast}$ is open dense in $B^{\ast}$,
and $B_0$ is open dense in $B$.

From the construction it follows that
for each $l \in \{ 1, \dots, m-1 \}$,
and for each $x \in B_l^{\ast}$,
we find an $(l,\delta_l)$-strainer map
$\varphi_x \colon V_x \to \R^l$
from an open neighborhood $V_x$ of $x$
contained in $\bigcap_{k=1}^l R_{k,\delta_k}(U - E(B))$,
and we find an open dense subset
$R_{\varphi_x} \left( S_{l+1,\delta_{l+1}}(V_x \cap B^{\ast}) \right)$
of $S_{l+1,\delta_{l+1}}(V_x \cap B^{\ast})$
satisfying the following properties:
(a)
$\varphi_x$ is $(1+\delta_{l+1})$-Lipschitz and
$(1+\delta_{l+1})$-open;
(b)
the restriction 
$\varphi_x|_{R_{\varphi_x} \left( S_{l+1,\delta_{l+1}}(V_x \cap B^{\ast}) \right)}$
is a $(1+2\delta_{l+1})$-bi-Lipschitz embedding
into $\R^l$.
In particular,
the restriction $\varphi_x|_{B_0}$ is a 
$(1+2\delta_{l+1})$-bi-Lipschitz embedding
into $\R^l$;
indeed,
the set $V_x \cap B_0$
is contained in $S_{l+1,\delta_{l+1}}(V_x)$,
and hence $\varphi_x|_{B_0}$ is a 
$(1+2\delta_{l+1})$-bi-Lipschitz embedding
since 
$R_{\varphi_x} \left( S_{l+1,\delta_{l+1}}(V_x \cap B^{\ast}) \right)$
is dense in $S_{l+1,\delta_{l+1}}(V_x \cap B^{\ast})$.
Moreover,
for every $y \in V_x$ we have
$\Haus^0 \left( \varphi_x^{-1} (\varphi_x(y)) \cap B \right) \le 1$.
This claim follows from the following:

\begin{slem}
\label{slem: fiber1}
Let $\varphi \colon U \to V$ be a continuous map between open subsets
of metric spaces.
Let $A$ be a subset of $U$,
and $A_0$ a dense subset of $A$.
If the restriction $\varphi|_{A_0}$ is a bi-Lipschitz embedding,
then for every $x \in U$ we have
\[
\Haus^0 \left( \varphi^{-1} (\varphi(x)) \cap A \right) \le 1.
\]
\end{slem}

This sublemma can be obtained by a standard argument.

Thus we conclude Lemma \ref{lem: bregsing}.
\end{proof}

From Lemma \ref{lem: bregsing} we derive the following:

\begin{prop}
\label{prop: gregsing}
For every $n \in \N$,
there exists a strictly monotone increasing $(n+1)$-tuple
$(\delta_1, \dots, \delta_n, \delta_{n+1})$
in $(0,1/2)$
satisfying the following:
Let $X$ be a $\GCBA(\kappa)$ space of $\dim X \le n$.
If $S(X)$ is non-empty and
contained in $S_{m,\delta_m}(X)$
for some $m \in \{ 1, \dots, n, n+1 \}$,
then $m \le n$,
and there exists a function 
\[
\nu_{R(S(U))} \colon R(S(U)) \to \{ 0, 1, \dots, m-1 \}
\]
on an open dense subset $R(S(U))$ of $S(U)$
satisfying the following:
\begin{enumerate}
\item
for $x \in R(S(U))$
we have $\nu_{R(S(U))}(x) = 0$ 
if and only if
$x$ is isolated in $S(U)$;
\item
if for $x \in R(S(U))$ we have $\nu_{R(S(U))}(x) \ge 1$,
then there exists a $(\nu(x),\delta_{\nu(x)})$-strainer map
$\varphi_x \colon V_x \to \R^{\nu(x)}$
from an open neighborhood $V_x$ of $x$
contained in $U$
such that
\begin{enumerate}
\item
$\varphi_x$ is $(1+\delta_{\nu(x)+1})$-Lipschitz and
$(1+\delta_{\nu(x)+1})$-open;
\item
the restriction $\varphi_x|_{R(S(U))}$
is a $(1+2\delta_{\nu(x)+1})$-bi-Lipschitz embedding
into $\R^{\nu(x)}$;
in particular,
for every $y \in V_x$ we have
\[
\Haus^0 \left( \varphi_x^{-1} (\varphi_x(y)) \cap S(U) \right) \le 1.
\]
\end{enumerate}
\end{enumerate}
\end{prop}

\begin{proof}
For $n \in \N$,
let $\delta_{n+1} \in (0,1/2)$ be sufficiently small,
and take a strictly monotone increasing sequence 
$(\delta_1, \dots, \delta_n, \delta_{n+1})$
in $(0,1)$
such that each $\delta_i$ is small enough for $\delta_{i+1}$.
From Proposition \ref{prop: fullstr}
we derive $m \le n$.
Applying Lemma \ref{lem: bregsing}
to $S(U)$,
we conclude the proposition.
\end{proof}

Such an open dense subset 
$R(S(U))$
of $S(U)$
stated in Proposition \ref{prop: gregsing}
will be called a
\emph{regularly topological singular set in $U$}.

\section{Fibers of strainer maps on GCBA spaces}

We discuss the structure of fibers of strainer maps 
on $\GCBA$ spaces.

\subsection{Local contractivity of fibers of strainer maps}

We first review the studies in \cite{lytchak-nagano1}
on the local contractivity of fibers of strainer maps
on strainer maps.

Let $\varphi \colon U \to \R^m$ be an $(m,\delta)$-strainer map
on $U$ in a $\GCBA(\kappa)$ space
with $\varphi = (d_{p_1},\dots,d_{p_m})$.
For $x \in U$,
let $\Pi_x$ denote the fiber of $\varphi$ through $x$
defined by $\Pi_x := \varphi^{-1}(\{ \varphi(x) \})$.

Assume $20m\delta \le 1$.
Then for every closed subset $B$ of $U$
there exists $\sigma \in (0,\infty)$ such that
for every $x \in B$,
and for every $r \in (0,\sigma)$,
the open metric ball $\Pi_x \cap U_r(x)$ of the fiber
is contractible inside itself
(\cite[Theorems 1.11 and 9.1]{lytchak-nagano1}).
For every $y \in U$,
and for every $r \in (0,\infty)$ with 
$3r \le \sigma_{(p_1,\dots,p_m)}(y)$,
either $\varphi |_{U_r(x)}$ is injective or
for all $z \in B_r(y)$
the open metric ball $\Pi_z \cap U_r(z)$ of the fiber
is a connected set of diameter $\ge r$,
where $\sigma_{(p_1,\dots,p_m)}(y)$ is the straining radius at $y$
with respect to $(p_1,\dots,p_m)$
(\cite[Lemma 9.2]{lytchak-nagano1}).
Moreover,
if $U$ is connected,
then either no fiber of $\varphi$ contains an isolated point
or all the fibers of $\varphi$ are discrete
(\cite[Proposition 9.3]{lytchak-nagano1}).

We restate the dichotomy \cite[Corollary 11.2]{lytchak-nagano1}
in a generalized form:

\begin{prop}
\label{prop: dichotomy}
For every $m \in \N$,
there exists $\delta \in (0,1)$ satisfying the following:
Let $\varphi \colon U \to \R^m$ be an $(m,\delta)$-strainer map
on $U$ in a $\GCBA(\kappa)$ space
with $\varphi = (d_{p_1},\dots,d_{p_m})$.
For $x \in U$,
let $\Pi_x$ denote the fiber of $\varphi$ through $x$.
Then one of the following possibility occurs:
\begin{enumerate}
\item
There exists no point $x$ in $U$
such that $\Pi_x$ has an isolated point.
In this case, 
for every open subset $V$ contained in $U$
we have $\dim V \le m+1$.
\item
The domain $U$ of $\varphi$ is a topological $m$-manifold.
Then for every $x \in U$,
and for every $r \in (0,D_{\kappa}/2)$
with $3r \le \sigma_{(p_1,\dots,p_m)}(x)$
the restriction $\varphi |_{U_r(x)}$ is 
a $(1+\epsilon)$-bi-Lipschitz embedding into $\R^m$.
\end{enumerate}
\end{prop}

\begin{proof}
Either no fiber of $\varphi$ has isolated points
or the map $\varphi$ is locally injective
(\cite[Proposition 9.3]{lytchak-nagano1}).
In the first case,
every point $x$ in $U$ is not isolated in $\Pi_x$.
Then we find $y \in \Pi_x - \{x\}$ arbitrary closes to $x$
that is an $(m+1,12\delta)$-strained point
(\cite[Theorem 10.5]{lytchak-nagano1}).
Hence $y \in R_{m+1,12\delta}(U)$.
From Lemma \ref{lem: capmdsusp}
it follows that $\dim T_yX \ge m+1$.
This implies the first claim.
In the second case,
by combining \cite[Lemma 9.2]{lytchak-nagano1}
and Proposition \ref{prop: almsubm},
we see the second claim.
\end{proof}

\begin{rem}
On the dimension of the fibers of strainer maps,
we see the following:
For every $n \in \N$,
there exists $\delta \in (0,1)$ satisfying the following:
Let $\varphi \colon U \to \R^m$ be an $(m,\delta)$-strainer map
on $U$ in a $\GCBA(\kappa)$ space
with $\dim U \le n$.
Then for every $x \in U$ we have
\[
\dim \Pi_x \le n-m.
\]
We omit the proof since this claim will not be used in this paper.
\end{rem}

\newpage

\subsection{Length structure of fibers}

Let $\varphi \colon U \to \R^m$ be a differentiable $c$-open map
on an open subset of a $\GCBA(\kappa)$ space.
For $x_0 \in U$,
let
$\Pi_{x_0}$ 
be the fiber of $\varphi$ through $x_0$.
Then for each $x \in \Pi_{x_0}$
the tangent space $T_x\Pi_{x_0}$ at $x$ in $\Pi_{x_0}$
defined by $T_x\Pi_{x_0} = (D_x\varphi)^{-1}(\{0\})$
is obtained as the pointed Gromov-Hausdorff limit of the rescaled spaces
$((1/t)\Pi_{x_0},x)$ as $t \to 0$
(\cite[Example 7.4]{lytchak1}).

By Proposition \ref{prop: almsubm}, 
we have the following:

\begin{lem}\label{lem: fibertcone}
For $m \in \N$, 
there exists $\delta \in (0,1)$ satisfying the following:
Let $\varphi \colon U \to \R^m$
be an $(m,\delta)$-strainer map on $U$ in a 
$\GCBA(\kappa)$ space.
For $x_0 \in U$,
let
$\Pi_{x_0}$ be the fiber of $\varphi$ through $x_0$.
Then for each $x \in \Pi_{x_0}$
the tangent space $T_x\Pi_{x_0}$ at $x$
defined by $T_x\Pi_{x_0} = (D_x\varphi)^{-1}(\{0\})$ is
defined by the pointed Gromov--Hausdorff limit of 
$((1/t)\Pi_{x_0},x)$ as $t \to 0$.
\end{lem}

We are going to prove the following:

\begin{prop}\label{prop: bilipfiber}
For every $\epsilon \in (0,1)$,
and for every $m \in \N$, 
there exists $\delta \in (0,1)$
satisfying the following:
Let $\varphi \colon U \to \R^m$
be an $(m,\delta)$-strainer map on $U$ in a $\GCBA(\kappa)$ space
$X$ with metric $d_X$.
Then
for each $x_0 \in U$
there exists $r_0 \in (0,D_{\kappa}/2)$ with $U_{r_0}(x_0) \subset U$
such that
for all $x, y \in \Pi_{x_0} \cap U_{r_0}(x_0)$
there exists a curve in $\Pi_{x_0} \cap U_{r_0}(x_0)$
of length smaller than $(1+\epsilon) d_X(x,y)$;
in particular,
the intrinsic metric $d_{\Pi_{x_0}}$ on $\Pi_{x_0} \cap U_{r_0}(x_0)$
induced from $d_X$
satisfies
\[
\vert d_{\Pi_{x_0}}(x,y) - d_X(x,y) \vert < \epsilon \, d_X(x,y).
\]
\end{prop}

For the proof of the proposition,
we use gradient flows
of distance functions on the fibers.
We review the gradient flows developed 
by Lytchak \cite{lytchak3}.
More general formulation 
can be seen in \cite{lytchak3}.

Let $X$ be a metric space with metric $d_X$.
Let $f \colon X \to \R$ be a locally Lipschitz function.
For $x \in X$, 
let $\left\vert \nabla_xf \right\vert$ denote
the \emph{absolute gradient of $f$ at $x$}
defined by
\[
\vert \nabla_xf \vert
:= \max 
\left\{ 0,
\limsup_{x_k \to x} 
\frac{f(x_k)-f(x)}{d_X(x_k,x)}
\right\}.
\]
We call $x \in X$ a \emph{critical point of $f$}
if $\vert \nabla_xf \vert = 0$.
We say that
$f$ \emph{has lower semi-continuous absolute gradients}
if for every sequence $(x_k)$ converging to $x$ in $X$
we have
\[
\vert \nabla_xf \vert \le 
\liminf_{k \to \infty}
\vert \nabla_{x_k}f \vert.
\]

Let $f \colon X \to \R$ be a locally Lipschitz function
with lower semi-continuous absolute gradients.
A curve $\gamma \colon [0,a) \to X$
is said to be a {gradient curve of $f$}
if for almost all $t \in [0,a)$ we have
\begin{equation}
\vert \dot{\gamma} \vert (t) = 
\vert \nabla_{\gamma(t)}f \vert,
\quad
(f \circ \gamma)^+(t) = 
\vert \nabla_{\gamma(t)}f \vert^2,
\label{eqn: gradcurve}
\end{equation}
where
$\vert \dot{\gamma} \vert (t)$
is the metric derivative of $\gamma$ at $t$.

Let $X$ be a locally compact metric space.
Let $f \colon X \to \R$ be a locally Lipschitz function
with lower semi-continuous absolute gradients.
Due to the Lytchak gradient flow theorem
\cite[Theorem 1.7]{lytchak3},
for each non-critical point $x \in X$ of $f$
there exists a maximal gradient curve
$\gamma \colon [0,a) \to X$
of $f$ starting at $x$.
If in addition $X$ is complete,
and if $f$ is bounded or admits a uniform Lipschitz constant,
then such a curve $\gamma$ is complete,
that is, $a = \infty$.

For the proof of Proposition \ref{prop: bilipfiber},
we first show the following:

\begin{lem}\label{lem: fibergrad}
For every $\epsilon \in (0,1)$, 
and for every $m \in \N$,
there exists $\delta \in (0,1)$ satisfying the following:
Let $\varphi \colon U \to \R^m$
be an $(m,\delta)$-strainer map on $U$ in a $\GCBA(\kappa)$ space
with $\varphi = (d_{p_1}, \dots, d_{p_m})$.
For $x_0 \in U$,
take $r \in (0,D_{\kappa}/2)$ with $U_r(x_0) \subset U$.
For each $y \in \Pi_{x_0} \cap U_r(x_0)$,
let
$d_y \colon \Pi_{x_0} \to \R$ be the distance function 
on the fiber $\Pi_{x_0}$
from $y$.
Then for every $x \in \Pi_{x_0} \cap U_r(x_0) - \{y\}$
we have
\begin{equation}
\frac{1}{\sqrt{1+\epsilon}} < \vert \nabla_x(-d_y) \vert \le 1.
\label{eqn: fibergrada}
\end{equation}
In particular,
$(-d_y)$ has no critical point in 
$\Pi_{x_0} \cap U_r(x_0) - \{y\}$.
Moreover,
$(-d_y)$ has lower semi-continuous absolute gradients.
\end{lem}

\begin{proof}
Let $x \in \Pi_{x_0} \cap U_r(x_0) - \{y\}$.
From Lemma \ref{lem: jl}
it follows that
$\Vert (D_x\varphi)(y_x') \Vert < 2\sqrt{m}\delta$.
By Proposition \ref{prop: almsubm},
$D_x\varphi$ is $(1+\epsilon)$-open,
and hence
$U_{2\sqrt{m}\delta} ((D_x\varphi)(y_x'))$
is contained in
$(D_x\varphi)(U_{(1+\epsilon)2\sqrt{m}\delta}(y_x'))$.
Hence we find 
$\xi \in U_{(1+\epsilon)2\sqrt{m}\delta}(y_x')$
with $(D_x\varphi)(\xi) = 0$.
If $\delta$ is small enough,
then we have $\cos \angle_x (y_x', \xi') 
> 1/\sqrt{1+\epsilon}$.
From Lemma \ref{lem: fibertcone}
it follows that $\xi \in T_x\Pi_{x_0}$
and 
for some sequence $(x_k)$ converging to $x$ in $\Pi_{x_0}$
we have $(x_k)_x' \to \xi'$.
By the first variation formula \eqref{eqn: 1vfd}
for distance functions, 
we see
\[
\vert \nabla_x(-d_y) \vert \ge 
\lim_{k \to \infty}
\frac{d_y(x) - d_y(x_k)}{d(x,x_k)}
= \cos \angle_x (y_x', \xi') 
> \frac{1}{\sqrt{1+\epsilon}}.
\]
This together with $\vert \nabla_x(-d_y) \vert \le 1$
implies \eqref{eqn: fibergrada}.

To show the semi-continuity of the absolute gradients,
consider a gradient element $\xi_0 \in T_x\Pi_{x_0}$
with
\[
\vert \nabla_x(-d_y) \vert = \cos \angle_x (y_x', \xi_0') 
=
\max_{\xi \in T_x\Pi_{x_0}, |\xi| = 1}
\cos \angle_x (y_x', \xi).
\]
By using Lemma \ref{lem: fibertcone},
and the upper semi-continuity of angles \eqref{eqn: uscangle},
we see that 
for every sequence $(x_k)$ in $\Pi_{x_0} \cap U_r(x_0) - \{y\}$
converging to $x$
\[
\vert \nabla_x(-d_y) \vert 
= \cos \angle_x (y_x', \xi_0') \le 
\liminf_{k \to \infty} \vert \nabla_{x_k}(-d_y) \vert.
\]
Thus 
$(-d_y)$ has lower semi-continuous absolute gradients.
\end{proof}

For a curve $\gamma \colon I \to X$ in a metric space $X$
defined on a bounded interval $I$,
we denote by $L(\gamma)$ the length of $\gamma$.

Now we prove Proposition \ref{prop: bilipfiber}.

\begin{proof}[Proof of Proposition \emph{\ref{prop: bilipfiber}}]
For $\epsilon \in (0,1)$ and $m \in \N$,
Let $\delta \in (0,1)$ be sufficiently small.
Let 
$\varphi \colon U \to \R^m$
be an $(m,\delta)$-strainer map 
on $U$ in a $\GCBA(\kappa)$ space
$X$ with metric $d_X$.
For $x_0 \in U$,
choose $r \in (0,D_{\kappa}/2)$
with $U_r(x_0) \subset U$.
Let $x$ and $y$ be distinct points in $U_r(x_0)$ .
It suffices to show that
$x$ and $y$ can be joined by a curve in $\Pi_{x_0} \cap U_r(x_0)$
whose length is smaller than $(1+\epsilon)d_X(x,y)$.

By Lemma \ref{lem: fibergrad},
the function $(-d_y) \colon \Pi_{x_0} \to \R$
has no critical point in 
$\Pi_{x_0} \cap U_r(x_0) - \{y\}$,
and has lower semi-continuous absolute gradients.
Due to the Lytchak gradient flow theorem
\cite[Theorem 1.7]{lytchak3},
there exists a curve
$\gamma \colon [0,a] \to \Pi_{x_0}$
from $x$ to $y$
whose restriction $\gamma |_{[0,a)}$ to $[0,a)$
is a gradient curve of $(-d_y)$.
From the properties \eqref{eqn: gradcurve} 
and \eqref{eqn: fibergrada} for $(-d_y)$,
we derive
\[
L(\gamma) \le a
< (1+\epsilon) \, d_X(x,y).
\]
This proves Proposition \ref{prop: bilipfiber}.
\end{proof}

\subsection{Local structure of fibers of codimension one}

For $N \in \N$,
a metric space is said to be $N$-doubling
if every open metric ball of radius $r$ can be covered by
at most $N$ open metric balls of radius $r/2$.

We say that
an open metric ball $U_r(x)$ in a $\GCBA(\kappa)$ space is 
\emph{tiny}
if $r < D_{\kappa}/100$,
and if the closed metric ball 
$B_{10r}(x)$ is a compact $\CAT(\kappa)$ space.
For every tiny ball $U_r(x)$ in a $\GCBA(\kappa)$ space,
the ball $B_{10r}(x)$ is $N$-doubling
(\cite[Proposition 5.2]{lytchak-nagano1}).
We say that a tiny ball $U_r(x)$ in a $\GCBA(\kappa)$ space
has capacity $\le N$
if the ball $B_{10r}(x)$ is $N$-doubling.

Let $\varphi \colon U \to \R^m$ be an $(m,\delta)$-strainer map
with $\varphi = (d_{p_1},\dots,d_{p_m})$.
For $\epsilon \in (0,1)$,
we say that a point $x$ in the domain $U$ of $\varphi$ is 
\emph{$\epsilon$-exceptional}
if there exists no $(m+1,\epsilon)$-strainer map
$\varphi^{\dag} \colon U_x \to \R^{m+1}$
on an open neighborhood $U_x$ of $x$ contained in $U$
such that $\varphi^{\dag}$ can be written as
$\varphi^{\dag} = (d_{p_1},\dots,d_{p_m},d_{p_{m+1}})$
for some point $p_{m+1}$ in $U$.
We denote by
$E_{\epsilon}(\varphi)$ the set of all 
$\epsilon$-exceptional points in the domain $U$ of $\varphi$,
and call it the
\emph{$\epsilon$-exceptional set}.

For every $m \in \N$,
there exists $\delta \in (0,1)$ such that
for every $N \in \N$,
we find a constant $\mu \in \N$
depending only on $\delta$ and $N$ satisfying the following
property:
Let $U$ be an open subset of a $\GCBA(\kappa)$ space
such that $U$ is contained in a tiny ball of capacity $\le N$.
Let $\varphi \colon U \to \R^m$ be an $(m,\delta)$-strainer map
on $U$.
Then for every $x \in U$ the fiber $\Pi_x$ of $\varphi$ through $x$
satisfies
\[
\Haus^0 \left( \Pi_x \cap E_{12\delta}(\varphi) \right) \le \mu
\]
(\cite[Theorem 10.5]{lytchak-nagano1}).

We now study the geometric structure of fibers
of codimension one.

\begin{lem}
\label{lem: llstr1}
For every $n \in \N$,
there exist $\delta, \delta^{\ast} \in (0,\epsilon)$ with 
$\delta < \delta^{\ast}$
such that for every $N \in \N$ 
we find a constant $\mu \in \N$ depending only on $\delta$
and $N$
satisfying the following:
Let $U$ be an open subset of a $\GCBA(\kappa)$ space
with $\dim U = n$
such that $U$ is contained in a tiny ball of capacity $\le N$.
Let $\varphi \colon U \to \R^{n-1}$
be an $(n-1,\delta)$-strainer map on $U$.
Take $x_0 \in U$.
Then 
there exists $r_0 \in (0,D_{\kappa}/2)$ with $U_{2r_0}(x_0) \subset U$
such that for every $r \in (0,r)$,
and for every $x \in U_{r_0}(x_0)$,
the fiber $\Pi_x \cap U_r(x)$ 
has the structure of a finite open metric tree such that
\begin{enumerate}
\item
$\Haus^1 \left( \Pi_x \cap U_r(x) \right) < \mu r$;
\item
$S(\Pi_x \cap U_r(x)) = \Pi_x \cap U_r(x) \cap E_{\delta^{\ast}}(\varphi)$,
where $E_{\delta^{\ast}}(\varphi)$ is the 
$\delta^{\ast}$-exceptional set;
\item
$S(\Pi_x \cap U_r(x)) = S_{n,\delta^{\ast}}(\Pi_x \cap U_r(x))$,
\end{enumerate}
where $S(\Pi_x \cap U_r(x))$ is the topological singular set
in the finite metric tree $\Pi_x \cap U_r(x)$
whose cardinality is at most $\mu$.
\end{lem}

\begin{proof}
Let $\delta^{\ast} \in (0,1)$ be sufficiently small,
and let $\delta \in (0,\delta^{\ast})$
be sufficiently small for $\delta^{\ast}$.
Fix $x_0 \in U$,
and let $r \in (0,1)$ be small enough.
Take $x \in U_r(x_0)$.
Due to \cite[Theorem 9.1]{lytchak-nagano1},
the fiber $\Pi_x \cap U_r(x)$ is contractible inside itself.
Notice that the radius $r$ can be chosen uniformly in $x$ in $U_r(x_0)$.
By \cite[Theorem 10.5]{lytchak-nagano1},
for a constant $\mu_0 \in \N$
we have
\[
\Haus^0 \left( \Pi_x \cap U_r(x) \cap E_{\delta^{\ast}}(\varphi) \right) 
\le \mu_0.
\]

First we
observe the geometric structure of $\Pi_x \cap U_r(x)$.
Pick a point $y$ in 
$\left( \Pi_x \cap U_r(x) \right) - E_{\delta^{\ast}}(\varphi)$.
For some open neighborhood $U_y$ of $y$
contained in $U_r(x)$,
there exists an $(n,\delta^{\ast})$-strainer map
$\varphi^{\dag} \colon U_y \to \R^n$
given by $\varphi^{\dag} := (\varphi,d_{p_n})$
for some $p_n \in U-U_y$.
By Proposition \ref{prop: fullstr},
$\varphi^{\dag}$ is locally $(1+\epsilon)$-bi-Lipschitz embedding into 
$\R^n$.
For some open neighborhood $V_y$ of $y$ contained in $U_y$,
and for some $s \in (0,\infty)$,
we find a $(1+\epsilon)$-bi-Lipschitz curve 
$\gamma \colon (-s,s) \to V_y$ with $\gamma(0)=y$
such that $\gamma((-s,s))$ coincides with $\Pi_x \cap V_y$.
Indeed,
$\gamma$ can be chosen in such a way that
$\varphi^{\dag} \circ \gamma$
is a part of the last coordinate line in $\R^n$
through $(\varphi(y),d_{p_n}(y))$.
This implies that
$\left( \Pi_x \cap U_r(x) \right) - E_{\delta^{\ast}}(\varphi)$
is a topological $1$-manifold.
In particular,
$S \left( \Pi_x \cap U_r(x) \right)$ is contained in 
$\Pi_x \cap U_r(x) \cap E_{\delta^{\ast}}(\varphi)$.
From Proposition \ref{prop: bilipfiber}
it follows that
any two points in $\Pi_x \cap U_r(x)$
can be joined by a $(1+\delta^{\ast})$-bi-Lipschitz curve in itself.
Thus
$\Pi_x \cap U_r(x)$ 
has the structure of a finite open metric tree with
\[
\Haus^1 \left( \Pi_x \cap U_r(x) \right) < \mu r
\]
for some constant $\mu \in \N$
depending only on $\mu_0$
such that
any two points in 
$\Pi_x \cap U_r(x) \cap E_{\delta^{\ast}}(\varphi)$
adjacent to each other can be joined by
a $(1+\delta^{\ast})$-bi-Lipschitz curve in $\Pi_x \cap U_r(x)$.

In order to determine the essential vertex set
of the graph $\Pi_x \cap U_r(x)$,
we are going to show 
\begin{equation}
\Pi_x \cap U_r(x) \cap E_{\delta^{\ast}}(\varphi)
\subset S \left( \Pi_x \cap U_r(x) \right).
\label{eqn: llstr11}
\end{equation}
By Lemma \ref{lem: fibertcone},
for each $y \in \Pi_x \cap U_r(x)$
the space of directions $\Sigma_y\Pi_x$ at $y$ in $\Pi_x$
defined by
\[
\Sigma_y\Pi_x := (D_y\varphi)^{-1}(\{0\}) \cap \Sigma_yX
\]
is a finite set of cardinality $\ge 2$.
It follows that a point $y$ in $\Pi_x \cap U_r(x)$
belongs to $S \left( \Pi_x \cap U_r(x) \right)$
if and only if $\Sigma_y\Pi_x$ has cardinality $\ge 3$.

To show \eqref{eqn: llstr11},
it suffices to prove the following:

\begin{slem}
\label{slem: tridirect}
For every $\delta^{\ast} \in (0,1)$,
and for every $n \in \N$,
there exists $\delta \in (0,\delta^{\ast})$ satisfying the following:
Let $Z$ be a geodesically complete $\CAT(1)$ space
of $\dim_{\mathrm{G}} Z = n-1$
admitting an $(n-1,\delta)$-suspender
$(p_1,\dots,p_{n-1})$.
If the subset $P_{(p_1,\dots,p_{n-1})}(Z)$ of $Z$ defined by
\[
P_{(p_1,\dots,p_{n-1})}(Z) := \left\{ \, p \in Z \mid
d_Z(p_i,p) = \frac{\pi}{2} \, \right\}
\]
has cardinality $2$,
then there exists a point $p_n$ in $Z$ for which
the $n$-tuple $(p_1,\dots,p_{n-1},p_n)$ is an 
$(n,\delta^{\ast})$-suspender in $Z$.
\end{slem}

\begin{proof}
Suppose that for some sequence $(\delta_k)$ in $(0,1)$
with $\delta_k \to 0$,
there exists a sequence $(Z_k)$ of
geodesically complete $\CAT(1)$ spaces
of $\dim_{\mathrm{G}} Z_k = n-1$
admitting some $(n-1,\delta_k)$-suspenders
$(p_{1,k},\dots,p_{n-1,k})$.
Suppose in addition that
the subset $P_{(p_{1,k},\dots,p_{n-1,k})}(Z_k)$ of $Z_k$ defined by
\[
P_{(p_{1,k},\dots,p_{n-1,k})}(Z_k) := \left\{ \, p_k \in Z_k \mid
d_{Z_k}(p_{i,k},p_k) = \frac{\pi}{2} \, \right\}
\]
has cardinality $2$,
and there exists no point $p_{n,k}$ in $Z_k$ for which
the $n$-tuple $(p_{1,k},\dots,p_{n-1,k},p_{n,k})$ is an 
$(n,\delta^{\ast})$-suspender in $Z_k$.

Let $(Z_{\omega},z_{\omega})$ be the ultralimit $\ulim (Z_k,z_k)$
of a sequence $(Z_k,z_k)$.
From \eqref{eqn: lscdim}
it follows that $\dim_{\mathrm{G}} Z_{\omega} = n-1$.
By Proposition \ref{prop: m0susp},
$Z_{\omega}$ is isometric to $\Sph^{n-2} \ast T^0$
for some discrete space $T^0$.
Then $T^0$ has cardinality $\ge 3$;
indeed,
if $T^0$ has cardinality $2$,
then $Z_{\omega}$ is isometric to $\Sph^{n-1}$,
and hence for all $\omega$-large $k$,
we find $p_{n,k} \in Z_k$ for which
$(p_{1,k},\dots,p_{n-1,k},p_{n,k})$ is an 
$(n,\delta^{\ast})$-suspender in $Z_k$.
Since $T^0$ has cardinality $\ge 3$,
for all $\omega$-large $k$,
there exist three points $q_{1,k}, q_{2,k}, q_{3,k}$ in $Z_k$
satisfying
$\left\vert d_{Z_k} \left( p_{i,k}, q_{l,k} \right) - \pi/2 \right\vert
< \delta^{\ast}/10n$
for all $i \in \{ 1, \dots, n-1 \}$ and $l \in \{ 1, 2, 3 \}$.
The Euclidean cone $C_0(Z_k)$ over $Z_k$
is a geodesically complete $\CAT(0)$ space of 
$\dim_{\mathrm{G}} C_0(Z_k) = n$.
We now identify $Z_k$ with the metric sphere $S_1(0_k)$ at the vertex $0_k$
of $C_0(Z_k)$.
Let $\varphi_k \colon U_k \to \R^n$
be an $(n-1,\delta_k)$-strainer map
from an open neighborhood $U_k$ of $0_k$.
For the directional derivative 
$D_{0_k}\varphi_k \colon T_{0_k}C_0(Z_k) \to \R^n$
of $\varphi_k$ at $0_k$.
Then we have
$\left\Vert (D_{0_k}\varphi_k)(q_{l,k}) \right\Vert < \delta^{\ast}/10$.
By Proposition \ref{prop: almsubm},
$D_{0_k}\varphi_k$ is $(1+\delta^{\ast})$-open,
provided $k$ is large enough.
Hence the origin $0$ in $\R^n$
belongs to $U_{\delta^{\ast}/10} \left( (D_{0_k}\varphi_k)(q_{l,k}) \right)$,
and the ball is contained in 
$(D_{0_k}\varphi_k) 
\left( U_{(1+\delta^{\ast})\delta^{\ast}/10} (q_{l,k}) \right)$.
This implies that
we find three points 
$\bar{q}_{l,k}$ in $U_{(1+\delta^{\ast})\delta^{\ast}/10} (q_{l,k})$,
$l \in \{ 1, 2, 3 \}$,
such that $(D_{0_k}\varphi_k) \left( \bar{q}_{l,k} \right) = 0$.
Therefore
we find three points $\tilde{q}_{l,k}$ in $Z_k$,
$l \in \{ 1, 2, 3 \}$,
belongs to $P_{(p_{1,k},\dots,p_{n-1,k})}(Z_k)$.
This is a contradiction.
\end{proof}

Sublemma \ref{slem: tridirect} leads to \eqref{eqn: llstr11}.
Thus we see that
\[
\Pi_x \cap U_r(x) \cap E_{\delta^{\ast}}(\varphi)
=
S \left( \Pi_x \cap U_r(x) \right).
\]

Next we study the structure of 
$S_{n,\delta^{\ast}} \left( \Pi_x \cap U_r(x) \right)$.
From the definition we see that
$S_{n,\delta^{\ast}} \left( \Pi_x \cap U_r(x) \right)$
is contained in 
$\Pi_x \cap U_r(x) \cap E_{\delta^{\ast}}(\varphi)$,
and hence in $S \left( \Pi_x \cap U_r(x) \right)$.

We are going to verify 
\begin{equation}
\Pi_x \cap U_r(x) \cap E_{\delta^{\ast}}(\varphi)
\subset S \left( \Pi_x \cap U_r(x) \right).
\label{eqn: llstr12}
\end{equation}
To verify \eqref{eqn: llstr12},
it is enough to show:

\begin{slem}
\label{slem: fibervertex}
For every $n \in \N$,
there exists $\delta^{\ast} \in (0,1)$
such that for every $\delta \in (0,\delta^{\ast})$
the following holds:
Let $Z$ be a $\CAT(1)$ space of $\dim_{\mathrm{G}}Z = n-1$
admitting an $(n-1,\delta)$-suspender.
If the subset $P_{(p_1,\dots,p_{n-1})}(Z)$ of $Z$ defined by
\[
P_{(p_1,\dots,p_{n-1})}(Z) := \left\{ \, p \in Z \mid
d_Z(p_i,p) = \frac{\pi}{2} \, \right\}
\]
has cardinality $3$,
then $Z$ admits no $(n,\delta^{\ast})$-suspender.
\end{slem}

Sublemma \ref{slem: fibervertex} can be verified by
an ultralimit argument combined with Proposition \ref{prop: m0susp}
amd Lemma \ref{lem: omdsusp}.
The details are omitted.

From Sublemma \ref{slem: fibervertex}
we derive \eqref{eqn: llstr12}.
Therefore we conclude
\[
S(\Pi_x \cap U_r(x)) = S_{n,\delta^{\ast}}(\Pi_x \cap U_r(x)).
\]

This completes the proof of Lemma \ref{lem: llstr1}.
\end{proof}

\subsection{Lengths of fibers of codimension one}

In order to the continuity of lengths of fibers of codimension one,
we review the studies of $\DC$ maps on $\GCBA$ spaces
and $\DC$-curves in $\GCBA$ spaces.

A function on a locally geodesic space 
is said to be 
\emph{$\DC$}
if it can be locally represented as a difference of 
two Lipschitz convex functions.

A locally Lipschitz map $\varphi \colon X \to Y$ between metric spaces
is said to be 
\emph{$\DC$}
if for every $\DC$-function 
$f \colon V \to \R$ on an open subset $V$ of $Y$
the composition $f \circ \varphi$ is 
$\DC$ on $\varphi^{-1}(V)$.
Notice that any composition of $\DC$ maps are $\DC$.
A map $\varphi \colon X \to \R^m$
on a metric space $X$
is $\DC$
if and only if each coordinate functions of $\varphi$ is $\DC$.
A $\DC$ map is called a
\emph{$\DC$ isomorphism}
if it is a bi-Lipschitz homeomorphism,
and if its inverse map is $\DC$.

For each $n \in \N$,
there exists $\delta \in (0,1)$ such that
if $\varphi \colon U \to \R^n$
is an $(n,\delta)$-strainer map 
on an open subset of a $\GCBA(\kappa)$ space
with $\dim U = n$,
then the map $\varphi \colon U \to \varphi(U)$
is a $\DC$ isomorphism
(\cite[Propoition 14.4]{lytchak-nagano1}).

Let $X$ be a locally geodesic space.
For $c \in (0,\infty)$,
we say that
a curve $\gamma \colon I \to X$ on a bounded closed interval $I$
is a \emph{$\DC$ curve of norm $\le c$}
if $\gamma$ is $c$-Lipschitz,
and if for every $1$-Lipschitz convex function
$f \colon X \to \R$
the composition $f \circ \gamma$
can be written as a difference of
two $c$-Lipschitz convex functions on $I$.

Let $U$ be an open subset of a $\GCBA(\kappa)$ space
contained in a tiny ball.
If for some $c \in (0,\infty)$
a sequence $(\gamma_k)$ of a 
$\DC$ curves $\gamma_k \colon I \to U$ of norm $\le c$
on a bounded closed interval $I$
converges to a curve $\gamma \colon I \to U$,
then we have
$L(\gamma_k) \to L(\gamma)$
(\cite[Proposition 14.9]{lytchak-nagano1}).

We now prove the continuity of lengths of fibers of codimension one.

\begin{lem}
\label{lem: llstr2}
For every $n \in \N$,
there exists $\delta \in (0,1)$
satisfying the following property:
Let $U$ be an open subset of a $\GCBA(\kappa)$ space
with $\dim U = n$.
Let $\varphi \colon U \to \R^{n-1}$
be an $(n-1,\delta)$-strainer map on $U$.
Take $x_0 \in U$.
Then there exists $r_0 \in (0,D_{\kappa}/2)$ 
with $U_{2r_0}(x_0) \subset U$
such that for every $r \in (0,r_0)$,
and for every $x \in U_{r_0}(x_0)$,
the following holds:
If for a sequence $(x_k)$ of points in $U_{r_0}(x_0)$ with $x_k \to x$,
and for a constant $c \in (0,\infty)$,
there exists a sequence $(\gamma_k)$
of injective $c$-Lipschitz curves 
$\gamma_k \colon I \to \Pi_{x_k} \cap U_r(x_k)$
on a bounded closed interval $I$
converges to a curve $\gamma \colon I \to \Pi_x \cap U_r(x)$
in the pointwise topology,
then we have 
\[
L(\gamma) = \lim_{k \to \infty} L(\gamma_k);
\]
in particular,
\[
\Haus^1 \left( \Pi_x \cap U_{r_0}(x_0) \right)
= \lim_{k \to \infty}
\Haus^1 \left( \Pi_{x_k} \cap U_{r_0}(x_0) \right).
\]
\end{lem}

\begin{proof}
For $\epsilon \in (0,1)$ and $n \in \N$,
let $\delta, \delta^{\ast} \in (0,1)$ be sufficiently small
with $\delta < \delta^{\ast}$.
By Lemma \ref{lem: llstr1},
there exists $r_0 \in (0,D_{\kappa}/2)$ 
such that for every $r \in (0,r_0)$,
and for every $x \in U_{r_0}(x_0)$,
the slice $\Pi_x \cap U_r(x)$ has the structure of a metric tree
whose topological singular set
$S(\Pi_x \cap U_r(x))$ coincides with
$\Pi_x \cap U_r(x) \cap E_{\delta^{\ast}}(\varphi)$
and has cardinality $\mu$.

Assume that
there exists an $(n,\delta^{\ast})$-strainer map
$\varphi^{\dag} \colon V \to \R^n$
given by $\varphi^{\dag} := (\varphi,d_{p_n})$
for some $V \subset U$ and $p_n \in U - V$
such that
$\varphi^{\ast}$ is a $(1+\epsilon)$-bi-Lipschitz embedding
and gives a $\DC$ isomorphism.
Take 
an injective $c$-Lipschitz curve
$\gamma \colon I \to \Pi_x \cap V$
in $\Pi_x \cap V$.
Let $f \colon V \to \R$ be a $1$-Lipschitz convex function.
Let $\bar{f} := f \circ (\varphi^{\ast})^{-1}$.
By \cite[Corollary 14.7]{lytchak-nagano1},
$\bar{f}$ can be written as
$\bar{f} = \bar{f}_1-\bar{f}_2$
for some two $(1+1/n^2)(1+\epsilon)$-Lipschitz convex functions.
Let $\bar{\gamma} := \varphi^{\dag} \circ \gamma$.
Then
the curve $\bar{\gamma} \colon I \to \R^n$
is a line segment in the last coordinate line in $\R^n$.
Hence $f \circ \gamma$
can be also written as
a difference
of two $(1+1/n^2)(1+\epsilon)$-Lipschitz convex functions.
Hence $\gamma$ is a $\DC$ curve of norm $\le c(\epsilon,n)$,
where $c(\epsilon,n)$ is a constant given by
$c(\epsilon,n) := \max \{ c, (1+1/n^2)(1+\epsilon) \}$.
Thanks to \cite[Proposition 14.9]{lytchak-nagano1},
if there exists a sequence $(\gamma_k)$
of injective $c$-Lipschitz curves 
$\gamma_k \colon I \to \Pi_{x_k} \cap V$
converges to $\gamma$ in the pointwise topology,
then we see $L(\gamma_k) \to L(\gamma)$.

For a sequence $(x_k)$ of points in $U_{r_0}(x_0)$ with $x_k \to x$,
and for a constant $c \in (0,\infty)$,
there exists a sequence $(\gamma_k)$
of injective $c$-Lipschitz curves 
$\gamma_k \colon I \to \Pi_{x_k} \cap U_r(x_k)$
on a bounded closed interval $I$
converges to a curve $\gamma \colon I \to \Pi_x \cap U_r(x)$
in the pointwise topology.
If $\gamma(I)$ contains no point
in $\Pi_x \cap U_r(x) \cap E_{\delta^{\ast}}(\varphi)$,
then it can be covered by at most finitely many 
$\DC$ curves of norm $\le c(\epsilon,n)$.
Hence 
we see 
$L(\gamma_k) \to L(\gamma)$.
Assume that
$\gamma(I)$ and
$\Pi_x \cap U_r(x) \cap E_{\delta^{\ast}}(\varphi)$
have non-empty intersection.
Since the latter set has cardinality $\le \mu$,
we also see 
$L(\gamma_k) \to L(\gamma)$.

By Lemma \ref{lem: llstr2},
the fiber $\Pi_x \cap U_r(x)$
has the structure of a finite open metric tree
in which the topological singular set
$S(\Pi_x \cap U_r(x))$ has cardinality is bounded above by
some uniform constant $\mu$.
This leads to the second claim concerning $\Haus^1$.
This finishes the proof. 
\end{proof}

\section{Geometric structure of codimension one singularity}

In this section,
we prove the wall singularity theorems \ref{thm: wst} and \ref{thm: dwst}.

\subsection{Tame relaxed wall singularity}

In order to prove Theorem \ref{thm: dwst},
we consider the following:

\begin{assumption}\label{assumption: nice}
For $n \in \N$,
let $\delta^{\ast} \in (0,1/10)$ be sufficiently small,
and let $\delta \in (0,\delta^{\ast})$ be also sufficiently small 
for $\delta^{\ast}$.
Let $X$ be a $\GCBA(\kappa)$ space with metric $d_X$,
and let $U$ be a purely $n$-dimensional open subset of $X$.
Assume that for some point $x_0$ in $W_{n,\delta,\delta^{\ast}}(U)$
and some open ball $U_{r_0}(x_0)$ contained in $U$
there exists an 
$\left( n-1,\delta \right)$-strainer map
$\varphi \colon U_{r_0} \left( x_0 \right) \to \R^{n-1}$
satisfying the following:
\begin{enumerate}
\item
$\varphi$ is $\left( 1+\delta^{\ast} \right)$-Lipschitz
and $\left( 1+\delta^{\ast} \right)$-open;
\item
$\varphi |_{S_{n,\delta^{\ast}} \left( U_{r_0}(x_0) \right)}$
is a $\left( 1+2\delta^{\ast} \right)$-bi-Lipschitz embedding;
\item
for each $y \in U_{r_0} \left( x_0 \right)$
the fiber $\Pi_y$ of $\varphi$ through $y$
defined by $\Pi_y := \varphi^{-1} \left( \{ \varphi(y) \} \right)$ 
is a metric tree equipped with the interior metric $d_{\Pi_y}$
such that for all $y_1, y_2 \in \Pi_y$ we have
\[
\left\vert
d_{\Pi_y} \left( y_1, y_2 \right) - d_X \left( y_1, y_2 \right)
\right\vert
< \delta^{\ast} \, d_X \left( y_1, y_2 \right).
\]
\end{enumerate}
\end{assumption}

Under the setting in Assumption \ref{assumption: nice},
since the map
$\varphi |_{S_{n,\delta^{\ast}} \left( U_{r_0}(x_0) \right)}$
is a $\left( 1+2\delta^{\ast} \right)$-bi-Lipschitz embedding,
for each $y \in U_{r_0} \left( x_0 \right)$ we have
\begin{equation}
\Haus^0 \left( S_{n,\delta^{\ast}} \left( \Pi_y \right) \right) \le 1.
\label{eqn: nicea}
\end{equation}
Moreover,
Lemma \ref{lem: llstr2} tells us that
if for a point $y$ in $U_{r_0}(x_0)$ and 
a sequence $(y_k)$ in $U_{r_0}(x_0)$ converging to $y$
we find injective Lipschitz curves
$\gamma \colon I \to \Pi_y \cap U_{r_0}(x_0)$,
$\gamma_k \colon I \to \Pi_{y_k} \cap U_{r_0}(x_0)$,
$k = 1, 2, \dots$,
from a closed bounded interval $I$ such that
the sequence $(\gamma_k)$ converges to $\gamma$
in the pointwise topology,
then we have
\begin{equation}
L(\gamma) = \lim_{k \to \infty} L(\gamma_k);
\label{eqn: niceb}
\end{equation}
in other words,
$d_{\Pi_{y_k}}$ converges to $d_{\Pi_y}$.

The goal of this subsection is to conclude the following:

\begin{prop}
\label{prop: embedtree}
Under the setting in Assumption {\em \ref{assumption: nice}},
there exists an open neighborhood $U_0$ of $x_0$ contained in 
$U_{r_0}(x_0)$
such that $U_0$ is homeomorphic to $\R^{n-1} \times T_{l_0}^1$
for some $l_0 \in \N$ with $l_0 \ge 3$;
moreover,
we have $S(U_0) = S_{n,\delta^{\ast}}(U_0)$ and
\[
0 < \Haus^{n-1} \left( S(U_0) \right) < \infty.
\]
\end{prop}

We are going to prove Proposition \ref{prop: embedtree}.

\begin{lem}\label{lem: projsing}
Under the setting in Assumption \emph{\ref{assumption: nice}},
let $r_1 \in \left( 0, r_0/2 \right]$,
and let $r \in \left( 0, r_1 \right]$.
Then for each $y \in U_r \left( x_0 \right)$
there exists a unique point $y_0$ in $\Pi_y \cap U_{r_1}(y)$ with 
$\left\{ y_0 \right\} 
= S_{n,\delta^{\ast}} \left( \Pi_y \cap U_{r_1}(y) \right)$.
In other words,
\[
U_r \left( x_0 \right)
= \left\{ \,
y \in U_r \left( x_0 \right) \mid
\Haus^0 
\left( S_{n,\delta^{\ast}} \left( \Pi_y \cap U_{r_1}(y) \right) \right) = 1
\, \right\}.
\]
\end{lem}

\begin{proof}
Define a subset $V_r \left( x_0 \right)$ of $U_r \left( x_0 \right)$ by
\[
V_r \left( x_0 \right) := \left\{ \,
y \in U_r \left( x_0 \right) \mid
\Haus^0 
\left( S_{n,\delta^{\ast}} \left( \Pi_y \cap U_{r_1}(y) \right) \right) = 1
\, \right\}.
\]
From \eqref{eqn: nicea}
we see that $x_0 \in V_r \left( x_0 \right)$,
and hence $V_r \left( x_0 \right)$ is non-empty.

We first show that $U_r \left( x_0 \right) - V_r \left( x_0 \right)$
is open in $U_r \left( x_0 \right)$.
To do this, take
$y \in U_r \left( x_0 \right) - V_r \left( x_0 \right)$.
From \eqref{eqn: nicea} it follows that
$S_{n,\delta^{\ast}} \left( \Pi_y \cap U_{r_1}(y) \right)$
is empty,
and hence 
the slice $\Pi_y \cap U_{r_1}(y)$ is contained in
$R_{n,\delta^{\ast}} \left( U_{r_0} \left( x_0 \right) \right)$.
By Lemma \ref{lem: llstr1},
the slice $\Pi_y \cap U_{r_1}(y)$
is homeomorphic to an open interval.
Since $\varphi$ is $\left( 1+\delta^{\ast} \right)$-Lipschitz
and $\left( 1+\delta^{\ast} \right)$-open,
for each point $z$ in $U_r \left( x_0 \right)$ sufficiently close to $y$,
the slice $\Pi_z \cap U_{r_1}(z)$ is contained in 
$R_{n,\delta^{\ast}} \left( U_{r_0} \left( x_0 \right) \right)$;
in particular, 
$z$ belongs to $U_r \left( x_0 \right) - V_r \left( x_0 \right)$.
Hence $U_r \left( x_0 \right) - V_r \left( x_0 \right)$
is open in $U_r \left( x_0 \right)$.

We next show that $U_r \left( x_0 \right) - V_r \left( x_0 \right)$
is closed in $U_r \left( x_0 \right)$.
Take a point $y$ in $U_{r_0} \left( x_0 \right)$,
and a sequence 
$(y_k)$ in $U_r \left( x_0 \right) - V_r \left( x_0 \right)$
converging to $y$.
From \eqref{eqn: nicea} we see that
for each $k$
the slice $\Pi_{y_k} \cap U_{r_1} \left( y_k \right)$ is contained in
$R_{n,\delta^{\ast}} \left( U_{r_0} \left( x_0 \right) \right)$.
By Lemma \ref{lem: llstr1},
the slice $\Pi_{y_k} \cap U_{r_1} \left( y_k \right)$
is homeomorphic to an open interval.
Since $\varphi$ is $\left( 1+\delta^{\ast} \right)$-Lipschitz
and $\left( 1+\delta^{\ast} \right)$-open,
we see that
\[
\lim_{k \to \infty}
d_{\mathrm{H}} \left(
\Pi_{y_k} \cap U_{r_1} \left( y_k \right), \Pi_y \cap U_{r_1}(y)
\right)
= 0,
\]
where $d_{\mathrm{H}}$ is the Hausdorff distance.
Suppose $y \in V_r \left( x_0 \right)$.
By Lemma \ref{lem: llstr1},
the slice $\Pi_y \cap U_{r_1}(y)$
is homeomorphic to $T_l^1$
for some $l \in \N$ with $l \ge 3$.
The intervals $\Pi_{y_k} \cap U_{r_1} \left( y_k \right)$,
$k = 1, 2, \dots,$
can not converge to $\Pi_y \cap U_{r_1}(y)$
in the $d_{\mathrm{H}}$-topology.
This is a contradiction.
Hence $y$ belongs to 
$U_r \left( x_0 \right) - V_r \left( x_0 \right)$,
and hence
$U_r \left( x_0 \right) - V_r \left( x_0 \right)$
is closed in $U_r \left( x_0 \right)$.

Thus $V_r \left( x_0 \right)$
is a non-empty, closed, open subset of a connected space 
$U_r \left( x_0 \right)$.
Therefore $U_r \left( x_0 \right)$ coincides with 
$V_r \left( x_0 \right)$.
\end{proof}

We can also determine the topology of the slices of the fibers.

\begin{lem}\label{lem: fibershape}
Under the setting in Assumption \emph{\ref{assumption: nice}},
there exists an integer $l_0 \in \N$ with $l_0 \ge 3$
satisfying the following properties
for all $r_1 \in \left( 0, r_0/2 \right]$ 
and $r \in \left( 0, r_1 \right]$:
For each $y \in U_r \left( x_0 \right)$,
the slice $\Pi_y \cap U_{r_1}(y)$
is a metric tree homeomorphic to $T_{l_0}^1$
whose single vertex
is a unique point $y_0$ with
$\left\{ y_0 \right\} = 
S_{n,\delta^{\ast}} \left( \Pi_y \cap U_{r_1}(y) \right)$.
In particular,
for each point
$z_0$ in $S_{n,\delta^{\ast}} \left( U_r \left( x_0 \right) \right)$,
the slice $\Pi_{z_0} \cap U_{r_1} \left( z_0 \right)$
is a metric tree homeomorphic to $T_{l_0}^1$
whose single vertex is the given point $z_0$.
\end{lem}

\begin{proof}
By Lemmas \ref{lem: llstr1} and \ref{lem: projsing},
we define a function
$l \colon U_r \left( x_0 \right) \to \N$
in such a way that
for each $y \in U_r \left( x_0 \right)$
we have $l(y) \in \N$ with $l(y) \ge 3$
such that 
the slice $\Pi_y \cap U_{r_1}(y)$
is a metric tree homeomorphic to $T_{l(y)}^1$
whose vertex
is the point $y_0$ satisfying
$\left\{ y_0 \right\} = 
S_{n,\delta^{\ast}} \left( \Pi_y \cap U_{r_1}(y) \right)$.
Set $l_0 := l \left( x_0 \right)$,
and define a subset 
$U_r^{l_0} \left( x_0 \right)$ of $U_r \left( x_0 \right)$ by
\[
U_r^{l_0} \left( x_0 \right) := 
\left\{ \, y \in U_r \left( x_0 \right) \mid l(y) = l_0 \, \right\}.
\]
It is enough to show that
$U_r \left( x_0 \right)$ coincides with $U_r^{l_0} \left( x_0 \right)$.

We show that $U_r^{l_0} \left( x_0 \right)$ is closed in 
$U_r \left( x_0 \right)$.
Take a point $y$ in $U_r \left( x_0 \right)$,
and a sequence $(y_k)$ in $U_r^{l_0} \left( x_0 \right)$
converging to $y$,
so that $l \left( y_k \right) = l_0$ for all $k$.
If we would have
$y \not\in U_r^{l_0} \left( x_0 \right)$,
that is, $l(y) \neq l_0$,
then we would get a contradiction 
\[
\lim_{k \to \infty}
d_{\mathrm{H}} \left(
\Pi_{y_k} \cap U_{r_1} \left( y_k \right), \Pi_y \cap U_{r_1}(y)
\right)
= 0,
\]
since the slices
$\Pi_{y_k} \cap U_{r_1} \left( y_k \right)$
homeomorphic to $T_{l_0}^1$
could not converge to
the slice 
$\Pi_y \cap U_{r_1}(y)$
homeomorphic to $T_{l(y)}^1$ with $l(y) \neq l_0$.
Hence we have $y \in U_r^{l_0} \left( x_0 \right)$,
so $U_r^{l_0} \left( x_0 \right)$ is closed in $U_r \left( x_0 \right)$.

Similarly, we can show that 
$U_r \left( x_0 \right) - U_r^{l_0} \left( x_0 \right)$ is open in 
$U_r \left( x_0 \right)$.
Thus $U_r^{l_0} \left( x_0 \right)$
is a non-empty, closed, open subset of a connected space 
$U_r \left( x_0 \right)$.
Therefore $U_r \left( x_0 \right)$ coincides with 
$U_r^{l_0} \left( x_0 \right)$.
\end{proof}

We can consider a projection to the singularity.

\begin{lem}\label{lem: projlip}
Under the setting in Assumption {\em \ref{assumption: nice}},
for $r_1 \in \left( 0, r_0/2 \right]$, let 
$\lambda_{r_1} \colon U_{r_1} \left( x_0 \right) 
\to S_{n,\delta^{\ast}} \left( U_{r_0} \left( x_0 \right) \right)$
be a map defined by assigning $y$
to the unique point $\lambda_{r_1}(y) := y_0$ satisfying
$\left\{ y_0 \right\} = 
S_{n,\delta^{\ast}} \left( \Pi_y \cap U_{r_1}(y) \right)$.
Then $\lambda_{r_1}$ is $\left( 1 + 4\delta^{\ast} \right)$-Lipschitz.
\end{lem}

\begin{proof}
Since $\varphi$ is a $\left( 1 + \delta^{\ast} \right)$-Lipschitz map,
and since 
$\varphi |_{S_{n,\delta^{\ast}} \left( U_{r_0} \left( x_0 \right) \right)}$
is a $\left( 1 + 2\delta^{\ast} \right)$-bi-Lipschitz embedding,
for all $y, z \in U_{r_1} \left( x_0 \right)$ we have
\begin{multline*}
d_X \left( \lambda_{r_1}(y), \lambda_{r_1}(z) \right)
\le \left( 1 + 2 \delta^{\ast} \right) 
d_{\R^{n-1}} 
\left( \varphi \left( y_0 \right), \varphi \left( z_0 \right) \right) \\
= \left( 1 + 2 \delta^{\ast} \right) 
d_{\R^{n-1}} \left( \varphi \left( y \right), \varphi \left( z \right) \right) 
< \left( 1 + 4 \delta^{\ast} \right)  d_X(y,z),
\end{multline*}
provided $\delta^{\ast}$ is small enough.
This proves the lemma.
\end{proof}

Next we bundle the slices of the fibers together.

\begin{lem}\label{lem: basenbd}
Under the setting in Assumption {\em \ref{assumption: nice}},
for a fixed number $r_1 \in \left( 0, r_0/4 \right]$,
for each $r \in \left( 0, r_1 \right]$
the set $Y_r \left( x_0 \right)$ defined by
\begin{equation}
Y_r \left( x_0 \right)
:= \bigcup 
\left\{ \, \Pi_{y_0} \cap U_{r_1} \left( y_0 \right) \mid 
y_0 \in S_{n,\delta^{\ast}} \left( U_r \left( x_0 \right) \right)
\, \right\}
\label{eqn: basenbda}
\end{equation}
is open in $X$;
in particular,
the open set $\varphi \left( Y_r \left( x_0 \right) \right)$ of $\R^{n-1}$
coincides with
$\varphi 
\left( S_{n,\delta^{\ast}} \left( U_r \left( x_0 \right) \right) \right)$.
\end{lem}

\begin{proof}
Let $y \in Y_r \left( x_0 \right)$ be arbitrary,
and take a point
$y_0$ in $S_{n,\delta^{\ast}} \left( U_r \left( x_0 \right) \right)$
with $y \in \Pi_{y_0} \cap U_{r_1} \left( y_0 \right)$.
Then $d_X \left( x_0, y \right) < 2r_1$,
and so $y \in U_{r_0/2} \left( x_0 \right)$.
From Lemma \ref{lem: projsing} we conclude that
$\left\{ y_0 \right\} = 
S_{n,\delta^{\ast}} 
\left( \Pi_{y_0} \cap U_{r_1} \left( y_0 \right) \right)$.
From Lemma \ref{lem: projlip} it follows that
the map
\[
\lambda_{r_0/2} \colon U_{r_0/2} \left( x_0 \right) 
\to S_{n,\delta^{\ast}} \left( U_{r_0} \left( x_0 \right) \right)
\]
defined in the lemma is
$\left( 1 + 4\delta^{\ast} \right)$-Lipschitz.

Let $\epsilon \in (0,\infty)$ be arbitrary.
Then we find 
$t \in \left( 0, \epsilon/1+4\delta_n \right)$
with $U_t(y) \subset U_{r_0/2} \left( x_0 \right)$
such that
if $z \in U_t(y)$,
then for the point $z_0$ defined by $z_0 := \lambda_{r_0/2}(z)$
we have $z_0 \in U_{\epsilon}(y_0)$;
indeed,
we have
\[
d_X \left( y_0, z_0 \right)
= d_X \left( \lambda_{r_0/2}(y), \lambda_{r_0/2}(z) \right)
\le \left( 1 + 4\delta^{\ast} \right) d_X(y,z) < \epsilon.
\]
If $\epsilon$ is small enough,
then $U_{\epsilon} \left( y_0 \right)$ is contained in 
$U_r \left( x_0 \right)$,
and hence 
$z_0$ belongs to 
$S_{n,\delta_n} \left( U_r \left( x_0 \right) \right)$.
Moreover,
if $\epsilon$ is small enough,
the point $z$ belongs to $U_{r_1} \left( z_0 \right)$.
This implies that
$Y_r \left( x_0 \right)$ is open in $X$.
Since $\varphi$ is $\left( 1 + \delta^{\ast} \right)$-open,
the image $\varphi \left( Y_r \left( x_0 \right) \right)$ 
is open in $\R^{n-1}$.
From the definition of $Y_r \left( x_0 \right)$
we immediately see that
$\varphi \left( Y_r \left( x_0 \right) \right)$ 
coincides with
$\varphi 
\left( S_{n,\delta^{\ast}} \left( U_r \left( x_0 \right) \right) \right)$.
\end{proof}

\begin{rem}\label{rem: afterbasenbd}
Under the same setting as in Lemma \ref{lem: basenbd},
for the fixed number $r_1 \in \left( 0, r_0/4 \right]$,
for each $r \in \left( 0, r_1 \right]$
the set $Y_r \left( x_0 \right)$
is contained in $U_{r_0/2} \left( x_0 \right)$.
The restriction $\lambda_{r_0/2} |_{Y_r \left( x_0 \right)}$
of the map 
$\lambda_{r_0/2}$
defined in Lemma \ref{lem: projlip}
to $Y_r \left( x_0 \right)$
is $\left( 1 + 4\delta^{\ast} \right)$-Lipschitz.
The restriction $\varphi |_{Y_r \left( x_0 \right)}$
is a $\left( 1 + 4\delta^{\ast} \right)$-bi-Lipschitz embedding
from the open subset $Y_r \left( x_0 \right)$
onto an open subset $\varphi \left( Y_r \left( x_0 \right) \right)$
of $\R^{n-1}$.
\end{rem}

Under the setting in Assumption \ref{assumption: nice},
for a fixed $r_1 \in \left( 0, r_0/4 \right]$,
for $r \in \left( 0, r_1 \right)$
let $Y_r \left( x_0 \right)$
be the open set in $X$ defined in \eqref{eqn: basenbda}.
As seen in Lemma \ref{lem: fibershape},
we find some $l_0 \in \N$ with $l_0 \ge 3$
such that for each 
$y_0 \in S_{n,\delta^{\ast}} \left( U_r \left( x_0 \right) \right)$
the slice $\Pi_{y_0} \cap  U_{r_1} \left( y_0 \right)$
is a metric tree homeomorphic to $T_{l_0}^1$
whose single vertex is the given point $y_0$.
For the original given point 
$x_0 \in S_{n,\delta^{\ast}} \left( U_r \left( x_0 \right) \right)$,
let $x_1, \dots, x_{l_0}$ be the $l_0$ boundary points 
of the slice $\Pi_{x_0} \cap  U_{r_1} \left( x_0 \right)$
with 
$\{ x_1, \dots, x_{l_0} \} 
= \Pi_{x_0} \cap S_{r_1} \left( x_0 \right)$;
in this case, we can write
\[
\Pi_{x_0} \cap B_{r_1} \left( x_0 \right)
= \bigcup_{i=1}^{l_0} \overline{x_0x_i},
\]
where each $\overline{x_0x_i}$ is the unique geodesic
in the slice $\Pi_{x_0} \cap  U_{r_1} \left( x_0 \right)$
from $x_0$ to $x_i$
whose length is contained in 
$\left[ r_1, \left( 1 + \delta_n \right) r_1 \right]$.
Similarly,
for each point 
$y_0$ in $S_{n,\delta^{\ast}} \left( U_r \left( x_0 \right) \right)$,
let $y_1, \dots, y_{l_0}$ be the $l_0$ boundary points 
of the slice $\Pi_{y_0} \cap  U_{r_1} \left( y_0 \right)$
with 
$\{ y_1, \dots, y_{l_0} \} 
= \Pi_{y_0} \cap S_{r_1} \left( y_0 \right)$;
in this case, we can write
\[
\Pi_{y_0} \cap B_{r_1} \left( y_0 \right)
= \bigcup_{i=1}^{l_0} \overline{y_0y_i},
\]
where each $\overline{y_0y_i}$ is the unique geodesic
in the slice $\Pi_{y_0} \cap  U_{r_1} \left( y_0 \right)$
from $y_0$ to $y_i$
whose length is contained in 
$\left[ r_1, \left( 1 + \delta_n \right) r_1 \right]$.

\begin{lem}\label{lem: choice}
Under the setting in Assumption {\em \ref{assumption: nice}},
for a fixed number $r_1 \in \left( 0, r_0/4 \right]$,
for each $r \in \left( 0, r_1/100 \right)$,
let $Y_r \left( x_0 \right)$
be the open set in $X$ defined in \eqref{eqn: basenbda}.
For the original given point 
$x_0 \in S_{n,\delta^{\ast}} \left( U_r \left( x_0 \right) \right)$,
and for a given point
$y_0 \in S_{n,\delta^{\ast}} \left( U_r \left( x_0 \right) \right)$,
we write 
\[
\Pi_{x_0} \cap B_{r_1} \left( x_0 \right)
= \bigcup_{i=1}^{l_0} \overline{x_0x_i},
\quad
\Pi_{y_0} \cap B_{r_1} \left( y_0 \right)
= \bigcup_{i=1}^{l_0} \overline{y_0y_i}.
\]
Then for each point 
$y$ in $\Pi_{y_0} \cap U_{r_1} \left( y_0 \right)$
distinct to $y_0$,
and for the number $i(y) \in \{ 1, \dots, l_0 \}$
satisfying $y \in \overline{y_0y_i} - \left\{ y_0, y_{i(y)} \right\}$,
there exists a unique number 
$i_0(y)$ in $\left\{ 1, \dots, l_0 \right\}$
such that
\[
d_X \left( x_{i_0(y)}, y_{i(y)} \right)
< \delta_n r_1 + 10r.
\]
Moreover,
the map
$\iota \colon 
Y_r \left( x_0 \right) -
S_{n,\delta^{\ast}} \left( U_r \left( x_0 \right) \right)
\to
\left\{ 1, \dots, l_0 \right\}$
defined by 
$\iota(y) := i_0(y)$
is continuous.
\end{lem}

\begin{proof}
Let $y \in \Pi_{y_0} \cap U_{r_1} \left( y_0 \right)$
be distinct to $y_0$,
and let $i(y)$ be the number in $\{ 1, \dots, l_0 \}$
satisfying $y \in \overline{y_0y_i} - \left\{ y_0, y_{i(y)} \right\}$.
Since $\varphi$ is $\left( 1 + \delta^{\ast} \right)$-Lipschitz
and $\left( 1 + \delta^{\ast} \right)$-open,
we have
\[
d_X \left( y_{i(y)}, \Pi_{x_0} \right)
\le
(1 + \delta^{\ast})^2 d_X \left( x_0, y_0 \right)
< 4r.
\]
Hence we find a point $x_{\perp}$ in $\Pi_{x_0}$
satisfying $d \left( x_{\perp}, y_{i(y)} \right) < 4r$.
Notice that
$x_{\perp} \in \Pi_{x_0} - \left\{ x_0 \right\}$.
From $d_X \left( y_0, y_{i(y)} \right) = r_1$ and 
$d_X \left( x_0, y_0 \right) < r$,
we derive
\begin{equation}
\left\vert
d_X \left( x_0, x_{\perp} \right) - r_1
\right\vert
< 5r.
\label{eqn: choice1}
\end{equation}
For the point $x_{\perp}$ in $\Pi_{x_0} - \left\{ x_0 \right\}$,
we find some number $i_0(y)$ in $\left\{ 1, \dots, l_0 \right\}$
such that
either $x_{\perp}$ lies on the geodesic
$\overline{x_0x_{i_0(y)}}$ in $\Pi_{x_0}$
or 
$x_{i_0(y)}$ lies on the geodesic
$\overline{x_0x_{\perp}}$.
In both the two cases,
since
we have \eqref{eqn: choice1}
and
$\left\vert d_{\Pi_{x_0}} - d_X \right\vert < \delta^{\ast} d_X$,
we see
\begin{equation}
d_X \left( x_{\perp}, x_{i_0(y)} \right)
\le d_{\Pi_{x_0}} \left( x_{\perp}, x_{i_0(y)} \right)
< \delta^{\ast} r_1 +5r.
\label{eqn: choice2}
\end{equation}
From \eqref{eqn: choice2} and
$d_X \left( x_{\perp}, y_{i(y)} \right) < 4r$,
we derive
\[
d_X \left( x_{i_0(y)}, y_{i(y)} \right)
< \delta^{\ast} r_1 + 10r.
\]

Suppose now that for another number
$i_0$ in $\left\{ 1, \dots, l_0 \right\}$
with $i_0 \neq i(y)$
we have
$d_X \left( x_{i_0}, y_{i(y)} \right)
< \delta^{\ast} r_1 + 10r$.
Since
$\left\vert d_{\Pi_{x_0}} - d_X \right\vert < \delta^{\ast} d_X$,
\begin{multline*}
d_X \left( x_{i_0}, x_{i_0(y)} \right)
> \frac{1}{1 + \delta^{\ast}}
d_{\Pi_{x_0}} \left( x_{i_0}, x_{i_0(y)} \right) \\
=
\frac{1}{1 + \delta^{\ast}}
\left\{
d_{\Pi_{x_0}} \left( x_0, x_{i_0} \right)
+
d_{\Pi_{x_0}} \left( x_0, x_{i_0(y)} \right)
\right\} 
\ge \frac{2r_1}{1 + \delta^{\ast}} > r_1,
\end{multline*}
and hence
$d_X \left( x_{i_0(y)}, y_{i(y)} \right) 
> r_1 - \delta^{\ast} r_1 + 10r$.
This yields a contradiction to our present assumption
$r \in \left( 0, r_1/100 \right)$,
since $\delta^{\ast}$ is assumed to be small enough 
with $\delta^{\ast} < 1/10$.
Thus 
there exists a unique number $i_0(y)$ in $\left\{ 1, \dots, l_0 \right\}$
such that
$d_X \left( x_{i_0(y)}, y_{i(y)} \right)
< \delta^{\ast} r_1 + 10r$.

For the given point $y$, take a point
$z$ in $Y_r \left( x_0 \right) -
S_{n,\delta^{\ast}} \left( U_r \left( x_0 \right) \right)$
sufficiently close to $y$.
From the same argument discussed above,
it follows that
there exists a unique number
$i_0(z)$ in $\left\{ 1, \dots, l_0 \right\}$
with
$d_X \left( x_{i_0(z)}, z_{i(z)} \right)
< \delta^{\ast} r_1 + 10r$,
and $i_0(z)$ must be equal to $i_0(y)$.
This implies that
the map
$\iota$ defined in the lemma is continuous.
\end{proof}

\begin{lem}\label{lem: projfiber}
Under the setting in Assumption {\em \ref{assumption: nice}},
for a fixed number $r_1 \in \left( 0, r_0/4 \right]$,
for each $r \in \left( 0, r_1/100 \right)$,
let 
$Y_r \left( x_0 \right)$
be the open set in $X$ defined in \eqref{eqn: basenbda}
such that
the slice $\Pi_{y_0} \cap U_{r_1} \left( y_0 \right)$
is homeomorphic to a metric tree $T_{l_0}^1$
written as the Euclidean cone
\[
T_{l_0}^1 = 
C_0 \left( \left\{ \xi_1, \dots, \xi_{\ell_0} \right\} \right).
\]
Define a map
$\psi \colon Y_r \left( x_0 \right) \to T_{l_0}^1$
by
\[
\psi(y) := 
\begin{cases}
d_{\Pi_y} \left( y, \lambda_{r_0/2}(y) \right) \xi_{\iota(y)}
&
\text{
if $y \in 
Y_r \left( x_0 \right) -
S_{n,\delta^{\ast}} \left( U_r \left( x_0 \right) \right)$}, \\
0 & \text{otherwise},
\end{cases}
\]
where
$\lambda_{r_0/2} \colon U_{r_0/2} \left( x_0 \right)
\to S_{n,\delta^{\ast}} \left( U_r \left( x_0 \right) \right)$
is the $\left( 1+4\delta^{\ast} \right)$-Lipschitz map defined
in Lemma {\em \ref{lem: projlip}},
and 
$\iota \colon 
Y_r \left( x_0 \right) -
S_{n,\delta^{\ast}} \left( U_r \left( x_0 \right) \right)
\to
\left\{ 1, \dots, l_0 \right\}$
is the map defined in Lemma {\em \ref{lem: choice}}.
Then $\psi$ is a $1$-open continuous map.
\end{lem}

\begin{proof}
Since $\lambda_{r_0/2}$ and $\iota$ are continuous,
and since $d_{\Pi_{y_k}}$ converges to $d_{\Pi_y}$,
the map $\psi$ is continuous too.
We are going to prove that
$\psi$ is $1$-open.
To do this,
take an arbitrary ball
$U_t(y) \cap Y_r \left( x_0 \right)$
in $Y_r \left( x_0 \right)$ with $y \in Y_r \left( x_0 \right)$.
In order to show 
\begin{equation}
U_t \left( \psi(y) \right) \cap 
\psi \left( Y_r \left( x_0 \right) \right)
\subset 
\psi \left( U_t(y) \cap Y_r \left( x_0 \right) \right),
\label{eqn: projfiber1}
\end{equation}
let
$\zeta \in U_t \left( \psi(y) \right) \cap 
\psi \left( Y_r \left( x_0 \right) \right)$ be arbitrary.

We first assume that
the center $y$ belongs to 
$R_{n,\delta^{\ast}} \left( U_r \left( x_0 \right) \right)$,
and let $y_0 := \lambda_{r_0/2}(y)$.
We write
\[
\Pi_{y_0} \cap B_{r_1} \left( y_0 \right)
= \bigcup_{i=1}^{\ell_0} \overline{y_0y_i}
\]
in such a way that
for some $i(y) \in \left\{ 1, \dots, l_0 \right\}$
we have
$y \in \overline{y_0y_{i(y)}}$.

\medskip

\noindent
\emph{Case} 1.
$\zeta \neq 0$, and 
the direction $\zeta'$ coincides with $\xi_{\iota(y)}$.
In this case,
there exists a unique point $z$ in $\overline{y_0y_{i(y)}}$
with $d_{\Pi_y} \left( y_0, z \right) = d_{T_{l_0}^1}(0,\zeta)$,
so that $\psi(z) = \zeta$.
If $z \in \overline{y_0y}$,
then we have
\begin{align*}
d_X(y,z) &\le d_{\Pi_y} \left( y, z \right) 
= d_{\Pi_y} \left( y_0, y_{i(y)} \right) 
- d_{\Pi_y} \left( y_0, z \right) 
- d_{\Pi_y} \left( y, y_{i(y)} \right) \\
&= d_{T_{l_0}^1} \left( 0, \psi \left( y_{i(y)} \right) \right)
- d_{T_{l_0}^1}(0,\zeta)
- d_{T_{l_0}^1} \left( \psi(y), \psi \left( y_{i(y)} \right) \right) \\
&= d_{T_{l_0}^1} \left( \psi(y), \zeta \right).
\end{align*}
If $z \in \overline{yy_{i(y)}}$,
then we have
\begin{align*}
d_X(y,z) &\le d_{\Pi_y} \left( y, z \right) 
= d_{\Pi_y} \left( y_0, y_{i(y)} \right) 
- d_{\Pi_y} \left( y_0, y \right) 
- d_{\Pi_y} \left( z, y_{i(y)} \right) \\
&= d_{T_{l_0}^1} \left( 0, \psi \left( y_{i(y)} \right) \right)
- d_{T_{l_0}^1} \left( 0, \psi(y) \right)
- d _{T_{l_0}^1}\left( \zeta, \psi \left( y_{i(y)} \right) \right) \\
&= d_{T_{l_0}^1} \left( \psi(y), \zeta \right).
\end{align*}
Hence we see $d_X(y,z) < t$.
This implies \eqref{eqn: projfiber1}.

\medskip

\noindent
\emph{Case} 2.
$\zeta \neq 0$, and 
$\zeta'$ coincides with some $\xi_j$
distinct to $\xi_{\iota(y)}$.
In this case,
we have a unique point $z \in \overline{y_0y_j}$
with $d_{\Pi_y} \left( y_0, z \right) = d_{T_{l_0}^1}(0,\zeta)$,
so that $\psi(z) = \zeta$.
Then we have
\begin{align*}
d_X(y,z) 
&\le d_X \left( y_0, y \right) + d_X \left( y_0, z \right)
= d_{\Pi_y} \left( y_0, z \right) + d_{\Pi_y} \left( y_0, z \right) \\
&= d_{T_{l_0}^1} \left( 0, \psi(y) \right) + d_{T_{l_0}^1}(0,\zeta) 
= d_{T_{l_0}^1} \left( \psi(y), \zeta \right).
\end{align*}
Hence $d_X(y,z) < t$,
and so \eqref{eqn: projfiber1}.

\medskip

\noindent
\emph{Case} 3.
$\zeta = 0$.
Similarly to Case 2,
we immediately show \eqref{eqn: projfiber1}.

\medskip

Thus we have shown \eqref{eqn: projfiber1},
provided $y$ belongs to 
$R_{n,\delta^{\ast}} \left( U_r \left( x_0 \right) \right)$.
On the other hand,
even if
$y$ belongs to 
$S_{n,\delta^{\ast}} \left( U_r \left( x_0 \right) \right)$,
we can prove \eqref{eqn: projfiber1} similarly to Case 3.
Therefore $\psi$ is $1$-open.
\end{proof}

\begin{rem}\label{rem: afterprojfiber}
If we define a map
$\tilde{\psi} \colon Y_r \left( x_0 \right) \to T_{l_0}^1$
by
\[
\tilde{\psi}(y) := 
d_X \left( y, \lambda_{r_0/2}(y) \right) \xi_{\iota(y)}
\]
if 
$y \in Y_r \left( x_0 \right) -
S_{n,\delta^{\ast}} \left( U_r \left( x_0 \right) \right)$,
and by $\tilde{\psi}(y) := 0$ otherwise,
then the author does not know whether
$\tilde{\psi}$ is open.
\end{rem}

We prove a key lemma
to obtain Proposition \ref{prop: embedtree}.

\begin{lem}\label{lem: embed}
Under the setting in Assumption {\em \ref{assumption: nice}},
for a fixed number $r_1 \in \left( 0, r_0/4 \right]$,
for each $r \in \left( 0, r_1/100 \right)$,
let 
$Y_r \left( x_0 \right)$
be the open set in $X$ defined in \eqref{eqn: basenbda}
such that
the slice $\Pi_{y_0} \cap U_{r_1} \left( y_0 \right)$
is homeomorphic to a metric tree
$T_{l_0}^1$ for some $l_0 \in \N$ with $l_0 \ge 3$.
We define a map
$\Phi \colon Y_r \left( x_0 \right) \to 
\R^{n-1} \times T_{l_0}^1$
by
\[
\Phi(y) := \left( \varphi(y), \psi(y) \right),
\]
where
$\psi \colon Y_r \left( x_0 \right) \to T_{l_0}^1$
is the map defined
in Lemma {\em \ref{lem: projfiber}}.
Then $\Phi$ is an open embedding 
from $Y_r \left( x_0 \right)$
onto an open subset 
of $\R^{n-1} \times T_{l_0}^1$
with
\begin{equation}
\left( \mathrm{pr}_{\R^{n-1}} \circ \Phi \right)
\left( S_{n,\delta^{\ast}} \left( Y_r \left( x_0 \right) \right) \right)
= F \left( S_{n,\delta^{\ast}} \left( U_r \left( x_0 \right) \right) \right),
\label{eqn: embeda}
\end{equation}
where 
$\mathrm{pr}_{\R^{n-1}} \colon 
\R^{n-1} \times T_{l_0}^1 \to \R^{n-1}$
is the standard projection onto $\R^{n-1}$.
In particular,
we have
\begin{equation}
0 < \Haus^{n-1} \left( \left( \mathrm{pr}_{\R^{n-1}} \circ \Phi \right)
\left( S_{n,\delta^{\ast}} \left( Y_r \left( x_0 \right) \right) \right) \right)
< \infty.
\label{eqn: embedb}
\end{equation}
\end{lem}

\begin{proof}
In Lemma \ref{lem: projfiber},
we already see that $\psi$ is $1$-open and continuous.
Since $\varphi$ is $\left( 1+\delta^{\ast} \right)$-Lipschitz and
$\left( 1+\delta^{\ast} \right)$-open,
the map $\Phi$ is open and continuous.
For each slice 
$\Pi_{y_0} \cap U_{r_1} \left( y_0 \right)$ 
in $Y_r \left( x_0 \right)$,
the restriction
$\psi |_{\Pi_{y_0} \cap U_{r_1} \left( y_0 \right)}$
is injective.
Hence $\Phi$ is injective.
By Lemma \ref{lem: basenbd},
the open subset
$\varphi \left( Y_r \left( x_0 \right) \right)$
of $\R^{n-1}$
coincides with
$\varphi 
\left( S_{n,\delta^{\ast}} \left( U_r \left( x_0 \right) \right) \right)$,
and so we have \eqref{eqn: embeda}.
Moreover,
we have
\[
0 < \Haus^{n-1} \left( \varphi
\left( Y_r \left( x_0 \right) \right) \right)
< \infty,
\]
and hence we obtain \eqref{eqn: embedb}. 
This proves the lemma.
\end{proof}

Now we arrive at Proposition \ref{prop: embedtree}.

\begin{proof}[Proof of Proposition $\ref{prop: embedtree}$]
Under the setting in Assumption \ref{assumption: nice},
for a fixed number $r_1 \in \left( 0, r_0/4 \right]$,
for each $r \in \left( 0, r_1/100 \right)$,
let 
$Y_r \left( x_0 \right)$
be the open set in $X$ defined in \eqref{eqn: basenbda}.
By Lemma \ref{lem: embed},
the set $Y_r \left( x_0 \right)$ is an open neighborhood of $x_0$ 
contained in 
$U_{r_0}(x_0)$
such that $Y_r \left( x_0 \right)$ is homeomorphic to 
$\R^{n-1} \times T_{l_0}^1$
for some $l_0 \in \N$ with $l \ge 3$;
moreover,
we have 
$S \left( Y_r \left( x_0 \right) \right) 
= S_{n,\delta^{\ast}} \left( Y_r \left( x_0 \right) \right)$ 
and
\[
0 < \Haus^{n-1} \left( S \left( Y_r \left( x_0 \right) \right) \right) < \infty.
\]
This completes the proof of Proposition \ref{prop: embedtree}.
\end{proof}

\subsection{Proof of the wall singularity theorems}

We first prove the relaxed wall singularity theorem \ref{thm: dwst}.

\begin{proof}[Proof of Theorem $\ref{thm: dwst}$]
For $n \in \N$,
let $\delta^{\ast} \in (0,1/10)$ be sufficiently small,
and let $\delta \in (0,\delta^{\ast})$ be also sufficiently small 
for $\delta^{\ast}$.
Let $X$ be a $\GCBA(\kappa)$ space with metric $d_X$,
and let $U$ be a purely $n$-dimensional open subset of $X$.
Take an arbitrary $x$ in $W_{n,\delta,\delta^{\ast}}(U)$.
By Proposition \ref{prop: almsubm},
there exists an $\left( n-1,\delta \right)$-strainer map
$\varphi_x \colon V_x \to \R^{n-1}$
from some open neighborhood $V_x$ of $x$ contained in $U$
that is locally $\left( 1+\delta^{\ast} \right)$-Lipschitz
and locally $\left( 1+\delta^{\ast} \right)$-open.
Let $r_x$ be a sufficiently small positive number 
with $U_{2r_x}(x) \subset V_x$
such that
$\varphi_x |_{U_{2r_x}(x)}$ is 
$\left( 1+\delta^{\ast} \right)$-Lipschitz and
$\left( 1-\delta^{\ast} \right)$-open.
From Proposition \ref{prop: regsing}
it follows that
there exists an open dense subset
$R_{\varphi_x} \left( S_{n,\delta^{\ast}} \left( U_{2r_x}(x) \right) \right)$ 
of $S_{n,\delta^{\ast}} \left( U_{2r_x}(x) \right)$
for which the restriction
$\varphi_x |_{R_{\varphi_x} 
\left( S_{n,\delta^{\ast}} \left( U_{2r_x}(x) \right) \right)}$
is $\left( 1+2\delta^{\ast} \right)$-bi-Lipschitz embedding.
Moreover,
by Propositions \ref{prop: bilipfiber} and Lemma \ref{lem: llstr2},
for every $y \in U_{r_x}(x)$,
and for the fiber $\Pi_y$ defined by 
$\Pi_y := \varphi_x^{-1} \left( \{ \varphi_x(y) \} \right)$,
and for every $s \in \left( 0, r_x \right)$,
the slice $\Pi_y \cap B_s(y)$
is a metric tree equipped with the interior metric $d_{\Pi_y}$
such that for all $y_1, y_2 \in \Pi_y \cap B_s(y)$ we have
\[
\left\vert
d_{\Pi_y} \left( y_1, y_2 \right) - d_X \left( y_1, y_2 \right)
\right\vert
< \delta^{\ast}.
\]

For every $\epsilon \in (0,\infty)$,
we find a point
$x_0$ in 
$R_{\varphi_x} \left( S_{n,\delta^{\ast}} \left( U_{2r_x}(x) \right) \right)$
with $d_X \left( x, x_0 \right) < \epsilon$,
and a sufficiently small number $r_0$ in $\left( 0, r_x \right)$
satisfying $U_{r_0} \left( x_0 \right) \subset U_{2r_x}(x)$,
such that
$S_{n,\delta^{\ast}} \left( U_{r_0} \left( x_0 \right) \right)$
is contained in
$R_{\varphi_x} \left( S_{n,\delta^{\ast}} \left( U_{2r_x}(x) \right) \right)$.
Then the $\left( n-1, \delta \right)$-strainer map
$\varphi_x |_{U_{r_0} \left( x_0 \right)}$
satisfies all the properties in Assumption \ref{assumption: nice}.
By virtue of Proposition \ref{prop: embedtree},
we find an open neighborhood $Y_r \left( x_0 \right)$ 
of $x_0$ for some $r$,
and an integer $l_0 \in \N$ with $l_0 \ge 3$,
and an open embedding
$\Phi \colon Y_r \left( x_0 \right) \to \R^{n-1} \times T_{l_0}^1$
such that
$S \left( Y_r \left( x_0 \right) \right) = 
S_{n,\delta^{\ast}} \left( Y_r \left( x_0 \right) \right)$
and
\[
0 < \Haus^{n-1} \left( S \left( Y_r \left( x_0 \right) \right) \right) < \infty.
\]
We have completed the proof of Theorem \ref{thm: dwst}.
\end{proof}

\begin{rem}
In the proof of Theorem \ref{thm: dwst},
the author does not know whether
an open embedding
$\Phi \colon Y_r \left( x_0 \right) \to \R^{n-1} \times T_{l_0}^1$
constructed in the proof 
of Proposition \ref{prop: embedtree},
is a bi-Lipschitz embedding.
\end{rem}

We next prove Theorem \ref{thm: wst}.

\begin{proof}[Proof of Theorem $\ref{thm: wst}$]
For $n \in \N$,
let $\delta^{\ast} \in (0,1/10)$ be sufficiently small,
and let $\delta \in (0,\delta^{\ast})$ be also sufficiently small 
for $\delta^{\ast}$.
Let $X$ be a $\GCBA(\kappa)$ space,
and let $U$ be a purely $n$-dimensional open subset of $X$.
Take an arbitrary $n$-wall point $x \in W_n(U)$ in $U$.
Lemma \ref{lem: walldwall} implies $x$ belongs to 
$W_{n,\delta,\delta^{\ast}}(U)$.
From Theorem \ref{thm: dwst}
we conclude Theorem \ref{thm: wst}.
\end{proof}

\section{Characterizations of regularity of codimension two}

In this section,
we prove Theorems \ref{thm: wrt}, \ref{thm: iwrt}, and \ref{thm: dwrt}.

\subsection{Homology manifolds with an upper curvature bound}

Let $H_{\ast}$ denote the singular homology with $\Z$-coefficients.
A locally compact, separable metric space $M$ is said to be a
\emph{homology $n$-manifold} 
if for every $p \in M$ 
the local homology
$H_{\ast}(M,M-\{ p \})$ at $p$
is isomorphic to $H_{\ast}(\R^n,\R^n-\{ 0 \})$,
where $0$ is the origin of $\R^n$.
A homology $n$-manifold $M$ is a 
\emph{generalized $n$-manifold} 
if $M$ is an $\ANR$ of finite topological dimension.
Every generalized $n$-manifold has dimension $n$.
Due to the theorem of Moore (see \cite[Chapter IV]{wilder}),
for each $n \in \{ 1, 2 \}$,
every generalized $n$-manifold is a topological $n$-manifold.

Every homology $n$-manifold with an upper curvature bound
is a $\GCBA$ generalized $n$-manifold.
We refer the readers to \cite{lytchak-nagano2} for advanced studies
of homology manifolds with an upper curvature bound.
Thurston \cite[Theorem 3.3]{thurston} showed that
every homology $3$-manifold with an upper curvature bound
is a topological $3$-manifold.
Lytchak and the author \cite[Theorem 1.2]{lytchak-nagano2} 
proved that
for every homology $n$-manifold $M$ with an upper curvature bound
there exists a locally finite subset $E$ of $M$ such that
$M-E$ is a topological $n$-manifold.

We recall the following homology manifold recognition
shown in \cite[Lemma 3.1 and Corollary 3.4]{lytchak-nagano2}:

\begin{prop}\label{prop: hmrecog}
\emph{(\cite{lytchak-nagano2})}
An open subset $U$ of a $\GCBA$ space $X$
is a homology $n$-manifold
if and only if 
for every $x \in U$
the space $\Sigma_xX$ has the same homology as $\Sph^{n-1}$;
in this case,
$\Sigma_xX$ is a homology $(n-1)$-manifold,
and $T_xX$ is a homology $n$-manifold.
\end{prop}

This leads to the following:

\begin{prop}\label{prop: hmnw}
Assume that an open subset $U$ of a $\GCBA$ space $X$
is a homology $n$-manifold.
Then the $n$-wall set $W_n(U)$ in $U$ is empty.
\end{prop}

\begin{proof}
Suppose that we find an $n$-wall point $x$ in $W_n(U)$.
Then $\Sigma_xX$ isometrically splits as
$\Sph^{n-2} \ast T_l^0$ for $l \in \N$ with $l \ge 3$.
Hence $H_{n-1}(\Sigma_xX)$ is isomorphic to $\Z^{l-1}$.
This contradicts Proposition \ref{prop: hmrecog}.
\end{proof}

\subsection{Properties of codimension two}

For the proofs of Theorems \ref{thm: wrt} and \ref{thm: dwrt},
we begin with the following:

\begin{lem}\label{lem: nwcd}
Let $X$ be a $\GCBA$ space,
and $U$ a purely $n$-dimensional open subset of $X$.
Then the following are equivalent:
\begin{enumerate}
\item
the $n$-wall set $W_n(U)$ in $U$ is empty;
\item
$\dim S_n(U) \le n-2$;
\item
$\dim_{\mathrm{H}} S_n(U) \le n-2$.
\end{enumerate}
\end{lem}

\begin{proof}
Assume first that $W_n(U)$ is empty.
Then $S_n(U)$ coincides with $S_{n-1}(U)$.
From Theorem \ref{thm: srect} we derive
\[
\dim_{\mathrm{H}} S_n(U)
= \dim_{\mathrm{H}} S_{n-1}(U) \le \dim_{\mathrm{H}} S_{n-1}(X) \le n-2.
\]
Hence we see the implications 
(1) $\Rightarrow$ (2) and (1) $\Rightarrow$ (3).

Assume next that $\dim S_n(U) \le n-2$.
Suppose that $W_n(U)$ is non-empty.
Let $\delta^{\ast} \in (0,1)$ be small enough,
and let $\delta \in (0,\delta^{\ast})$ be sufficiently small
for $\delta^{\ast}$.
Lemma \ref{lem: walldwall} implies that
$W_n(U)$ is contained in $W_{n,\delta,\delta^{\ast}}(U)$.
Hence $W_{n,\delta,\delta^{\ast}}(U)$ is non-empty.
Due to Theorem \ref{thm: dwst},
there exists an open subset $U_0$ contained in $U$ for which 
$U_0$ is homeomorphic to $\R^{n-1} \times T_l^1$ 
for $l \in \N$ with $l \ge 3$;
moreover,
we have $S(U_0) = S_{n,\delta^{\ast}}(U_0)$.
Since $S_{n,\delta^{\ast}}(U_0)$ is contained in $S_n(U_0)$,
we obtain
\[
n-1 = \dim S_{n,\delta^{\ast}}(U_0) 
\le \dim S_n(U_0) \le \dim S_n(U),
\]
which contradicts the present assumption.
Thus we see the implications 
(2) $\Rightarrow$ (1) and (3) $\Rightarrow$ (1).
\end{proof}

Similarly to Lemma \ref{lem: nwcd},
we verify the following:

\begin{lem}\label{lem: cnwcd}
Let $X$ be a $\GCBA$ space,
and $U$ a purely $n$-dimensional open subset of $X$.
If $\dim S(U) \le n-2$,
then $W_n(U)$ is empty.
\end{lem}

\begin{proof}
Suppose that $W_n(U)$ is non-empty.
Due to Theorem \ref{thm: wst},
there exists an open subset $U_0$ contained in $U$ for which 
$U_0$ is homeomorphic to $\R^{n-1} \times T_k^1$ with $k \ge 3$.
Hence 
\[
n-1 = \dim S(U_0) \le \dim S(U),
\]
which is a contradiction.
\end{proof}

\subsection{Relaxed regularity of codimension two}

We are going to prove Theorem \ref{thm: dwrt}.
Lemma \ref{lem: nwcd} can be relaxed into the following:

\begin{lem}
\label{lem: rnwcd}
For every $n \in \N$,
there exist $\delta^{\ast} \in (0,1)$ 
and $\delta \in (0,\delta^{\ast})$ satisfying the following:
Let $X$ be a $\GCBA$ space,
and $U$ a purely $n$-dimensional open subset of $X$.
Then the following are equivalent:
\begin{enumerate}
\item
the $(n,\delta,\delta^{\ast})$-wall set 
$W_{n,\delta,\delta^{\ast}}(U)$ in $U$ is empty;
\item
$\dim S_{n,\delta,\delta^{\ast}}(U) \le n-2$;
\item
$\dim_{\mathrm{H}} S_{n,\delta,\delta^{\ast}}(U) \le n-2$.
\end{enumerate}
\end{lem}

\begin{proof}
For $n \in \N$,
let $\delta^{\ast} \in (0,1)$ be sufficiently small,
and let $\delta \in (0,\delta^{\ast})$ be sufficiently small
for $\delta^{\ast}$.
Assume first that $W_{n,\delta,\delta^{\ast}}(U)$ is empty.
Then $S_{n,\delta^{\ast}}(U)$ coincides with $S_{n-1,\delta^{\ast}}(U)$.
Since $S_{n-1,\delta^{\ast}}(U)$ is contained in $S_{n-1}(U)$,
Theorem \ref{thm: srect} leads to
\begin{align*}
\dim_{\mathrm{H}} S_{n,\delta_n}(U) &= 
\dim_{\mathrm{H}} S_{n-1,\delta_n}(U)
\le \dim_{\mathrm{H}} S_{n-1}(U) \\
&\le \dim_{\mathrm{H}} S_{n-1}(X) \le n-2.
\end{align*}
Therefore we see the implications
(1) $\Rightarrow$ (2) and (1) $\Rightarrow$ (3).

Assume next that $\dim S_{n,\delta^{\ast}}(U) \le n-2$.
Suppose that $W_{n,\delta,\delta^{\ast}}(U)$ is non-empty.
By Theorem \ref{thm: dwst},
there exists an open subset $U_0$ contained in $U$
for which $U_0$ is homeomorphic to 
$\R^{n-1} \times T_{l_0}^1$ for $l_0 \in \N$ with $l_0 \ge 3$;
moreover,
we have $S(U_0) = S_{n,\delta^{\ast}}(U_0)$.
Hence
$\dim S_{n,\delta^{\ast}}(U_0) = n-1$,
which is a contradiction.
Therefore we see 
(2) $\Rightarrow$ (1) and (3) $\Rightarrow$ (1).
\end{proof}

Next we prove the following:

\begin{lem}
\label{lem: rnwngb}
For every $n \in \N$,
there exist $\delta^{\ast} \in (0,1)$ 
and $\delta \in (0,\delta^{\ast})$ satisfying the following:
Let $X$ be a $\GCBA$ space,
and $U$ a purely $n$-dimensional open subset of $X$.
Then the following are equivalent:
\begin{enumerate}
\item
the $(n,\delta,\delta^{\ast})$-wall set 
$W_{n,\delta,\delta^{\ast}}(U)$ in $U$ is empty;
\item
there exists no open subset $U_0$ contained in $U$
for which 
$U_0$ is homeomorphic to 
$\R^{n-1} \times T_{l_0}^1$ for some $l_0 \in \N$ with $l_0 \ge 3$
such that $S(U_0) = S_{n,\delta^{\ast}}(U_0)$
and $\Haus^{n-1}(S(U_0)) \in (0,\infty)$;
\item
the $n$-wall set $W_n(U)$ in $U$ is empty.
\end{enumerate}
\end{lem}

\begin{proof}
For $n \in \N$,
let $\delta^{\ast} \in (0,1)$ be sufficiently small,
and let $\delta \in (0,\delta^{\ast})$ be sufficiently small.
By Theorem \ref{thm: dwst},
for every $x \in W_{n,\delta,\delta^{\ast}}(U)$,
there exists an open subset $U_0$ contained in $U$
for which 
$U_0$ is homeomorphic to 
$\R^{n-1} \times T_{l_0}^1$ for some $l_0 \in \N$ with $l_0 \ge 3$
such that $S(U_0) = S_{n,\delta^{\ast}}(U_0)$
and $\Haus^{n-1}(S(U_0)) \in (0,\infty)$.
This implies the implication (2) $\Rightarrow$ (1).

We next verify (1) $\Rightarrow$ (3).
Assume that $W_{n,\delta_n}(U)$ is empty.
From Lemma \ref{lem: walldwall}
it follows that $W_n(U)$ is contained in $W_{n,\delta_n}(U)$.
Hence $W_n(U)$ is empty.
This implies (1) $\Rightarrow$ (3).

The rest is to show (3) $\Rightarrow$ (2).
Assume that we find an $n$-wall point $x$ in $W_n(U)$.
Lemma \ref{lem: walldwall} implies $x \in W_{n,\delta,\delta^{\ast}}(U)$.
By Theorem \ref{thm: dwst},
there exists an open subset $U_0$ contained in $U$
for which 
$U_0$ is homeomorphic to 
$\R^{n-1} \times T_{l_0}^1$ for some $l_0 \in \N$ with $l_0 \ge 3$
such that $S(U_0) = S_{n,\delta^{\ast}}(U_0)$
and $\Haus^{n-1}(S(U_0)) \in (0,\infty)$.
This implies (3) $\Rightarrow$ (2).
\end{proof}

Thus we arrive at Theorem \ref{thm: dwrt}.

\begin{proof}[Proof of Theorem \emph{\ref{thm: dwrt}}]
Lemmas \ref{lem: rnwcd} and \ref{lem: rnwngb}
lead to the theorem.
\end{proof}

\subsection{Regularity of codimension two}

In order to prove Theorem \ref{thm: wrt},
we first show the following:

\begin{lem}
\label{lem: nbngb}
Let $X$ be a $\GCBA$ space,
and $U$ a purely $n$-dimensional open subset of $X$.
Assume that the $n$-wall set $W_n(U)$ in $U$ is empty.
Then there exists no open subset $U_0$ contained in $U$
for which $U_0$ is homeomorphic to $\R^{n-1} \ast T_{l_0}$
for some $l_0 \in \N$ with $l_0 \ge 3$.
\end{lem}

\begin{proof}
Suppose that 
an open subset $U_0$ contained in $U$ 
is homeomorphic to $\R^{n-1} \ast T_{l_0}$
for $l_0 \in \N$ with $l_0 \ge 3$.
Then 
$\dim S(U_0) = n-1$,
and each point in $S(U_0)$ is not isolated in $S(U_0)$.
Moreover,
for every $x \in S(U_0)$,
and for every sufficiently small open neighborhood $V$ of $x$,
we have
\[
\dim \left( S(U) \cap V \right) = n-1.
\]
By Proposition \ref{prop: gregsing},
for some non-empty open sense subset $R(S(U_0))$ of $S(U_0)$,
there exists a function 
$\nu_{R(S(U_0))} \colon R(S(U_0)) \to \{ 1, \dots, n-1 \}$,
such that if for a point $x$ in $R(S(U_0))$ we set
$m := \nu_{R(S(U_0))}(x)$,
then we find an $(m,\delta_m)$-strainer map
$\varphi_x \colon V_x \to \R^m$ 
for some open neighborhood $V_x$ of $x$
whose restriction 
$\varphi |_{R(S(U_0))}$ to $R(S(U_0))$
is a $(1+2\delta_{m+1})$-bi-Lipschitz embedding into $\R^m$.
In this case, 
we see the following:

\begin{slem}
\label{slem: snbngb}
For every $x \in R(S(U_0))$ we have
\[
\nu_{R(S(U_0))}(x) = n-1.
\]
In particular,
the set $R(S(U_0))$ is contained in
$R_{n-1,\delta_{n-1}}(U_0)$.
\end{slem}

\begin{proof}
Suppose that we find $x_0 \in R(S(U_0))$
with
$\nu_{R(S(U_0))}(x_0) \le n-2$.
Set $m_0 := \nu_{R(S(U_0))}(x_0)$.
In this case,
we find an $(m_0,\delta_{m_0})$-strainer map
$\varphi_{x_0} \colon V_{x_0} \to \R^{m_0}$ for 
some open neighborhood $V_{x_0}$ of $x_0$
whose restriction 
$\varphi_{x_0} |_{R(S(U_0))}$ to $R(S(U_0))$
is a $(1+2\delta_{m_0+1})$-bi-Lipschitz embedding into $\R^{m_0}$.
Then there exists an open neighborhood $V_1$ of $x_0$
contained in $V_0$ such that
$R(S(U_0)) \cap V_1$ coincides with $S(U_0) \cap V_1$.
Hence we obtain
\[
\dim \left( S(U_0) \cap V_1 \right)
\le \dim_{\mathrm{H}} \left( R(S(U_0)) \cap V_1 \right) \le m_0,
\]
and hence we have 
$\dim \left( S(U_0) \cap V_1 \right) \le n-2$.
On the other hand,
in our setting,
we have
$\dim \left( S(U_0) \cap V_1 \right) = n-1$.
This is a contradiction.
\end{proof}

From Sublemma \ref{slem: snbngb}
it follows that
the set $R(S(U_0))$ is contained in
$W_{n,\delta_n}(U_0)$;
In particular,
$W_{n,\delta_{n-1},\delta_n}(U_0)$ is non-empty.
By Lemma \ref{lem: rnwngb},
the $n$-wall set $W_n(U_0)$ in $U_0$ is non-empty.
Hence the set $W_n(U)$ is non-empty.
This finishes the proof of Lemma \ref{lem: nbngb}.
\end{proof}

Now we prove Theorem \ref{thm: wrt}.

\begin{proof}
[Proof of Theorem \ref{thm: wrt}]
Let $X$ be a $\GCBA$ space,
and $U$ a purely $n$-dimensional open subset of $X$.
In Lemma \ref{lem: nbngb},
we verify the implication (1) $\Rightarrow$ (2).
The opposite one (2) $\Rightarrow$ (1)
can be derive from Theorem \ref{thm: dwrt}.
In Lemma \ref{lem: nwcd},
we show that (1), (3), and (4) are mutually equivalent.
The implication (6) $\Rightarrow$ (5) follows from 
$\dim \le \dim_{\mathrm{H}}$,
and (5) $\Rightarrow$ (1) follows from Lemma \ref{lem: cnwcd}.
The rest is to prove (1) $\Rightarrow$ (6).
Assume that $W_n(U)$ is empty.
Let $\delta^{\ast} \in (0,1)$ be small enough.
Theorem \ref{thm: dwrt} implies
$\dim_{\mathrm{H}} S_{n,\delta^{\ast}}(U) \le n-2$.
Since $S(U)$ is contained in $S_{n,\delta^{\ast}}(U)$,
\[
\dim S(U)
\le \dim_{\mathrm{H}} S(U)
\le \dim_{\mathrm{H}} S_{n,\delta^{\ast}}(U) \le n-2.
\]
Therefore we see (1) $\Rightarrow$ (6).
This finishes the proof of Theorem \ref{thm: wrt}.
\end{proof}

\subsection{Convergences of regularity of codimension two}

We first review the following 
(\cite[Lemma 3.3]{lytchak-nagano2}):

\begin{lem}
\label{lem: limhm}
\emph{(\cite{lytchak-nagano2})}
Let $r \in (0,D_{\kappa}/2)$.
Take a sequence $(X_k,p_k)$ of pointed 
proper geodesic $\CAT(\kappa)$ spaces
converging to some pointed metric space $(X,p)$
in the pointed Gromov--Hausdorff topology.
If all the open metric ball $U_r(p_k)$ in $X_k$
are homology $n$-manifolds,
then the open metric ball $U_r(p)$ in $X$ is a homology $n$-manifold.
\end{lem}

We next show the following:

\begin{lem}
\label{lem: limnw}
Let $r \in (0,D_{\kappa}/2)$.
Take a sequence $(X_k,p_k)$ of pointed purely $n$-dimensional
proper geodesic $\GCBA$ $\CAT(\kappa)$ spaces
converging to some pointed metric space $(X,p)$
in the pointed Gromov--Hausdorff topology.
If all the $n$-wall sets $W_n(U_r(p_k))$ in $U_r(p_k)$
are empty,
then the $n$-wall set $W_n(U_r(p))$ in $U_r(p)$ is empty.
\end{lem}

\begin{proof}
By Lemma \ref{lem: limpd},
the limit space $X$ is purely $n$-dimensional.

Suppose that we find an $n$-wall point $x$ in $W_n(U_r(p))$.
Take a sequence $(x_k)$ with $x_k \in X_k$
converging to the point $x$ in $X$.
For $n \in \N$,
let $\delta^{\ast} \in (0,1)$ be sufficiently small,
and let $\delta \in (0,\delta^{\ast})$ be small enough for 
$\delta^{\ast}$.
To get a contradiction,
we divide into the following two cases:

\emph{Case 1}.
Assume that
for some $s \in (0,D_{\kappa}/2)$ 
the open metric balls $U_s(x_k)$
are contained in $R_{n,\delta^{\ast}}(U_r(p_k))$
for all sufficiently large $k$.
Then all the balls $U_s(x_k)$ are topological $n$-manifolds.
By Lemma \ref{lem: limhm},
the ball $U_s(x)$ is a homology $n$-manifold.
From Proposition \ref{prop: hmnw}
it follows that $W_n(U_s(x))$ is empty.
This is a contradiction.

\emph{Case 2}.
Assume that Case 1 does not occur.
In this case,
we may assume that for some sequence $(s_k)$ with $s_k \to 0$
we find some points $y_k$ in 
$S_{n,\delta^{\ast}}(U_r(p_k)) \cap U_{s_k}(x_k)$
for all $k$.
Since the sequence $(y_k)$
converges to the point $x$ in $X$,
the stability of the strainers 
in Lemmas \ref{lem: usstab} implies that
$y_k \in R_{n-1,\delta}(U_r(p_k))$ for all sufficiently large $k$;
in particular,
we have
$y_k \in W_{n,\delta,\delta^{\ast}}(U_r(p_k))$.
From Theorem \ref{thm: dwrt}
we deduce that 
$W_n(U_r(p_i))$ is non-empty.
This is also a contradiction.

In Cases 1 and 2,
we obtain a contradiction.
Therefore we conclude the lemma.
\end{proof}

From Theorem \ref{thm: wrt} and Lemma \ref{lem: limnw}
we derive the following:

\begin{prop}
\label{prop: limn-2}
Let $(X_k,p_k)$
be a sequence of pointed 
purely $n$-dimensional proper geodesic $\CAT(\kappa)$ $\GCBA$ spaces
converging to some metric space $(X,p)$
in the pointed Gromov--Hausdorff topology.
Then for every $r \in (0,D_{\kappa}/2)$ the following hold:
\begin{enumerate}
\item
if for all $k$
we have $\dim S_n(U_r(p_k)) \le n-2$,
then we have $\dim S_n(U_r(p)) \le n-2$;
\item
if for all $k$
we have $\dim_{\mathrm{H}} S_n(U_r(p_i)) \le n-2$,
then we have $\dim_{\mathrm{H}} S_n(U_r(p)) \le n-2$;
\item
if for all $k$
we have $\dim S(U_r(p_i)) \le n-2$,
then we have $\dim S(U_r(p)) \le n-2$;
\item
if for all $k$
we have $\dim_{\mathrm{H}} S(U_r(p_i)) \le n-2$,
then we have $\dim_{\mathrm{H}} S(U_r(p)) \le n-2$.
\end{enumerate}
\end{prop}

\subsection{Infinitesimal regularity of codimension two}

For the proof of Theorem \ref{thm: iwrt},
we verify the following:

\begin{lem}
\label{lem: conenw}
Let $Z$ be a purely $(n-1)$-dimensional $\GCBA$ $\CAT(1)$ space
that is not a singleton.
Let $C_0(Z)$ be the Euclidean cone over $Z$.
Then the $n$-wall set $W_n(C_0(Z))$ in $C_0(Z)$
is empty if and only if
the $(n-1)$-wall set $W_{n-1}(Z)$ in $Z$ is empty.
\end{lem}

\begin{proof}
The Euclidean cone $C_0(Z)$ over $Z$
is a purely $n$-dimensional $\GCBA$ $\CAT(0)$ space. 

We assume that $W_n(C_0(Z))$ is empty.
Suppose that $W_{n-1}(Z)$ is non-empty.
Then we find a point $z_0$ in $Z$ at which
$T_{z_0}Z$ isometrically splits as $\R^{n-2} \times T_{l_0}^1$
for some $l_0 \in \N$ with $l_0 \ge 3$.
Hence for the point $[(z_0,1)] \in C_0(Z)$
the tangent space $T_{[(z_0,1)]}C_0(Z)$ 
isometrically splits as $\R^{n-1} \times T_{l_0}^1$.
This implies $[(z_0,1)] \in W_n(C_0(Z))$,
which is a contradiction.
Thus $W_{n-1}(Z)$ is empty.

Conversely,
we next assume that $W_{n-1}(Z)$ is empty.
Suppose that $W_n(C_0(Z))$ is non-empty.
Then we find $w_0$ in $C_0(Z)$
at which 
$T_{w_0}C_0(Z)$ isometrically splits as $\R^{n-1} \times T_{l_0}^1$
for some $l_0 \in \N$ with $l_0 \ge 3$.
In the case where $w_0$ is the vertex $0$ of $C_0(Z)$,
since $C_0(Z)$ is isometric to $T_0C_0(Z)$,
the Euclidean cone $C_0(Z)$ isometrically splits as 
$\R^{n-1} \ast T_{l_0}^1$.
Hence $Z$ isometrically splits as 
$\Sph^{n-1} \ast T_{l_0}^0$.
This implies that $W_{n-1}(Z)$ is non-empty,
which is a contradiction.
In the case where $w_0$ is distinct to the vertex $0$ of $C_0(Z)$,
we find $z_0 \in Z$ with $w_0 = [(z_0,1)]$
so that $T_{z_0}Z$ isometrically splits as $\R^{n-1} \times T_{l_0}^1$.
This implies $z_0 \in W_{n-1}(Z)$,
which is a contradiction.
Thus $W_n(C_0(Z))$ is empty.
\end{proof}

Next we show the following:

\begin{lem}\label{lem: iwrtnw}
Let $X$ be a $\GCBA$ space,
and let $U$ be a purely $n$-dimensional open subset of $X$.
Then the following are equivalent:
\begin{enumerate}
\item
the $n$-wall set $W_n(U)$ in $U$ is empty;
\item
for every $x \in U$
the $n$-wall set $W_n(T_xX)$ in $T_xX$ is empty;
\item
for every $x \in U$
the $(n-1)$-wall set $W_{n-1}(\Sigma_xX)$ in $\Sigma_xX$ is empty.
\end{enumerate}
\end{lem}

\begin{proof}
By Lemma \ref{lem: conenw},
it suffices prove the equivalence (1) $\Leftrightarrow$ (2).

First we prove (1) $\Rightarrow$ (2).
Assume that $W_n(U)$ is empty.
Let $x \in U$ be arbitrary.
Considering the Gromov--Hausdorff blow-up limit
\[
\lim_{\lambda \to \infty} (\lambda X, x) \to (T_xX,0_x),
\]
we find a sufficiently large $\lambda_0 \in (0,\infty)$
such that for every $\lambda \in (\lambda_0,\infty)$
the open metric ball $U_1^{\lambda X}(x)$
in $\lambda X$ of radius $1$ around $x$
is contained in $U$.
Hence the $n$-wall set $W_n(U_1^{\lambda X}(x))$ 
in $U_1^{\lambda X}(x)$ is empty.
From Lemma \ref{lem: limnw} it follows that
the $n$-wall set $W_n(U_1(0_x))$ in 
the open ball $U_1(0_x)$ in $T_xX$ is empty. 
This implies that the $n$-wall set $W_n(T_xX)$
in the whole space $T_xX$ is empty.
Thus we see (1) $\Rightarrow$ (2).

Next we show (2) $\Rightarrow$ (1).
Assume that for any $x \in U$ 
the $n$-wall set $W_n(T_xX)$ in $T_xX$ is empty.
Suppose that $W_n(U)$ is non-empty.
Then we find a point $x_0$ in $U$
at which $T_{x_0}X$ isometrically splits as $\R^{n-1} \times T_{l_0}^1$
for some $l_0 \in \N$ with $l_0 \ge 3$.
Since $T_{0_{x_0}}T_{x_0}X$ is isometric to $T_{x_0}X$,
the vertex $0_{x_0}$ is an $n$-wall point in $T_{x_0}X$.
This is a contradiction.
Hence $W_n(U)$ is empty.
Thus we see (2) $\Rightarrow$ (1).
\end{proof}

\begin{proof}[Proof of Theorem \emph{\ref{thm: iwrt}}]
Combining Theorem \ref{thm: wrt} and Lemma \ref{lem: iwrtnw}
leads to the theorem.
\end{proof}

\section{Applications}

\subsection{Low-dimensional GCBA spaces without wall singularity}

We first quote the local topological regularity theorem 
established by Lytchak and the author \cite[Theorem 1.1]{lytchak-nagano2}
in the following form:

\begin{thm}
\label{thm: loctopreg}
\emph{(\cite{lytchak-nagano2})}
Let $U$ be an open subset of a connected locally compact 
$\CBA(\kappa)$ space $X$.
Then the following are equivalent:
\begin{enumerate}
\item
$U$ is a topological $n$-manifold;
\item
for every $x \in U$
the space $\Sigma_xX$ is homotopy equivalent to $\Sph^{n-1}$;
\item
for every $x \in U$
the space $T_xX$ is homeomorphic to $\R^n$.
\end{enumerate}
\end{thm}

From Theorem \ref{thm: loctopreg} and
our infinitesimal regularity theorem \ref{thm: iwrt}
of codimension two,
we derive the following:

\begin{prop}
\label{prop: ldnowall}
Let $X$ be a $\GCBA$ space,
and let $U$ be a purely $n$-dimensional open subset of $X$.
Assume $\dim S(U) \le n-2$.
\begin{enumerate}
\item
If $n = 2$,
then $U$ is a topological $2$-manifold.
\item
If $n = 3$,
then for every $x \in U$ the space of directions $\Sigma_xX$ at $x$
is a closed topological $2$-manifold.
\end{enumerate}
\end{prop}

\begin{proof}
(1)
Assume $n = 2$.
Then for every $x \in U$ 
by Theorem \ref{thm: iwrt}
the space $\Sigma_xX$ is a closed topological $1$-manifold,
and so it is homeomorphic to $\Sph^1$.
Due to the local topological regularity theorem
\ref{thm: loctopreg},
we see that $U$ is a topological $2$-manifold.

(2)
Assume $n = 3$.
Then for every $x \in U$ 
by Theorem \ref{thm: iwrt}
the space $\Sigma_xX$ is a purely $2$-dimensional $\GCBA$ space
with $\dim S(\Sigma_xX) \le 0$.
Hence (2) implies that $\Sigma_xX$ is a closed topological $2$-manifold.
\end{proof}

\subsection{Volume pinching theorems for CAT(1) spaces}

We say that a triple of points in a $\CAT(1)$ space is a
\emph{tripod}
if the three points have pairwise distance $\pi$.
Lytchak and the author proved a 
\emph{capacity sphere theorem}
\cite[Theorem 1.5]{lytchak-nagano2}
for $\CAT(1)$ spaces
stating that
if a compact, geodesically complete $\CAT(1)$ space
admits no tripod,
then it is homeomorphic to a sphere.
As its application,
Lytchak and the author also proved a 
\emph{volume sphere theorem}
\cite[Theorem 8.3]{lytchak-nagano2}
for $\CAT(1)$ spaces
stating that
if a purely $m$-dimensional, compact, geodesically complete
$\CAT(1)$ space $Z$ satisfies
$\Haus^m(Z) < (3/2) \Haus^m(\Sph^m)$.
The author provided the following characterization
(\cite[Theorem 1.1]{nagano4}):

\begin{thm}
\label{thm: just3/2}
\emph{(\cite{nagano4})}
If a purely $m$-dimensional, compact, 
geodesically complete $\CAT(1)$ space $Z$ satisfies
\[
\Haus^m(Z) = \frac{3}{2} \Haus^m(\Sph^m),
\]
then $Z$ is either homeomorphic to $\Sph^m$
or isometric to $\Sph^{m-1} \ast T_3^0$.
If in addition $Z$ has a tripod,
then $Z$ is isometric to the $m$-triplex or $\Sph^{m-1} \ast T_3^0$,
where the $m$-triplex is a concrete example
of $\CAT(1)$ spaces defined in \cite[Example 1.1]{nagano4}
that is homeomorphic to $\Sph^m$.
\end{thm}

A function $\rho \colon [0,r) \to [0,\infty)$ with $\rho(0) = 0$ is 
said to be a \emph{contractivity function} 
if $\rho$ is continuous at $0$,
and if $\rho(t) \ge t$ for all $t \in [0,r)$.
For a contractivity function $\rho \colon [0,r) \to [0,\infty)$,
a metric space $X$ is said to be $\LGC(\rho)$
if for every $x \in X$ and for every $s \in (0,r)$
the ball $B_s(x)$ is contractible inside 
$B_{\rho(s)}(x)$.
Every $\CAT(\kappa)$ space is $\LGC(\id_{[0,D_{\kappa})})$.

We review the Petersen homotopic stability theorem
\cite[Theorem A]{petersen}
for $\LGC$ spaces.
Let $\rho$ be a contractivity function.
Then for every $m \in \N$ there exists $\delta \in (0,\infty)$
depending only on $m$ and $\rho$ such that
if two compact $\LGC(\rho)$ spaces of dimension at most $m$
have Gromov--Hausdorff distance less than $\delta$,
then they are homotopy equivalent.

By the same idea as \cite[Proposition 5.6]{nagano4},
we show the following:

\begin{lem}
\label{lem: wnwvpcat}
For every $m \in \N$,
there exists $\delta \in (0,\infty)$
such that if a purely $m$-dimensional, compact,
geodesically complete $\CAT(1)$
space $Z$ with $\dim S(Z) \le m-2$ satisfies 
\eqref{eqn: nwvptcata},
then $Z$ is homotopy equivalent to $\Sph^m$.
\end{lem}

\begin{proof}
Suppose that the lemma does not hold true.
By virtue of the volume sphere theorem 
\cite[Theorem 8.3]{lytchak-nagano2},
we may suppose that
there exists a sequence of purely $m$-dimensional
compact geodesically complete $\CAT(1)$ spaces $Z_i$, $i = 1, 2, \dots$,
with $\dim S(Z_i) \le m-2$
satisfying 
$\lim_{i \to \infty} \Haus^m(Z_i) = (3/2) \Haus^m(\Sph^m)$
such that
each $Z_i$ is not homotopy equivalent to $\Sph^m$.
Due to the capacity sphere theorem 
\cite[Theorem 1.5]{lytchak-nagano2},
we may assume that each $Z_i$ has a tripod.
The sequence $Z_i$, $i = 1, 2, \dots$,
has a convergent subsequence $Z_j$, $j = 1, 2, \dots$,
tending to some compact metric space $Z$
in the Gromov--Hausdorff topology
(see \cite[Proposition 3.4]{nagano4}, 
and \cite[Proposition 6.5]{nagano1}).
By Lemma \ref{lem: limpd},
the limit $Z$ is a purely $m$-dimensional, compact,
geodesically complete $\CAT(1)$ space.
Since each $Z_j$ has a tripod,
the limit $Z$ has a tripod too.
From the volume convergence theorem
\cite[Theorem 3.3]{nagano4}
(\cite[Theorem 1.1]{nagano2})
we derive
$\Haus^m(Z) = (3/2) \Haus^m(\Sph^m)$.

Theorem \ref{thm: just3/2} then implies that
the limit $Z$ is either isometric to the $m$-triplex or 
$\Sph^{m-1} \ast T_3^0$.
Proposition \ref{prop: limn-2} implies $\dim S(Z) \le m-2$,
and hence $Z$ must be isometric to the $m$-triplex.
The $m$-triplex $Z$ is homeomorphic to $\Sph^m$.
The Petersen homotopic stability theorem \cite[Theorem A]{petersen}
for $\LGC$ spaces tells us that
$Z_j$ is homotopy equivalent to $\Sph^m$ 
for all sufficiently large $j$.
This is a contradiction.
\end{proof}

We provide the following criterion 
(compare with \cite[Theorem 1.3]{nagano4}).

\begin{prop}
\label{prop: nwmfdrvolg}
For every $n \in \N$ there exists $\delta \in (0,\infty)$
with the following property:
Let $X$ be a purely $n$-dimensional, proper,
geodesically complete 
$\CAT(\kappa)$ space with $\dim S(X) \le n-2$,
and let $U$ be an open subset of $X$.
If for every $x \in U$ there exists $r \in (0,D_{\kappa})$ satisfying
\begin{equation}
\frac{\Haus^n(B_r(x))}{\omega_{\kappa}^n(r)} < \frac{3}{2} + \delta,
\label{eqn: nwmfdrvolga}
\end{equation}
then $U$ is a topological $n$-manifold.
\end{prop}

\begin{proof}
For $n \in \N$,
let $\delta \in (0,\infty)$ be sufficiently small.
Let $X$ be a purely $n$-dimensional, proper, geodesically complete 
$\CAT(\kappa)$ space
with $\dim S(X) \le n-2$,
and let $U$ be an open subset of $X$.
Assume that for every $x \in U$ 
there exists $r \in (0,D_{\kappa})$ satisfying \eqref{eqn: nwmfdrvolga}.
Then 
\[
\frac{\Haus^{n-1}(\Sigma_xX)}{\Haus^{n-1}(\Sph^{n-1})} \le
\frac{\Haus^n(B_r(x))}{\omega_{\kappa}^n(r)} < \frac{3}{2} + \delta
\]
(see \cite[Lemma 3.6]{nagano4}).
From Theorem \ref{thm: iwrt} and Proposition \ref{prop: pure},
it follows that
$\Sigma_xX$ is a purely $(n-1)$-dimensional, compact,
geodesically complete $\CAT(1)$ space 
with $\dim S(\Sigma_xX) \le n-3$.
Lemma \ref{lem: wnwvpcat} implies that
$\Sigma_xX$ is homotopy equivalent to $\Sph^{n-1}$.
From the local topological regularity theorem \ref{thm: loctopreg},
we see that $U$ is a topological $n$-manifold.
\end{proof}

\begin{proof}[Proof of Theorem \ref{thm: nwvptcat}]
For $m \in \N$,
let $\delta \in (0,\infty)$ be small enough.
Let $Z$ be a purely $m$-dimensional, compact, geodesically complete 
$\CAT(1)$ space
with $\dim S(Z) \le m-2$ satisfying
\eqref{eqn: nwvptcata}.
The assumption \eqref{eqn: nwvptcata} 
together with the relative volume comparison
\cite[Proposition 3.2]{nagano4}
(\cite[Proposition 6.3]{nagano2}) tells us that
for any $z \in Z$ and $r \in (0,\pi)$ we have
\[
\frac{\Haus^m(B_r(z))}{\omega_1^m(r)} \le
\frac{\Haus^m(Z)}{\Haus^m(\Sph^m)} < 
\frac{3}{2} + \frac{\delta}{\Haus^m(\Sph^m)}.
\]
From Proposition \ref{prop: nwmfdrvolg}
it follows that
$Z$ is a topological $m$-manifold,
provided $\delta$ is small enough.
The volume sphere theorem for $\CAT(1)$ homology manifold
\cite[Theorem 1.2]{nagano2} leads to Theorem \ref{thm: nwvptcat}.
\end{proof}

\subsection{Asymptotic topological regularity of CAT(0) spaces}

Let $X$ be a purely $n$-dimensional, proper, 
geodesically complete $\CAT(0)$ space.
Then the following are equivalent
(\cite[Proposition 3.8]{nagano5}):
(1) $X$ is doubling;
(2) $X$ has the Gromov--Hausdorff asymptotic cone;
(3) for some $c \in [1,\infty)$ we have $\mathcal{G}_0^n(X) \le c$;
in this case,
for every $x \in X$ 
\begin{equation}
\frac{\Haus^{n-1}(\Sigma_xX)}{\Haus^{n-1}(\Sph^{n-1})} \le
\frac{\Haus^{n-1}(\partial_{\mathrm{T}}X)}{\Haus^{n-1}(\Sph^{n-1})} =
\mathcal{G}_0^n(X),
\label{eqn: nwatrcatb}
\end{equation}
where $\partial_{\mathrm{T}}X$ is the Tits boundary $X$.

\begin{proof}[Proof of Theorem \ref{thm: nwatrcat}]
For $n \in \N$,
let $\delta \in (0,\infty)$ be small enough.
Let $X$ be a purely $n$-dimensional, compact, geodesically complete 
$\CAT(0)$ space
with $\dim S(X) \le n-2$
satisfying
\eqref{eqn: nwatrcata}.
Then by \eqref{eqn: nwatrcatb}
for every $x \in X$ we have
\[
\frac{\Haus^{n-1}(\Sigma_xX)}{\Haus^{n-1}(\Sph^{n-1})} 
< \frac{3}{2} + \delta.
\]
From Theorem \ref{thm: iwrt} and Proposition \ref{prop: pure}
it follows that
$\Sigma_xX$ is a purely $(n-1)$-dimensional, compact, 
geodeically complete $\CAT(1)$ space 
with $\dim S(\Sigma_xX) \le n-3$.
From Proposition \ref{prop: nwmfdrvolg}
we conclude that $X$ is a topological $n$-manifold.
The asymptotic topological regularity theorem
\cite[Theorem 1.4]{nagano5} leads to 
Theorem \ref{thm: nwatrcat}.
\end{proof}


\end{document}